\documentclass[12pt]{article}
\usepackage{amsmath}
\usepackage{latexsym}
\usepackage{amssymb}
\usepackage{ntheorem}
\usepackage{xcolor}
\usepackage{url}
\usepackage{mathrsfs}
\usepackage{pdflscape}
\usepackage{tikz}
\usepackage{rotating}

\usepackage{multicol}
\usepackage{enumitem}

\newtheorem{theorem}{Theorem}[subsection]
\newtheorem{thmintro}{Theorem}

\newtheorem{lemma}[theorem]{Lemma}
\newtheorem{definition}[theorem]{Definition}
\newtheorem{Remark}[theorem]{Remark}

\newtheorem{proposition}[theorem]{Proposition}
\newtheorem{corollary}[theorem]{Corollary}
\newtheorem{Example}[theorem]{Example}

\newtheorem{Number}[theorem]{\!\!}
\newtheorem{Claim}{Claim}
\newtheorem*{Claim-non}{Claim}

\newenvironment{example}{\begin{Example}\rm}{\end{Example}}
\newenvironment{remark}{\begin{Remark}\rm}{\end{Remark}}

\newenvironment{proof}{{\noindent\bf Proof.}}%
                  {\nopagebreak\hspace*{\fill}$\Box$\medskip\par}

\newcommand{\fl}[1]{\mathscr{#1}}

\DeclareMathOperator{\Aut}{Aut}
\DeclareMathOperator{\Isom}{Isom}
\DeclareMathOperator{\End}{End}

\newcommand{\tree}{T}
\DeclareMathOperator{\id}{id}

\DeclareMathOperator{\stab}{{stab}}
\DeclareMathOperator{\supp}{{supp}}

\DeclareMathOperator{\con}{{\sf con}}
\DeclareMathOperator{\parb}{{\sf par}}
\DeclareMathOperator{\lev}{{\sf lev}}
\DeclareMathOperator{\nub}{{\sf nub}}
\DeclareMathOperator{\lnub}{{\sf lnub}}
\DeclareMathOperator{\COS}{\mathsf{CO}}
\DeclareMathOperator{\red}{\Phi}
\DeclareMathOperator{\rk}{{rk}}
\DeclareMathOperator{\modulo}{mod}

\newcommand{\open}[1]{\mathcal{#1}}
\newcommand{\Up}[2]{#1_{#2+}}
\newcommand{\Um}[2]{#1_{#2-}}
\newcommand{\Upm}[2]{#1_{#2\pm}}
\newcommand{\Uz}[2]{#1_{#20}}

\newcommand{\Upp}[2]{\widehat{#1}_{#2+}}
\newcommand{\Umm}[2]{\widehat{#1}_{#2-}}
\newcommand{\Utpm}[2]{\widehat{#1}_{#2\pm}}
\def\tree{T}

\def\fp{\mathfrak{p}}
\def\MUab{(U,\alpha,\beta,-)}
\def\PUab{(U,\alpha,\beta,+)}
\def\PMUab{(U,\alpha,\beta,\pm)}

\def\triv{\{\id\}}
\def\tdlc{t.d.l.c.}
\def\roo{\rho}
\definecolor{cardinal}{rgb}{0.77, 0.12, 0.23}
\definecolor{carnelian}{rgb}{0.7, 0.11, 0.11}
\definecolor{cornellred}{rgb}{0.7, 0.11, 0.11}
\definecolor{crimsonglory}{rgb}{0.75, 0.0, 0.2}
\definecolor{darkcandyapplered}{rgb}{0.64, 0.0, 0.0}
\definecolor{deepcarminepink}{rgb}{0.94, 0.19, 0.22}
\definecolor{electriccrimson}{rgb}{1.0, 0.0, 0.25}
\definecolor{fireenginered}{rgb}{0.81, 0.09, 0.13}
\definecolor{green(colorwheel)(x11green)}{rgb}{0.0, 1.0, 0.0}
\definecolor{green(html/cssgreen)}{rgb}{0.0, 0.5, 0.0}
\definecolor{bondiblue}{rgb}{0.0, 0.58, 0.71}
\title{Flat groups of automorphisms of totally disconnected, locally compact groups}
\author{George A. Willis\thanks{This research was supported by the Australian Research Council grant FL170100032}}
\begin{document}
\maketitle
\begin{abstract}
A group, $\fl{H}$, of automorphisms of a totally disconnected locally compact group, $G$, is flat if there is a compact open $U\leq G$ such that the index $[\alpha(U):U\cap \alpha(U)]$ is mininimized for every $\alpha\in\fl{H}$. The stabilizer of $U$ in $\fl{H}$ is a normal subgroup, $\fl{H}_u$; the quotient $\fl{H}/\fl{H}_u$ is a free abelian group; and the rank of $\fl{H}$ is the rank of this free abelian group. Each singly generated group $\langle\alpha\rangle$ is flat and has rank either $0$ or $1$. Higher rank groups may be seen in Lie groups over local fields and automorphism groups of buildings. 
	
Flat groups of automorphisms exhibit many of the features of these special examples, including  analogues of roots and a factoring of $U$ into analogues of root subgroups. New proofs of improved versions of these results are presented here. 
\end{abstract}
{\bf Classification:} Primary 22D05; 
Secondary
%
%
%
\tableofcontents
\section*{Introduction}
\label{sec:Intro}

These notes describe the structure of flat groups of automorphisms of a totally disconnected, locally compact (\tdlc)~group. Cartan subgroups of semi-simple $p$-adic Lie groups, or more precisely the groups of inner automorphisms they induce, are motivating examples flat groups and it is seen that flat groups share some of the features of Cartan subgroups.

A group, $\fl{H}$, of automorphisms of $G$ is flat if there is a compact open subgroup, $U$, of $G$ that is minimizing for every $\alpha\in\fl{H}$. For this notion to have any content, it is necessary that $G$ be non-discrete, for the trivial subgroup is compact, open and minimizing for every automorphism otherwise. For it to be meaningful for a  \tdlc~group, it is necessary that the group have compact open subgroups , and that was shown in general by van Dantzig,~\cite{vD_CO}, see~\cite[Theorem~7.7]{HandR}. Since $U$ is minimizing for $\alpha$ if and only if it is tidy for $\alpha$, see \cite{Structure94,FurtherTidy}, results about subgroups tidy for a single automorphism are used to study flat groups. Relevant results about tidy subgroups are collected from the literature in the first section.

The second section gives complete proofs that a flat group $\fl{H}$ is abelian modulo the stabilizer of a minimizing subgroup and that there exist surjective homomorphisms  $\fl{H}\to\mathbb{Z}$ that are analogous to roots of a Cartan subgroup. It is also shown that tidy subgroups factor into analogues of root subgroups. These things were shown in \cite{SimulTriang}, but stronger statements and new and clearer proofs are presented here.

\section{Single automorphisms and endomorphisms}
\label{sec:single_automorphism}
To each automorphism, $\alpha$, of a \tdlc~group, $G$, is associated a positive integer $s(\alpha)$, called its scale, and certain compact open subgroups of $G$ that are said to be tidy for $\alpha$. Tidy subgroups and their properties reveal the dynamics of the action of $\alpha$ on $G$. This section recalls the definitions and theorems about the scale and tidy subgroups needed to study flat groups of automorphisms. 

The following notation will be used throughout. The group of automorphisms of the \tdlc~group $G$ will be denoted by $\Aut(G)$ and the monoid of endomorphisms by $\End(G)$. Individual automorphisms and endomorphisms will be denoted $\alpha$, $\beta$,  \dots. The set of all compact open subgroups of $G$ will be denoted by $\COS(G)$ and individual subgroups by letters $U$, $V$,  \dots. 

\subsection{Minimizing subgroups and the scale}
\label{sec:Min-Scale}

This section gives the definitions and summarizes basic results about the scale and minimizing subgroups associated with a single endomorphism of a \tdlc~group. We begin with endomorphisms even though the later results on flatness only concern automorphisms. Little additional effort is needed to formulate the results (although not their proofs) in the more general context, and it would be desirable to extend the idea of flatness to endomorphisms. Results on endomorphisms are quoted from \cite{Endo15} and extend ideas initially developed for (inner) automorphisms in \cite{Structure94,FurtherTidy}. 

Here is the basic definition. It relies on observation that, if $U\in\COS(G)$, then $\alpha(U)\cap U$ is an open subgroup of the compact group $\alpha(U)$ and therefore has finite index.
\begin{definition}
\label{defn:minimizing}
Suppose that $G$ is a \tdlc~group and that $\alpha\in\End(G)$. The \emph{scale of $\alpha$} is the positive integer
$$
s(\alpha) = \min\left\{ [\alpha(U) : \alpha(U)\cap U] \mid U\in\COS(G)\right\}.
$$
The subgroup $U$ is \emph{minimizing for $\alpha$} if $s(\alpha)$ is attained at $U$.
\end{definition}

This definition is not very useful by itself. Its value comes from a characterization of subgroups minimizing for an endomorphism in terms of dynamics of the action of the endomorphism.
\begin{theorem}[\cite{FurtherTidy} Theorem~3.1 \& \cite{Endo15} Theorem~2]\ 
\label{thm:characterise_minimizing}

\noindent Let $G$ be a \tdlc~group and $\alpha\in\End(G)$. Suppose that $U\in\COS(G)$ and set
\begin{align*}
\Up{U}{\alpha} &= \left\{ u\in U \mid \exists\{u_n\}_{n\geq0}\subset U\mbox{ such that }\alpha(u_{n+1}) = u_n \forall n\geq0\mbox{ and } u_0=u\right\}\\
\Um{U}{\alpha} &= \left\{ u\in U \mid \alpha^n(u)\in U\ \forall n\geq0\right\}.
\end{align*}
Then $\Up{U}{\alpha}$ and $\Um{U}{\alpha}$ are closed subgroups of $U$, and $U$ is minimizing for $\alpha$ if and only if
\begin{description}
\item[TA:] $U = \Up{U}{\alpha}\Um{U}{\alpha}$ \mbox{ and }
\item[TB:] $\Umm{U}{\alpha} := \bigcup_{n\geq0} \alpha^{-n}(U_{\alpha-})$ is closed.\label{TB}
\end{description}
When $U$ satisfies these conditions the scale of $\alpha$ is equal to $[\alpha(\Up{U}{\alpha}) : \Up{U}{\alpha}]$.
\end{theorem}

\begin{definition}
\label{defn:tidy}
The compact open subgroup $U$ is said to be \emph{tidy above} for $\alpha$ if it satisfies {\bf TA} and to be \emph{tidy below} if it satisfies {\bf TB}. The subgroup is \emph{tidy for $\alpha$} if it is both tidy above and below for $\alpha$. 
\end{definition}

\begin{remark}
\label{rem:characterise_minimizing}
Tidy subgroups were defined for automorphisms in~\cite{Structure94} and tidiness was shown to be equivalent to the minimizing condition in~\cite{FurtherTidy}. These papers defined $\Up{U}{\alpha} = \bigcap_{n\geq0}\alpha^n(U)$ and $\Um{U}{\alpha} = \bigcap_{n\geq0}\alpha^{-n}(U)$, which is equivalent to the definition given here when $\alpha$ is an automorphism. The extension to endomorphisms was made in~\cite{Endo15}. 
\end{remark}



The following claims are established in Exercises~\ref{ex:comparable_tidy},~\ref{exer:scale_prop1} and~\ref{exer:scale_prop2}.
\begin{proposition}
	\label{prop:tidy_conditions}
Let $\alpha\in\Aut(G)$. 
\begin{enumerate}
	\item 	\label{prop:tidy_conditions1}
	If $U\in\COS(G)$ satisfies that either $\alpha(U)\leq U$ or $\alpha(U)\geq U$, then $U$ is tidy for $\alpha$.
\item 	\label{prop:tidy_conditions2}
If $U$ is tidy for $\alpha$, then $U$ is tidy for $\alpha^n$ for every $n\in\mathbb{Z}$, and for $\alpha^{-1}$ in particular.
\end{enumerate}
\endproof
\end{proposition}

Part~\ref{prop:tidy_conditions1}.~of Proposition~\ref{prop:tidy_conditions} does not hold for endomorphisms, see Exercise~\ref{ex:comparable_not_tidy}. Part~\ref{prop:tidy_conditions2}.~does hold for non-negative $n$ and the proof of tidiness below is seen in Exercise~\ref{exer:endo_tidy_below}. It is not true, even for automorphisms, that any $U$ tidy for $\alpha^n$ is tidy for $\alpha$, see Exercise~\ref{exer:power_tidy_not}.

The following criteria for membership of $\Um{U}{\alpha}$ and $\Uz{U}{\alpha} := \bigcap_{k\in\mathbb{N}} \alpha^k(\Um{U}{\alpha})$ when $U$ is tidy for $\alpha$ are used frequently. The converse to these criteria is the observation that, if $u\in \Um{U}{\alpha}$, then $\left\{\alpha^k(u)\right\}_{k\in\mathbb{N}}\subset \Um{U}{\alpha}$ and thus in particular is contained in a compact set. 
\begin{proposition}[\cite{Structure94} Lemma~9 \& \cite{Endo15} Proposition~11]
\label{prop:tidy_criteria}
Let $\alpha$ be an endomorphism of $G$. Suppose that $U$ is tidy for $\alpha$ and that $u\in U$. 
\begin{enumerate}
\item If either $\left(\alpha^k(u)\right)_{k\in\mathbb{N}}$ or $\left(u\alpha(u)\dots\alpha^k(u)\right)_{k\in\mathbb{N}}$ has an accumulation point, then $u$ belongs to $\Um{U}{\alpha}$. 
\item If either $\left\{\alpha^k(u)\right\}_{k\in\mathbb{Z}}$ or $\left\{u\alpha(u)\dots\alpha^k(u)\right\}_{k\in\mathbb{Z}}$ is contained in a compact set, then $u\in \Up{U}{\alpha}\cap\Um{U}{\alpha} =  \Uz{U}{\alpha}$. 
\end{enumerate}
\end{proposition}

Here is a description of the $\alpha$-orbit $\{\alpha^k(x)\}_{k\in\mathbb{Z}}$ of $x$ in the case when $\alpha$ is an automorphism. Endomorphisms need not be injective and the notion of \emph{$\alpha$-trajectory} is used instead in \cite{Endo15}, where the relationship between trajectories for $\alpha$ and subgroups tidy for $\alpha$ is described in detail.
\begin{proposition}
\label{prop:tidy_orbit}
Suppose that $\alpha\in\Aut(G)$ and that $U$ is tidy for $\alpha$. Then, for all $m\leq n\in\mathbb{Z}$, 
\begin{equation}
\label{eq:tidy_orbit}
\alpha^m(U)\cap\alpha^n(U)  = \bigcap_{k=m}^n \alpha^k(U).
\end{equation}
Hence, given $x\in G$, the intersection of the $\alpha$-orbit of $x$ with $U$ is either:
\begin{multicols}{2}
	\begin{enumerate}[label={(\alph*)}]
\item $\{\alpha^k(x)\}_{k=m}^\infty$ with $m\in\mathbb{Z}$;   
\label{tidy-orbit4}
\item  $\{\alpha^k(x)\}_{k=-\infty}^n$ with $m\in\mathbb{Z}$;
\label{tidy-orbit5}
\item $\{\alpha^k(x)\}_{k=m}^n$ with $m\leq n$;
\item $\{\alpha^k(x)\}_{k\in\mathbb{Z}}$; or
\label{tidy-orbit6}
\item $\emptyset$.
\item[] 
\end{enumerate}
\end{multicols}
In case \ref{tidy-orbit4}, $x\in \Umm{U}{\alpha}$, in \ref{tidy-orbit5} $x\in \Upp{U}{\alpha}$ and in case \ref{tidy-orbit6} $x\in\Uz{U}{\alpha}$.
 \end{proposition}
It is not true that Equation~\eqref{eq:tidy_orbit} is equivalent to tidiness $U$.  See Exercise \ref{exer:not_tidy}.

It will be important to know in the next section that the set of compact open subgroups tidy for $\alpha$ is closed under finite intersections. 
\begin{theorem}[\cite{Structure94} Lemma~10 \& \cite{Endo15} Proposition~12]
\label{thm:tidy_intersection}
Suppose that $\alpha\in\End(G)$ and that $U$ and $V$ are tidy for $\alpha$. Then $U\cap V$ is tidy for $\alpha$. 
\end{theorem}

The following may be verified directly from the definition of tidiness but it is easier to use the equivalence of tidiness with the minimizing property. 
\begin{lemma}
\label{lem:conjugation_invariance}
Let $\alpha\in\End(G)$ and $\beta\in\Aut(G)$. Then $U$ is tidy for $\alpha$ if and only if $\beta(U)$ is tidy for $\beta\alpha\beta^{-1}$. The scale is constant on $\Aut(G)$-conjugacy classes in $\End(G)$. 
\endproof
\end{lemma}
Flat groups are groups of automorphisms and these notes focus on automorphisms from now on. Some results to follow may be extended to endomorphisms, see~\cite{Endo15}, but are easier to state for automorphisms. Flatness has not, as yet, been extended to semigroups of endomorphisms. 

Since $[\alpha(U) : \alpha(U)\cap U]=1$ if and only if $\alpha(U)\leq U$, it holds that $s(\alpha)=1$ if and only if there is $U\in\COS(G)$ with $\alpha(U)\leq U$. The scale has the following further properties. 
\begin{theorem}[Properties of the scale]
\label{thm:scale_properties}
Let $G$ be a \tdlc\ group. The scale function $s:\Aut(G)\to \mathbb{N}$ satisfies, for all $\alpha\in\Aut(G)$:
\begin{enumerate}
\item \label{thm:scale_properties1}
$s(\alpha) = 1 = s(\alpha^{-1})$ if and only if there is $U\in\COS(G)$ with $\alpha(U)=U$; 
\item \label{thm:scale_properties2}
$s(\alpha^n) = s(\alpha)^n$, for all $n\in\mathbb{N}$; 
\item \label{thm:scale_properties3}
$\Delta(\alpha) = s(\alpha)/s(\alpha^{-1})$, with $\Delta : \Aut(G)\to \mathbb{R}^+$ being the module.
\end{enumerate}
The composite function $s\circ\alpha_{\bullet} : G\to \mathbb{N}$, where $\alpha_g : G\to G$ is the inner automorphism 
$$
\alpha_g(x) = gxg^{-1} \qquad (x\in G),
$$ 
is continuous for the \tdlc~topology on $G$ and the discrete topology on $\mathbb{N}$. 
\end{theorem}
For the proofs of these properties of the scale other than its continuity on~$G$, see Exercises \ref{exer:scale_prop1} and \ref{exer:scale_prop2}. Continuity of $s\circ\alpha_{\bullet} : G\to \mathbb{N}$ is equivalent to $s\circ\alpha_{\bullet}$ being locally constant, which is implied by the following structural theorem. In this theorem and subsequently, $s\circ\alpha_{\bullet}$ is abbreviated simply as the \emph{scale function $s : G\to \mathbb{N}$} and `$U$ is tidy for $g$' means `$U$ is tidy for $\alpha_g$'. 

\begin{theorem}[\cite{Structure94} Theorem~3]
Let $g\in G$ and $U$ be tidy for $g$. Then $U$ is tidy for every $h$ in the double coset~$UgU$. Hence $s\circ\alpha_{\bullet}$ is constant on this double coset.
\end{theorem}
The \emph{Braconnier topology} is a natural group topology on $\Aut(G)$ -- it is the coarsest group topology finer than the compact-open topology, \cite{HandR} {\color{fireenginered} cite Braconnier?}. The scale function may fail to be continuous with respect to the Braconnier topology however, see~\cite[Example 1]{axioms6040027}, and the following questions are open.
\begin{problem*}
\label{prob:scale_continuous} 
\begin{enumerate}
\item Characterise the \tdlc~groups $G$ for which $s : \Aut(G)\to \mathbb{N}$ is continuous with respect to the Braconnier topology.
\item Describe, in terms of the action on $G$, the coarsest group topology on $\Aut(G)$ with respect to which the scale function is continuous. Does this topology extend to the monoid $\End(G)$ so that the scale is continuous? 
\end{enumerate}
\end{problem*}

In the next section, it will be necessary to know that the restriction of flat groups to certain subgroups are flat. To see this, the following general but partial results will be combined with special arguments applying in the particular circumstances. These results were shown in~\cite{FurtherTidy} in connection with the stability of the scale under passing to subgroups and quotients. 
 \begin{proposition}[\cite{FurtherTidy} Lemmas 4.1 \& 4.5]
 \label{prop:functor}
 Let $G$ be a \tdlc~group, $\alpha\in\Aut(G)$ and suppose that $H\leq G$ is closed and $\alpha$-stable. Let $U$ be tidy for $\alpha$. Then:
 \begin{enumerate}
 \item\label{prop:functor1}
  $U\cap H$ is tidy below for $\alpha|_{H}$; and 
 \item\label{prop:functor2}
  if $H$ is also normal in $G$, then $UH/H$ is tidy above for $\alpha|^{G/H}$.
 \end{enumerate}
 \end{proposition}

It is generally not the case that the intersection of a subgroup tidy for $\alpha$ with a subgroup $H$ invariant under $\alpha$ is tidy for the restriction of $\alpha$ to $H$, see Example 6.4 in \cite{FurtherTidy} or Exercise~\ref{ex:tidymeet}. However, the following does hold and will be used later.
\begin{lemma}
\label{lem:tidymeetsshrinking}
Let $\alpha\in\Aut(G)$ and suppose that $U\leq G$ is tidy for $\alpha$. Suppose that $H\leq G$ is closed, is $\alpha$-stable and that $\left\{\alpha^n(x)\right\}_{n\geq0}$ has compact closure for every $x\in H$. Then $H\cap U \leq \Um{U}{\alpha}$ and $\alpha(H\cap U)\leq H\cap U$. \\
Hence $H\cap U$ is tidy for $\alpha|_H$.
\end{lemma}
\begin{proof}
If $x\in H\cap U$, then $x\in \Um{U}{\alpha}$ by Proposition~\ref{prop:tidy_criteria}. Hence $H\cap U\leq \Um{U}{\alpha}$. Hence $\alpha(H\cap U)\leq U$ and, since $H$ is $\alpha$-stable, $\alpha(H\cap U)\leq H\cap U$. That $H\cap U$ is tidy for $\alpha|_H$ follows by Proposition~\ref{prop:tidy_conditions}\eqref{prop:tidy_conditions1}.
\end{proof}
\bigskip 

We conclude this overview of the scale and tidy subgroups with an application that proves a  conjecture made by K.~H.~Hofmann in \cite{Hofmann_Conjecture} (although it does not relate directly to flatness.). The proof uses only that there is a function on the group $G$ having the properties of the scale stated in Theorem~\ref{thm:scale_properties}. The following proposition is obvious if $G$ is discrete and may fail to hold if $G$ is locally compact and connected, see Exercise~\ref{ex:P(G)_not_closed}. 
\begin{proposition}[\cite{Wi:tdHM} Theorem~2]
	\label{prop:per_closed}
	For a topological group $G$, denote 
	$$
	P(G) = \left\{ x\in G\mid  x \mbox{ is contained in a compact subgroup of }G\right\}.
	$$ 
	Then $P(G)$ is closed if $G$ is a \tdlc~group.
\end{proposition}
\begin{proof}
	Consider $x\in P(G)$ and denote by $\langle x\rangle^-$ the closed subgroup of $G$ generated by $x$. Then $\langle x\rangle^-$ is compact by hypothesis and $s(\langle x\rangle^-)$ is finite by continuity of $s$. Since $s(\langle x\rangle^-)$ is a finite set of positive integers, multiplicativity of $s$ on positive powers implies that $s(x) = 1$. 
	
	Next, let $y$ be in the closure of $P(G)$. Then $s(y)=1$ by continuity of $s$. Hence there is $U\in\COS(G)$ such that $yUy^{-1} = U$. Since $y\in \overline{P(G)}$ and $yU$ is a neighbourhood of $y$, there is $x\in P(G)$ such that $x\in yU$. Then $xUx^{-1}=U$ and  $\langle y\rangle \subset \langle x\rangle U$, which is compact. Therefore $y\in P(G)$ and $P(G)$ is closed.
\end{proof}

\subsection{Subgroups associated with automorphisms} 
 
 The minimizing, or tidy, subgroups associated with an endomorphism are generally not unique. For example, if $U$ is tidy for the automorphism $\alpha$, then it is a special case of Lemma~\ref{lem:conjugation_invariance} that $\alpha(U)$ is tidy for $\alpha$. This section introduces certain subgroups associated with an automorphism that \emph{are} unique while also being closely related to tidy subgroups and the scale. Although these subgroups may be defined for endomorphisms as well, see~\cite{Endo15}, the summary is confined to automorphisms.
 
\begin{definition}
\label{defn:paretc}
Let $G$ be a topological group and suppose that $\alpha\in\Aut(G)$. Define 
\begin{align*}
\parb(\alpha) &= \bigl\{ x\in G \mid (\alpha^n(x))_{n\in\mathbb{N}} \mbox{ has an accumulation point}\bigr\}\\
\lev(\alpha) &=   \parb(\alpha)\cap  \parb(\alpha^{-1})\\
\con(\alpha) &=  \bigl\{ x\in G \mid \alpha^n(x)\to \id_G \mbox{ as }n\to\infty\bigr\}\\
\mbox{ and }\ \nub(\alpha) &= \overline{\con(\alpha)}\cap \lev(\alpha) = \overline{\con(\alpha)}\cap \parb(\alpha^{-1}).
\end{align*}
These sets are subgroups of $G$ and are called, respectively, the \emph{parabolic}, \emph{Levi}, \emph{contraction} and \emph{nub} subgroups for $\alpha$.
\end{definition}

That $\con(\alpha)$ is a subgroup of $G$ may be directly verified but is generally not closed, see Exercise~\ref{ex:con_not_closed}. An argument using tidy subgroups shows $\parb(\alpha)$ is a subgroup of $G$ that is always closed. 
\begin{proposition}[\cite{Structure94} Proposition~3]
\label{prop:par_closed}
Let $G$ be a \tdlc~group and $\alpha$ be in $\Aut(G)$. Then $\parb(\alpha)$ is a closed subgroup of $G$ and
$$
\parb(\alpha) = \bigl\{ x\in G \mid\overline{\{\alpha^n(x)\}}_{n\in\mathbb{N}} \mbox{ is compact}\bigr\}.
$$
Hence $\lev(\alpha)$ and $\nub(\alpha)$ are closed subgroups too.
\end{proposition}

The names \emph{parabolic subgroup} and \emph{Levi subgroup} for~$\alpha$ are used in \cite{ContractionB} because they specialize to the subgroups of $p$-adic Lie groups with those names. The \emph{nub subgroup for $\alpha$} is defined in~\cite{nub_2014}, where the following alternative characterizations are given.
\begin{theorem}[\cite{nub_2014} Proposition~4.4 \& Theorem~4.1]
\label{thm:nub}
Let $G$ be a \\ \tdlc~group and $\alpha$ be in $\Aut(G)$. Then $\nub(\alpha)$ is:
\begin{itemize}
\item the intersection of all subgroups tidy for $\alpha$;
\item the largest compact $\alpha$-stable subgroup of $G$ having no relatively open, $\alpha$-stable subgroups; and 
\item the largest $\alpha$-stable closed subgroup of $G$ on which $\alpha$ acts ergodically. 
\end{itemize}
\end{theorem} 

The characterization of the nub in terms of the ergodic action of $\alpha$ extends the proof, for \tdlc~groups, of a conjecture of Halmos in~\cite{Halmos_ergodic} given in~\cite{Aoki_nub}. A proof using the scale and tidy subgroups was given in \cite{Previts_Wu}. The role of the nub in the study of the scale and tidy subgroups is that it gives a criterion for tidiness below of a compact open subgroup that applies independently of tidiness above, as follows.
\begin{proposition}[\cite{nub_2014} Corollaries 4.1 \& 4.2]
\label{prop:nub_in}
Let $\alpha\in\Aut(G)$ and suppose that $U\in\COS(G)$. Then: 
\begin{itemize}
	\item $U$ is tidy below for $\alpha$ if and only if it contains $\nub(\alpha)$, and 
	\item  if $\nub(\alpha)\leq U$, there is $n\geq0$ such that $\bigcap_{k=0}^n \alpha^k(U)$ is tidy for $\alpha$.
\end{itemize}
\end{proposition}

The following observations are used frequently.
\begin{lemma}
\label{lem:nub_normalised}
Let $\alpha\in\Aut(G)$. Then $\parb(\alpha)$ normalises $\con(\alpha)$, and $\lev(\alpha)$ normalises $\nub(\alpha)$.
\end{lemma}
\begin{proof}
Let $y\in\parb(\alpha)$. Then $\{\alpha^n(y)\}_{n\in\mathbb{N}}$ has compact closure by Proposition~\ref{prop:par_closed} and it follows that $\bigcap\left\{\alpha^n(y)U\alpha^n(y)^{-1}\mid n\in\mathbb{N} \right\}$ is an open neighbourhood of $\id_G$ for every $U\in\COS(G)$. Hence, for each $x\in \con(\alpha)$, 
$$
\alpha^n(yxy^{-1}) = \alpha^n(y)\alpha^n(x)\alpha^n(y)^{-1} \to \id_G \mbox{ as }n\to\infty
$$
and $\con(\alpha)$ is normalised by $\parb(\alpha)$. Hence, in particular, $\lev(\alpha)$ normalises $\con(\alpha)$ and therefore normalises $\nub(\alpha) = \overline{\con(\alpha)}\cap\lev(\alpha)$ too.
\end{proof}

\begin{lemma}
\label{lem:nubstable}
Let $\alpha\in\Aut(G)$ and suppose that $K\leq G$ is compact and $\alpha$-stable. Then, denoting the inner automorphism induced by $x$ as $\alpha_x$,\\ 
$\nub(\alpha_x\alpha) = \nub(\alpha)$ for every $x\in N$.
\end{lemma}
\begin{proof}
Since $K$ is compact and $\alpha$-stable, it is contained in $\lev(\alpha)$. Moreover, for each $x\in K$ and $n\in\mathbb{Z}$, there is $y_n\in K$ such that $(\alpha_x\alpha)^n = \alpha_{y_n}\alpha^n$. Then the argument of the previous lemma shows that $\con(\alpha) \leq \con(\alpha_x\alpha)$ for every $x\in K$. Similarly, if $u\in \parb(\alpha^{-1})$, so that $C = \overline{\{\alpha^{-n}(u)\}_{n\geq0}}$ is compact, then
$$
(\alpha_x\alpha)^{-n}(u) = \alpha_{y_{-n}}\alpha^{-n}(u)\in KCK,
$$
which is compact. Hence $\parb(\alpha)\leq \parb(\alpha_x\alpha)$. Therefore $\nub(\alpha)\leq \nub(\alpha_x\alpha)$. The same argument, applied to $\alpha_x\alpha$ and $\alpha_x^{-1}(\alpha_x\alpha)$, shows the reverse inclusion.
\end{proof}

\paragraph{Relative contraction groups} The notion of a contraction group for $\alpha$ extends to that of a relative contraction group modulo an $\alpha$-stable subgroup. Relative contraction groups are used in~\cite{ContractionB} when deriving the results about $\con(\alpha)$ reviewed in this section, and are also important for the study of flat groups in the next section.
\begin{definition}
\label{defn:relative_contraction}
Let $\alpha\in\Aut(G)$ and suppose that $H\leq G$ is compact and $\alpha$-stable. Then $\{x_n\}_{n\in\mathbb{N}}\subset G$ is said to \emph{converge to the identity modulo $H$}, denoted $\lim_{n\to\infty} x_n = \id\modulo {H}$, if for every open $\open{O}\supseteq H$ there is $N>0$ such that $x_n\in \open{O}$ for all $n\geq N$. Define 
\begin{equation*}
\con(\alpha/H) = \left\{ x\in G \mid \lim_{n\to\infty}\alpha^n(x) = \id\modulo {H}\right\}.
\end{equation*}
\end{definition}
It may be verified that $\con(\alpha/H)$ is a group\footnote{If $H$ is a general closed $\alpha$-stable subgroup of $G$, then $x_n$ is said to converge to $\id$ modulo $H$ if for every open $\open{O}\ni \id$ there is $N>0$ such that $x_n\in \open{O}H\open{O}$ for all $n\geq N$. This is equivalent to Definition \ref{defn:relative_contraction} when $H$ is compact and, with this definition, Theorem \ref{thm:convergence_mod_H} holds when $H$ is closed and not compact.}.  When $G$ is metrizable, the following is \cite[Theorem~3.8]{ContractionB} and it is \cite[Theorem~1]{Jaworski_contraction} in the general case. 
\begin{theorem}[\cite{ContractionB} \& \cite{Jaworski_contraction}]
\label{thm:convergence_mod_H}
Let $G$ be a \tdlc~group, $\alpha\in \Aut(G)$ and $H$ an $\alpha$-stable compact subgroup of $G$. Then $\con(\alpha/H) = \con(\alpha)H$.
\end{theorem}

The following results, which relativize the definition of the nub to any compact, $\alpha$-stable subgroup containing the nub, is an immediate consequence of Theorem~\ref{thm:convergence_mod_H}.
\begin{corollary}
\label{cor:convergence_mod_H}
Let $C\leq G$ be compact and $\alpha$-stable, with $C\geq \nub(\alpha)$. Then $\con(\alpha/C)\cap \parb(\alpha^{-1}) = C$.
\end{corollary}
\begin{proof}
	For each $u\in \con(\alpha/C)\cap \parb(\alpha^{-1})$ we have $u = u'c$ with $u'\in\con(\alpha)$ and $c\in C$. Then $u'\in\parb(\alpha^{-1})$, because $C\leq\parb(\alpha^{-1})$, and so $u'\in\nub(\alpha)$ by Definition~\ref{defn:paretc}. Hence $\con(\alpha/C)\cap \parb(\alpha^{-1})\leq C$. The reverse inclusion holds because $C\leq\con(\alpha/C)$ and $C$ is compact and $\alpha^{-1}$-stable.
\end{proof}

The relationship between contraction subgroups and the scale is seen in the next couple of results, shown in~\cite{ContractionB}. Although~\cite{ContractionB} treats metrizable groups only, this hypothesis unnecessary after~\cite{Jaworski_contraction}.
\begin{proposition}[\cite{ContractionB} Proposition~3.16]
\label{prop:U--^con}
Let $U\in\COS(G)$ and suppose that $\alpha\in\Aut(G)$. Set $\Uz{U}{\alpha} = \bigcap_{n\in\mathbb{Z}} \alpha^n(U)$. Then 
$$
\Umm{U}{\alpha} = \con(\alpha)\Uz{U}{\alpha}.
$$
\end{proposition}

Note that $\Uz{U}{\alpha}$ is contained in $\parb(\alpha)$. Hence Lemma~\ref{lem:nub_normalised} implies that $\Uz{U}{\alpha}$ normalises $\con(\alpha)$ and that the product is automatically a group.

\begin{proposition}[\cite{ContractionB} Proposition~3.21]
\label{prop:restriction_of_scale}
Let $\alpha\in\Aut(G)$ and suppose that $U$ is tidy for $\alpha$. Then $\overline{\con(\alpha^{-1})}$ and $\Upp{U}{\alpha}$ are stable under $\alpha$ and
$$
s(\alpha) = s(\alpha|_{H}) = \Delta(\alpha|_{H})
$$
for every closed $\alpha$-stable group $H$ between $\overline{\con(\alpha^{-1})}$ and $\Upp{U}{\alpha}$. 
\endproof
\end{proposition}
In particular, if $s(\alpha)>1$, then $\overline{\con(\alpha^{-1})}$ is not compact. Also see \cite[Proposition~3.24]{ContractionB}. 

It follows from Theorem \ref{thm:nub} that $\nub(\alpha) = \nub(\alpha^{-1})$ which, combined with Definition \ref{defn:paretc}, implies the next characterisation of $\nub(\alpha)$.
\begin{proposition}
\label{prop:con&nub}
Let $\alpha\in\Aut(G)$. Then $\nub(\alpha) = \overline{\con(\alpha)}\cap \overline{\con(\alpha^{-1})}$. 
\end{proposition} 
Note, however, that it may happen that ${\con(\alpha)}\cap {\con(\alpha^{-1})}$ is trivial while $\nub(\alpha)$ is not, see~\cite[Proposition 5.4 and Example 5.1]{nub_2014}.

The preceding results imply yet another characterisation of the nub, which will be useful later.
\begin{proposition}
\label{prop:nub_smallest}
Let $\alpha\in\Aut(G)$. Then $\con(\alpha/\nub(\alpha)) = \overline{\con(\alpha)}$ and, for any $C\leq G$ compact and $\alpha$-stable, $\con(\alpha/C)$ is closed if and only if $\nub(\alpha)\leq C$.
\end{proposition}
\begin{proof}
Suppose that $C$ is compact, $\alpha$-stable and contains $\nub(\alpha)$. Then Theorem \ref{thm:convergence_mod_H} and the fact that $\overline{\con(\alpha)} = \con(\alpha)\nub(\alpha)$, which is a case of Proposition~\ref{prop:U--^con}, imply that
$$
\con(\alpha/C) =  \con(\alpha)C= \overline{\con(\alpha)}C,
$$ 
and hence that $\con(\alpha/C)$ is closed. In the particular case that $C = \nub(\alpha)$, the equation reduces to $\con(\alpha/\nub(\alpha))=\overline{\con(\alpha)}$.

Conversely, suppose that $C$ is such that $\con(\alpha/C)$ is closed. Then, for any $U\in \COS(\con(\alpha/C))$ containing $C$ and tidy for $\alpha|_{\con(\alpha/C)}$,  we have
$$
\text{(i)}\ \bigcup_{m\leq0}\alpha^m(U) = \con(\alpha/C) \text{ and (ii)}\ \bigcap_{n\geq0}\alpha^n(U) = C.
$$
We also have that $\overline{\con(\alpha)}\leq \con(\alpha/C)$ and, in particular, $\nub(\alpha)\leq \con(\alpha/C)$. Then, since $\nub(\alpha)$ is compact, (i) implies that there is $m\leq0$ such that $\nub(\alpha)\leq \alpha^m(U)$. Hence, since $\nub(\alpha)$ is $\alpha$-stable,  $\nub(\alpha)\leq \bigcap_{n\geq m}\alpha^n(U)$. Then (ii) implies that $\nub(\alpha)\leq C$ as claimed.
\end{proof}

The connection between the nub and contraction group is highlighted more clearly in the special case when $\nub(\alpha)$ is trivial, as follows.
\begin{proposition}[\cite{ContractionB} Theorem~3.32]
\label{prop:con_closed}
Let $\alpha\in\Aut(G)$. Then the following are equivalent:
\begin{enumerate}
\item $\con(\alpha)$ is closed;
\item $\con(\alpha^{-1})$ is closed;
\item $\nub(\alpha)$ is trivial;
\item every neighbourhood of $\id_G$ contains a subgroup tidy for $\alpha$;
\item every $U\in\COS(G)$ that is tidy above for $\alpha$ is tidy.
\end{enumerate}
\end{proposition} 

 \begin{remark}
 \label{rem:closed_contraction}
The equivalent conditions in Proposition~\ref{prop:con_closed} are satisfied, for instance, by automorphisms of $p$-adic Lie groups and are not satisfied by hyperbolic elements of the automorphism group of a regular tree. 
 
When a contraction group is closed, it is itself a locally compact group when equipped with the subspace topology. Although beyond the scope of these notes, the structure of such groups can be described in great detail. In~\cite{ContractionG}, a composition series that breaks the group down into `simple' contraction group factors is described. It is further shown that every such contraction group is the direct sum of a torsion subgroup and a divisible subgroup, and that the divisible subgroup is a sum of nilpotent $p$-adic Lie groups. In~\cite{ContractionG2021}, it is shown that there are uncountably many nilpotent torsion contraction groups and, in~\cite{GlocknerWillis+2021}, it is shown that every locally pro-$p$ contraction group is nilpotent. 
 \end{remark}

\begin{remark}
\label{rem:tits_core}
The \emph{Tits core}, $G^\dagger$,  of a totally disconnected, locally compact group $G$ is the subgroup generated by $\{\overline{\con(x)} \mid x\in G\}$. This notion is defined in~\cite{CapReidW_Titscore} and extends a definition by J.~Tits in the case of algebraic groups. It may be shown that any dense subgroup of $G$ that is normalised by $G^\dagger$ must contain $G^\dagger$. In particular, if $G$ is topologically simple and $G^\dagger$ non-trivial, then $G^\dagger$ is an abstractly simple group.
\end{remark} 

\begin{remark}
\label{rem:contraciton_endo} 
It is shown by Tim Bywaters, Stephan Tornier and Helge Gl\"ockner in~\cite{ByGlTo} that many of the above results concerning contraction groups for automorphisms extend to endomorphisms.
\end{remark}

The next results can mostly be found, at least implicitly, in \cite{Structure94} but the conclusions are stated more clearly here. Part~\ref{prop:acc.pt_in_lev3} strengthens \cite[Proposition~3.4]{ContractionB}. The statements are in terms of $\bigcap_{l\in\mathbb{N}} \overline{\{\alpha^k(y)\}}_{k\geq l}$, which is non-empty if and only if the set $\{\alpha^k(y)\}_{k\in \mathbb{N}}$ is finite or has an accumulation point. 
\begin{proposition}
\label{prop:acc.pt_in_lev}
Suppose that $y\in G$. Then:
\begin{enumerate} 
\item $y\in \parb(\alpha)$ if and only if $\bigcap_{l\in\mathbb{N}} \overline{\{\alpha^k(y)\}}_{k\geq l}\ne \emptyset$; 
\label{prop:acc.pt_in_lev1}
\item $\bigcap_{l\in\mathbb{N}} \overline{\{\alpha^k(y)\}}_{k\geq l}\subseteq \lev(\alpha)$;   
\label{prop:acc.pt_in_lev2}
\item $\parb(\alpha) = \lev(\alpha)\con(\alpha)$ and indeed, if $U$ is tidy for $\alpha$ and $y\in \parb(\alpha)$, then $y = xuz$ with $x\in \bigcap_{l\in\mathbb{N}} \overline{\{\alpha^k(y)\}}_{k\geq l}$, $u\in \Uz{U}{\alpha}$ and $z\in\con(\alpha)$.  
\label{prop:acc.pt_in_lev3}
\end{enumerate}
\end{proposition}
\begin{proof}
\ref{prop:acc.pt_in_lev1}. If $y\in \parb(\alpha)$, then $\{\alpha^n(y)\}_{n\in\mathbb{N}}$ is either finite or has an accumulation point and in both cases $\bigcap_{l\in\mathbb{N}} \overline{\{\alpha^k(y)\}}_{k\geq l}\ne \emptyset$. Conversely, suppose that $\bigcap_{l\in\mathbb{N}} \overline{\{\alpha^k(y)\}}_{k\geq l}$ is not empty and let $x$ be in this set. If $x$ appears infinitely often in $\{\alpha^k(y)\}_{k\in\mathbb{N}}$, then this set is in fact a periodic sequence and $y\in \parb(\alpha)$. Otherwise, $x$ is an accumulation point of $\{\alpha^k(y)\}_{k\in\mathbb{N}}$ and $y\in\parb(\alpha)$. 

\ref{prop:acc.pt_in_lev2}. If $\bigcap_{l\in\mathbb{N}} \overline{\{\alpha^k(y)\}}_{k\geq l}$ is empty there is nothing to prove and, if $\{\alpha^k(y)\}_{k\in\mathbb{N}}$ is finite, then $\alpha^k(y)\in\lev(\alpha)$ for each $k$. Suppose that $\bigcap_{l\in\mathbb{N}} \overline{\{\alpha^k(y)\}}_{k\geq l}\ne \emptyset$ and that $\{\alpha^k(y)\}_{k\in\mathbb{N}}$ is infinite. Since $\bigcap_{l\in\mathbb{N}} \overline{\{\alpha^k(y)\}}_{k\geq l}$ is invariant under $\langle\alpha\rangle$, it suffices to show that this set is compact. Consider $x\in \bigcap_{l\in\mathbb{N}} \overline{\{\alpha^k(y)\}}_{k\geq l}$, so that $x$ is an accumulation point of $\{\alpha^k(y)\}_{k\in\mathbb{N}}$, and let $U\in \COS(G)$ be tidy for $\alpha$. Then $xU$ is a neighbourhood of~$x$ and so there are $m<n$ and $u\in U$ with $\alpha^n(y) = \alpha^m(y)u$ in $xU$. We will show that $u\in U_{\alpha-}$. 

Applying $\alpha^{i(n-m)}$ with $i\geq1$ and substituting $\alpha^m(y)u$ for $\alpha^n(y)$ yields
\begin{equation*}
\alpha^{i(n-m)}(\alpha^m(y)) = \alpha^{m}(y)\alpha^j(u\alpha^{n-m}(u)\dots \alpha^{(i-1)(n-m)}(u)).
\end{equation*}
Then, applying $\alpha^j$ with $j\in\{0,\dots, n-m-1\}$, we have
\begin{equation}
\label{eq:acc.pt_in_lev}
\alpha^{i(n-m)+j}(\alpha^m(y)) = \alpha^{m+j}(y)\alpha^j(u\alpha^{n-m}(u)\dots \alpha^{(i-1)(n-m)}(u)).
\end{equation}
At least one of the subsets $C_j = \left\{\alpha^{i(n-m)+j}(\alpha^m(y))\right\}_{i\geq1}$ has $x$ as an accumulation point. Hence, by~\eqref{eq:acc.pt_in_lev}, $\left\{u\alpha^{n-m}(u)\dots \alpha^{(i-1)(n-m)}(u)\right\}_{i\geq1}$ has an accumulation point and it follows, by Proposition~\ref{prop:tidy_criteria}, that $u\in U_{\alpha-}$.

Equation~\eqref{eq:acc.pt_in_lev} implies that  
$$
\alpha^{i(n-m)+j}(\alpha^m(y)) \in \alpha^{m+j}(y)\Um{U}{\alpha}
$$ 
for all $i\geq1$ and $j\in\{0,\dots, n-m-1\}$. In other words
$$
\alpha^l(y) \in \left\{\alpha^{m}(y),\dots, \alpha^{n-1}(y)\right\} \Um{U}{\alpha},
$$
which is compact, for all $l\geq n$. Therefore $\bigcap_{l\in\mathbb{N}} \overline{\{\alpha^k(y)\}}_{k\geq l}$ is compact as claimed. 

\ref{prop:acc.pt_in_lev3}. It is clear that $\parb(\alpha)$ contains both $\con(\alpha)$ and $\lev(\alpha)$, and hence it contains their product because it is a group, by Proposition~\ref{prop:par_closed}. For the reverse inclusion, consider $y\in\parb(\alpha)$ and choose $x'\in \bigcap_{l\in\mathbb{N}} \overline{\{\alpha^k(y)\}}_{k\geq l}$. If $\bigcap_{l\in\mathbb{N}} \overline{\{\alpha^k(y)\}}_{k\geq l}$ is finite, then $y\in \lev(\alpha)$. Otherwise, $\bigcap_{l\in\mathbb{N}} \overline{\{\alpha^k(y)\}}_{k\geq l}$ is infinite and, as seen in \ref{prop:acc.pt_in_lev2}., there is $m\geq0$ such that $\alpha^m(y)\in x'\Um{U}{\alpha}$. Proposition~\ref{prop:U--^con} shows that $U_{\alpha-}\leq \Uz{U}{\alpha}\con(\alpha)$ and we have seen in part~\ref{prop:acc.pt_in_lev2}.~that $x'$ is in $\lev(\alpha)$. Hence $\alpha^{m}(y)\in x'\Uz{U}{\alpha}\con(\alpha) \subseteq \lev(\alpha)\con(\alpha)$. Finally, we have that $y \in x\Uz{U}{\alpha}\con(\alpha)$ with $x=\alpha^{-m}(x')\in \lev(\alpha)$ because $\con(\alpha)$ and $\lev(\alpha)$ are invariant under $\alpha$.
\end{proof}

The following result is a case of~\cite[Theorem 4.5]{Reid_DynamicsNYJ_2016}.
\begin{proposition}
	\label{prop:tidy_normalized}
Let $U\in\COS(G)$ be tidy for $\alpha$ and $L\leq G$ be compact and $\alpha$-stable. Then $U'= \bigcap_{x\in L} xUx^{-1}$ is tidy for $\alpha$.
	\end{proposition}
\begin{proof}
That $U'$ is compact and open holds because $L$ is compact and the intersection defining $U'$ is finite. Towards showing that $U'$ is tidy above, denote $\con(\alpha^{\pm1})\cap U=\Upm{C}{\alpha}$ and observe that Proposition~\ref{prop:U--^con} implies that $\Upm{U}{\alpha} = \Upm{C}{\alpha}\Uz{U}{\alpha}$. Then, since $\Uz{U}{\alpha}$ normalises $\con(\alpha)$ and $\con(\alpha^{-1})$, by Lemma~\ref{lem:nub_normalised}, tidiness above of $U$ may be expressed as
$$
U = \Up{C}{\alpha}\Uz{U}{\alpha}\Um{C}{\alpha}.
$$
In particular, if $u\in U'$, there are $c_\pm\in \Upm{C}{\alpha}$ and $w\in \Uz{U}{\alpha}$ such that $u = c_+wc_-$. To show that $U'$ is tidy above for $\alpha$, it suffices to show that $c_\pm$ are in $U'$, that is, that $xc_\pm x^{-1}$ are in $U$ for every $x\in L$. (That $w\in U'$ then follows because $U'$ is a group.) To this end, let $x\in L$. Then, since $xux^{-1}$ is in $U$,  there are $c'_\pm\in C_\pm$ and $w'\in \Uz{U}{\alpha}$ such that 
\begin{equation}
\label{eq:U'_tidyabove}
xux^{-1} = (xc_+wx^{-1})(xc_-x^{-1}) = c'_+w'c'_-.
\end{equation}
Hence
$$
(xc_-x^{-1})(c'_-)^{-1} = (xc_+wx^{-1})^{-1}c'_+w'.
$$
Then $(xc_-x^{-1})(c'_-)^{-1}\in \con(\alpha)$ because Lemma~\ref{lem:nub_normalised} shows that $\con(\alpha)$ is normalised by $L$ and that the right side of the equation belongs to $\parb(\alpha^{-1})$. Hence, by Definition~\ref{defn:paretc}, $(xc_-x^{-1})(c'_-)^{-1}\in \nub(\alpha)$ and it follows by Theorem~\ref{thm:nub} that $(xc_-x^{-1})(c'_-)^{-1}\in U$. Since $c_-'\in U$, we have that $xc_-x^{-1}\in U$ as required. A similar argument shows that $xc_+x^{-1}\in U$ too.

That $U'$ is tidy below holds because $\nub(\alpha)$ is normalised by the compact, $\alpha$-stable subgroup $L$, by Lemma~\ref{lem:nub_normalised}. Hence $\nub(\alpha)\leq U'$ and $U'$ is tidy below by Proposition~\ref{prop:nub_in}. Therefore $U'$ is tidy for every $\alpha\in\fl{H}$. 
\end{proof}

\begin{remark}
\label{examp:tidy_normalized}	
The claim of Proposition~\ref{prop:tidy_normalized} does not hold for subgroups that are only tidy above. 

For example, let $F$ be a finite group and define $G = (F\times F)^{\mathbb{Z}}\rtimes \langle\sigma\rangle$ where $\sigma$ acts on $(f_n,g_n)\in(F\times F)^{\mathbb{Z}}$ by $\sigma(f_n,g_n) = (g_n,f_n)$. 
Let $\alpha$ be the shift automorphism of $G$ defined by $\alpha(f_n,g_n) = (f_{n+1},g_{n+1})$. Then $\langle\sigma\rangle$ is compact and is $\alpha$-stable because $\alpha$ and $\sigma$ commute. Let $U$ be the compact open subgroup of $G$
$$
U = \left\{(f_n,g_n)\in (F\times F)^{\mathbb{Z}}\mid f_0 = \id_F=g_3\right\}.
$$
Then $U$ is tidy above for $\alpha$ with 
\begin{align*}
U_+ &= \left\{ (f_n,g_n)\in(F\times F)^{\mathbb{Z}}\mid f_n=\id_F\text{ if }n\leq0,\ g_n=\id_F\text{ if }n\leq3\right\}\\
\text{and }U_- &= \left\{ (f_n,g_n)\in(F\times F)^{\mathbb{Z}}\mid f_n=\id_F\text{ if }n\geq0,\ g_n=\id_F\text{ if }n\geq3\right\}.
\end{align*}
However 
$$
\bigcap_{n\in\mathbb{Z}} \sigma^n(U) = U\cap \sigma(U) = \left\{(f_n,g_n)\in (F\times F)^{\mathbb{Z}}\mid f_n = \id_F=g_n\text{ if }n=0\text{ or }3\right\}
$$
is not tidy above for $\alpha$.

\paragraph{The tidying procedure} A 3-step procedure which, given a compact open subgroup $U$ of a \tdlc~group $G$ and an automorphism $\alpha$ of $G$, produces a subgroup tidy for $\alpha$ is described in~\cite{Structure94} and~\cite{FurtherTidy}. \begin{description}
	\item[Step 1] Find $n\geq0$ such that $V = \bigcap_{k=0}^n \alpha^k(U)$ is tidy above for $\alpha$. 
	\item[Step 2] Identify a compact $\alpha$-stable group $L\leq G$ that, in the more recent language of~\cite{nub_2014}, contains $\nub(\alpha)$. 
	\item[Step 3] Adjoin $L$ to $V$ to produce a subgroup tidy for $\alpha$. 
\end{description}
`Adjoining' $L$ to $V$ in Step 3 is not straightforward because the product $VL$ need not be a group, see~\cite[Example 6.1]{FurtherTidy}. In~\cite{Structure94}, Step~3 is done by defining $V' = \bigcap_{x\in L} xVx^{-1}$ and forming the group $V'L$. Then $V'L$ is tidy below and applying Step~1 again produces a group that is tidy for $\alpha$. Proposition~\ref{prop:tidy_normalized} implies that the second application of Step~1 is not required after all, that is, that $V'L$ is tidy, even though the above example shows that $V'$ need not be tidy above. In~\cite{FurtherTidy}, $L$ is adjoined to $V$ by defining $V'' = \left\{v\in V\mid Lv\subseteq VL\right\}$. Then $V''L$ is a group, $V''$ is tidy above, $V''L$ is tidy for $\alpha$. This argument is more direct and also shows that $[\alpha(W):\alpha(W)\cap W]$ is minimized and equal to $s(\alpha)$ when $W$ is tidy. See the comparison of the two approaches at the end of~\cite[\S3]{FurtherTidy}, and that the second approach is required to prove minimality in~\cite[Example 6.3]{FurtherTidy}.
\end{remark}


\subsection{The tree representation theorem and scale groups}
\label{sec:tree_rep}

The following representation of $\Upp{U}{\alpha}\rtimes\langle\alpha\rangle$ as an isometry group of a regular tree is not needed for the description of flat groups to come. It does apply to the analogues of root subgroups, however, and shows that these abstract defined subgroups of general \tdlc~groups may be represented as concrete groups of isometries of a tree. The regular tree in which every vertex has degree~$q$ is denoted by $\tree_q$.
\begin{theorem}[Tree Representation Theorem \cite{ContractionB}]
\label{thm:tree_rep}
Let $\alpha\in \Aut(G)$ and $U\in\COS(G)$ be tidy for $\alpha$. Suppose that $s(\alpha)=s$. Then there is a homomorphism $\phi : \Upp{U}{\alpha}\rtimes \langle\alpha\rangle\to \Isom(\tree_{s+1})$  such that:
\begin{enumerate}
\item $\phi(\Upp{U}{\alpha}\rtimes \langle\alpha\rangle)$ fixes an end, $\omega$, of $\tree_{s+1}$ and is transitive on the vertices of $\tree_{s+1}$;
\item $\phi(\Upp{U}{\alpha})$ is contained in the set of elliptic elements of $\Isom(\tree_{s+1})$ (that is, the tree isometries having finite orbits) and $\phi(\alpha)$ is a hyperbolic isometry (that is, has infinite orbits) that has $\omega$ as an attracting end and which translates its axis by distance~$1$;
\item $\phi(\Upp{U}{\alpha}\rtimes \langle\alpha\rangle)$ is a closed subgroup of $\Isom(\tree_{s+1})$ and the kernel of $\phi$ is the largest compact normal subgroup of $\Upp{U}{\alpha}\rtimes \langle\alpha\rangle$; and
\item there is a vertex, $v_0$, on the axis of $\phi(\alpha)$ such that $\phi(\Up{U}{\alpha})$ is the subgroup of $\phi(\Upp{U}{\alpha}\rtimes \langle\alpha\rangle)$ that fixes $v_0$, and hence also fixes every vertex on the ray $[v_0,\omega)$.
\end{enumerate}
\end{theorem}

The tree representation theorem is a particular case of the Bass-Serre theory that represents a graph of groups on a regular tree. This comes about because $\alpha^{-1}$ is an endomorphism $\Up{U}{\alpha}\to \Up{U}{\alpha}$ such that $\alpha^{-1}(\Up{U}{\alpha})$ has index $s$ in $\Up{U}{\alpha}$ and thus gives rise to an HNN extension $\Up{U}{\alpha}\ast_{\alpha^{-1}}$. However, the tree $\tree_{s+1}$ may be defined directly in this case because we already have the group $\Upp{U}{\alpha}\rtimes \langle\alpha\rangle$ and do not need to define it using the HNN extension. For this, let the vertex set be the coset space $(\Upp{U}{\alpha}\rtimes \langle\alpha\rangle)/\Up{U}{\alpha}$ and the edge set be $\left\{ \{x\alpha^n\Up{U}{\alpha},x\alpha^{n+1}\Up{U}{\alpha}\} \mid x\in \Upp{U}{\alpha},\ n\in\mathbb{Z}\right\}$. Then the translation action of $\Upp{U}{\alpha}\rtimes \langle\alpha\rangle$ on cosets represents the group as tree isometries.

As stated in Theorem~\ref{thm:tree_rep}, $\phi(\Upp{U}{\alpha}\rtimes \langle\alpha\rangle)$ is a closed group of isometries that is transitive on the vertices and fixes an end of $\tree_{s+1}$. Such groups are called \emph{scale groups} in~\cite{GWil_scalegps}. Every scale group is a \tdlc~group and all scale groups occur in this manner. A close connection between scale groups and self-replicating groups, see~\cite{Nekrash} 
, is seen in~\cite{GWil_scalegps}.

\section{Flat Groups of Automorphisms}
\label{sec:flatness}

Groups of automorphisms of a \tdlc~group for which there exist a common tidy subgroup are the main topic of these notes. Such groups of automorphisms are called \emph{flat} (a name suggested by U.~Baumgartner), and it is seen here that they have a particular structure that leads to refinements of the tidiness conditions and the scale.  

The name `flat' follows from the structure of the groups, namely: a group of automorphisms, $\fl{H}\leq \Aut(G)$, sharing a tidy subgroup, $U\leq G$, is free abelian modulo the normaliser of $U$; and the free abelian group $\mathbb{Z}^d$ is geometrically flat in the sense that it is cocompact in the euclidean space $\mathbb{R}^d$. The exponent $d$ is called the \emph{rank} of the flat group $\fl{H}$. Further evidence for the aptness of the name `flat' is seen in \cite{BaumRemyWill}, where it is shown that maximal flat groups of inner automorphisms of the isometry group of a building translate a (metrically flat) apartment of the building  and that the rank of the group is equal to the dimension of the apartment; and also in \cite{BaumMollWill}, where it is shown that flat groups of inner automorphisms of hyperbolic \tdlc~groups can have rank at most~$1$. The geometric interpretation of flatness of a group of automorphisms is given a partial converse in~\cite{BaumSchlichtW_geomflat}. 

The tidiness conditions of Theorem~\ref{thm:characterise_minimizing} resolve into conditions that $U$ satisfies for every automorphism $\alpha\in \fl{H}$: tidiness above now expresses $U$ as a product of subgroups, each of which is expanded or shrunk by every $\alpha\in \fl{H}$; and tidiness below says that each of the factors of $U$ dilates under the action of $\fl{H}$ to a closed subgroup of $G$. The amount by which $\alpha\in \fl{H}$ expands a given factor of $U$ is a positive integer called the \emph{scale of $\alpha$ relative to the factor}, and the scale of $\alpha$ is equal to the product of its relative scales. In examples with $G$ a Lie group over a local field $\mathbb{K}$ or the isometry group of a building and $\fl{H}$ a maximal abelian subgrou, logarithms of the relative scales correspond to the $\mathbb{Z}$-dual roots defined in \cite{Tits_Kac-Moody}, and the factors of $U$ correspond to root subgroups.  

Flatness of groups of automorphisms was introduced in~\cite{SimulTriang}, although not under that name, and the results summarised above were established. The treatment of these ideas presented here differs from that in \cite{SimulTriang} by, in addition to using the word `flat', making use of contraction subgroups from \cite{ContractionB} and the nub subgroup from \cite{Reid_DynamicsNYJ_2016,nub_2014}, and concentrating on the structure of flat groups and the factoring of subgroups tidy for them. Unlike the argument in \cite{SimulTriang}, the proof given here does not attempt to show at the same time that finitely generated abelian groups are flat. A separate proof of that fact, consolidated with a proof of the result shown in \cite{Shalom-Willis} that finitely generated nilpotent groups are flat and that polycyclic groups are virtually flat, will be  presented elsewhere. The treatment in these notes is consequently more direct and, it is hoped, more accessible. 
\subsection{Flat groups and affiliated subgroups}
\label{sec:basic-flat}

Flat groups, the central concept of these notes, are defined in this subsection along with their affiliated nub and Levi subgroups. The properties of the nub and Levi subgroups used in the description of flat groups and their tidy subgroups are derived here.
\begin{definition}
\label{defn:flat} A group $\fl{H}$ of automorphisms of a \tdlc~group $G$ is \emph{flat} if there is $U\in\COS(G)$ that is tidy for every $\alpha\in\fl{H}$. The subgroup $U$ is said to be \emph{tidy for $\fl{H}$}. 

The \emph{uniscalar} subset\footnote{The term `uniscalar' was introduced by Theodore Palmer,~\cite[Definition 12.3.25]{PalmerII} for single automorphisms, and the notation $\fl{H}_u$ for the uniscalar subset was first used by Colin Reid in~\cite{Reid_DynamicsNYJ_2016}. 
} of the flat group $\fl{H}$ is 
$$
\fl{H}_u = \left\{ \alpha\in\fl{H} \mid s(\alpha) = 1 = s(\alpha^{-1})\right\}.
$$
\end{definition} 


\begin{proposition}
\label{prop:invariance_of_tidiness}
Suppose that $\fl{H}\leq\Aut(G)$ is flat with $U$ tidy for $\fl{H}$. Then, for any $\beta\in\fl{H}$, $\beta(U)$ is tidy for $\fl{H}$. 
\end{proposition}
\begin{proof}
Lemma~\ref{lem:conjugation_invariance} shows that $\beta(U)$ is tidy for $\beta\alpha\beta^{-1}$ for all $\alpha,\beta\in\fl{H}$. Since $\beta\fl{H}\beta^{-1}=\fl{H}$, it follows that $\beta(U)$ is tidy for $\fl{H}$. 
\end{proof}

The following is immediate from Theorem~\ref{thm:scale_properties} and Proposition~\ref{prop:invariance_of_tidiness}.
\begin{proposition}
\label{prop:uniscalar}
The uniscalar subset $\fl{H}_u$ of the flat group $\fl{H}$ is equal to 
$$
\left\{ \alpha\in\fl{H} \mid \alpha(U)=U \mbox{ for any }U\mbox{ tidy for }\fl{H}\right\},
$$ 
and is a normal subgroup of $\fl{H}$.
\endproof
\end{proposition}  

The following immediate consequence is a step towards describing the structure of subgroups tidy for flat groups.
\begin{corollary}
\label{cor:uniscalar}
Let $\fl{H}\leq\Aut(G)$ be flat and suppose that $\alpha,\beta\in\fl{H}$ and that $\alpha\in\fl{H}_u$. Then $\alpha(\Up{U}{\beta}) = \Up{U}{\beta}$ for every $U$ tidy for $\fl{H}$. 
\end{corollary}
\begin{proof}
Proposition~\ref{prop:invariance_of_tidiness} and Theorem~\ref{thm:tidy_intersection} show that $\bigcap_{k=0}^n \beta^k(U)$ is tidy for $\fl{H}$ for every $n\geq0$. Then Proposition~\ref{prop:uniscalar} shows that 
$$
\alpha\bigl({\textstyle\bigcap_{k=0}^n \beta^k(U)}\bigr) = {\textstyle\bigcap_{k=0}^n \beta^k(U)}\mbox{ for all }n\geq0
$$ 
and hence that $\alpha(\Up{U}{\beta}) = \alpha\bigl(\bigcap_{k=0}^\infty \beta^k(U)\bigr) = \Up{U}{\beta}$. 
\end{proof}

The following extends the notion of the contraction group modulo a stable subgroup, see Definition~\ref{defn:relative_contraction}, to sets, $\fl{A}$, of automorphisms that are typically semigroups.
\begin{definition}
\label{defn:relative_contraction_semigroup}
Let $\fl{A}\subset \fl{H}$ and suppose that $K\leq G$ is compact and $\fl{H}$-stable. Define 
\begin{equation*}
\con(\fl{A}/K) = \bigcap\left\{ \con(\alpha/K) \mid \alpha\in\fl{A}\right\}.
\end{equation*}
\end{definition}

For an individual automorphism $\alpha$, the parabolic and contraction subgroups are defined in terms of the monoid generated by $\alpha$, while the nub and Levi subgroups, on the other hand, are defined in terms of the group generated by $\alpha$. The definitions of the nub and Levi subgroups may therefore be extended to flat groups of automorphisms as follows.
\begin{definition}
\label{defn:nub}
The \emph{nub} for the flat group $\fl{H}\leq\Aut(G)$ is
\begin{align*}
\nub(\fl{H}) &= \bigcap\left\{ U \mid U\mbox{ is tidy for }\fl{H}\right\}\\
\intertext{and the \emph{Levi subgroup} is }
\lev(\fl{H}) &= \left\{ x\in G \mid \{ \alpha(x)\mid \alpha\in\fl{H}\}^-\mbox{ is compact}\right\}.
\end{align*}
\end{definition}
The definitions of $\lev(\fl{H})$ and $\nub(\fl{H})$ are motivated by Definition~\ref{defn:paretc} and Theorem~\ref{thm:nub} respectively, with the definition of $\nub(\fl{H})$ first being given in~\cite{Reid_DynamicsNYJ_2016}. It is immediate that $\nub(\fl{H})$ is a compact group and it follows from Proposition~\ref{prop:invariance_of_tidiness} that it is $\fl{H}$-stable. The following lemmas will also be useful.
\begin{lemma}
\label{lem:open_nub}
Let $\fl{H}\leq\Aut(G)$ be flat and $\open{O}\supseteq \nub(\fl{H})$ be open. Then there is a subgroup, $U$, tidy for $\fl{H}$ such that $\open{O}\supseteq U\supseteq \nub(\fl{H})$. 
\end{lemma}
\begin{proof}
The set $U\setminus\open{O}$ is compact for every $U$ tidy for $\fl{H}$ and the intersection of all such sets is empty because $\nub(\fl{H})\subset\open{O}$. Since compact sets satisfy the finite intersection property, there is therefore a finite set $\{U_1,\dots,U_n\}$ of subgroups tidy for $\fl{H}$ such that $(U_1\setminus\open{O})\cap \dots \cap (U_n\setminus\open{O})$ is empty. Then $U_1\cap \dots \cap U_n$ is tidy for $\fl{H}$ and contained in $\open{O}$.
\end{proof}

\begin{lemma}
\label{lem:lev_closed}
Let $\fl{H}\leq\Aut(G)$ be flat and $U$ be tidy for $\fl{H}$. Then 
\begin{equation}
\label{eq:lev_closed}
U\cap\lev(\fl{H}) = U\cap\bigcap\left\{\lev(\alpha)\mid\alpha\in\fl{H}\right\} = \bigcap \left\{\alpha(U)\mid a\in\fl{H}\right\} 
\end{equation}
and $\lev(\fl{H})$ is a closed subgroup of $G$.
\end{lemma}
\begin{proof}
Since $\lev(\alpha)\geq \lev(\fl{H})$ for every $\alpha\in\fl{H}$, the left side of \eqref{eq:lev_closed}
 is contained in the right. To see the reverse inclusion, let $\alpha\in\fl{H}$ and let $x$ be in $U\cap \lev(\alpha)$. Then $\{\alpha^n(x)\}_{n\in\mathbb{Z}}$ has compact closure and so, by Proposition~\ref{prop:tidy_criteria}, $\{\alpha^n(x)\}_{n\in\mathbb{Z}}\subset U$. In particular, $\alpha(x)\in U$. Since this holds for all $\alpha\in\fl{H}$,
$$
U\cap\bigcap\left\{\lev(\alpha)\mid\alpha\in\fl{H}\right\}\subseteq U\cap\lev(\fl{H}).
$$

That $\lev(\fl{H})$ is a group holds because
$$
\{ \alpha(xy^{-1})\mid \alpha\in\fl{H}\}^- \subseteq \{ \alpha(x)\mid \alpha\in\fl{H}\}^-\left(\{ \alpha(y)\mid \alpha\in\fl{H}\}^-\right)^{-1}
$$
and the product of compact sets is compact. Since $\lev(\alpha)$ is closed for every $\alpha\in\fl{H}$, by Proposition~\ref{prop:par_closed}, Equation~\eqref{eq:lev_closed} shows that $U\cap\lev(\fl{H})$ is closed. Therefore $\lev(\fl{H})$ is closed, by \cite[Theorem II.5.9]{HandR} or \cite[Ch. III, p.~221]{BourbakiTop}.
\end{proof}
\begin{remark}
	\label{rem:Levi}
Lemma~\ref{lem:lev_closed} does preclude that 
	$$
	\lev(\fl{H})< \bigcap\bigl\{ \lev(\alpha) \mid \alpha\in\fl{H}\bigr\}.
	$$ 
	For example, let $G$ be the discrete group of finitary permutations of $\mathbb{N}$ and $\fl{H}$ be the group of inner automorphisms of $G$. Then every element of $\fl{H}$ has finite order and $\lev(\alpha)=G$ for every $\alpha\in\fl{H}$, whereas $\lev(\fl{H}) = \{\id\}$. An additional hypothesis that $\fl{H}$ be finitely generated does not help, as shown by the examples of finitely generated infinite groups in which every element has finite order constructed in \cite{Golod}, \cite{Adian} and \cite{Grigorchuk_Burnside}.
\end{remark}

The group $\bigcap \left\{\alpha(U)\mid a\in\fl{H}\right\}$ is compact and $\fl{H}$-stable, and is denoted by $\Uz{U}{\fl{H}}$ in what follows. On the other hand, given a compact $\fl{H}$-stable subgroup, $N\leq G$ there is a subgroup, $V\leq G$, tidy for $\fl{H}$ and such that $N$ is contained in $\Uz{V}{\fl{H}}$. That is one part of the next result. Some of the other parts are contained in \cite[Theorem~4.5 \& Corollary~4.6]{Reid_DynamicsNYJ_2016}.
\begin{proposition}
	\label{prop:Nnormal_core_of_tidy}
Let $\fl{H}\leq \Aut(G)$ be flat and let $L\leq G$ be compact and $\fl{H}$-stable. Then, for every $U\in\COS(G)$ tidy for $\fl{H}$, the subgroup
$$
U' := \bigcap\left\{xUx^{-1}\mid x\in L\right\} 
$$
is tidy for $\fl{H}$ and normalised by $L$. Hence:
\begin{enumerate}
	\item \label{prop:Nnormal_core_of_tidy1}
$V := LU'$ is in $\COS(G)$, is tidy for $\fl{H}$ and contains $L$;
	\item \label{prop:Nnormal_core_of_tidy2}
	$\nub(\fl{H})$ is normalised by $L$ and $\tilde{L} := L\nub(\fl{H})$ is a compact $\fl{H}$-stable group; and 
	\item \label{prop:Nnormal_core_of_tidy3}
	$\fl{H}' := \fl{H}\ltimes \left\{\alpha_x\mid x\in L\right\}$ is a flat subgroup of $\Aut(G)$.
\end{enumerate}
Moreover, 
\begin{enumerate}
	\setcounter{enumi}{3}
	\item  \label{prop:Nnormal_core_of_tidy4}
	For every $\alpha\in\fl{H}$, $L$ normalises $\con(\alpha/\nub(\fl{H}))$ and $\con(\alpha/\tilde{L}) = L\con(\alpha/\nub(\fl{H}))$. 
\end{enumerate}
\end{proposition}
\begin{proof}
That $U'$ is tidy for $\fl{H}$ holds because it is tidy for every $\alpha\in\fl{H}$, by Proposition~\ref{prop:tidy_normalized}. 

\ref{prop:Nnormal_core_of_tidy1}. Since $L$ normalises $U'$, $V := LU'$ is a group: it is compact because $L$ and $U'$ are, and is open because $U'$ is. Since $U'$ is tidy for every $\alpha\in \fl{H}$, we have that $\nub(\alpha)\leq U'\leq V$ and $V = (L\Up{U'}{\alpha})\Um{U'}{\alpha}$. It may be checked that $L\Up{U'}{\alpha}\leq \Up{V}{\alpha}$ and $\Um{U'}{\alpha}\leq \Um{V}{\alpha}$. Hence $V$ is tidy for every $\alpha\in\fl{H}$.

\ref{prop:Nnormal_core_of_tidy2}. By Definition~\ref{defn:nub}, $\nub(\fl{H})$ is the intersection of all subgroups tidy for $\fl{H}$. Since every tidy subgroup contains one normalised by $L$, it follows that $\nub(\fl{H})$ is normalised by $L$ too. 

\ref{prop:Nnormal_core_of_tidy3}. Note that $\left\{\alpha_x\mid x\in L\right\}$ is a group of inner automorphisms of $G$ because $L$ is a group. This group is normalized by $\fl{H}$ because $\fl{H}$ stabilizes $L$. Hence $\fl{H}'$ is a group. We shall see that $\fl{H}'$ is flat by showing that $V$ is a tidy subgroup. Consider $(\alpha,\alpha_x)\in \fl{H}'$. Then  $\nub(\alpha,\alpha_x) = \nub(\alpha)$, by Lemma~\ref{lem:nubstable}, and $V$ is tidy below for $(\alpha,\alpha_x)$. For tidiness above, denote $\con(\alpha^{\pm1})\cap V$ by $\Upm{C}{\alpha}$. Then $\Upm{V}{\alpha} = \Upm{C}{\alpha}\Uz{V}{\alpha}$ and, since $\Upm{C}{\alpha}$ and $\Uz{V}{\alpha}$ are normalized by $L$, we have that $\Upm{V}{(\alpha,\alpha_x)} = \Upm{V}{\alpha}$. Hence $V$ is tidy above for $(\alpha,\alpha_x)$. Therefore $V$ is tidy for every $(\alpha,\alpha_x) \in\fl{H}'$ and $\fl{H}'$ is flat. 

\ref{prop:Nnormal_core_of_tidy4}. Theorem~\ref{thm:convergence_mod_H} shows that $\con(\alpha/\nub(\fl{H})) = \con(\alpha)\nub(\fl{H})$ and, since $L\leq \lev(
\alpha)$, Lemma~\ref{lem:nub_normalised} shows that $L$ normalizes $\con(\alpha)$. Hence $L$ normalizes $\con(\alpha/\nub(\fl{H}))$ and, appealing to Theorem~\ref{thm:convergence_mod_H} again,
$$
L\con(\alpha/\nub(\fl{H})) = \con(\alpha)L\nub(\fl{H}) = \con(\alpha)\tilde{L} = \con(\alpha/\tilde{L}).
$$
\end{proof}

\begin{remark}
	The fact that $\nub(\alpha)$ is normalised by $L$ is used at a couple of places in the proof of Proposition~\ref{prop:Nnormal_core_of_tidy}. This fact implies directly that the group $\langle\nub(\alpha)\mid \alpha\in\fl{H}\rangle^-$ is normalised by $L$. However, this group, which is called the \emph{lower nub} $\lnub(\fl{H})$ in \cite{Reid_DynamicsNYJ_2016}, may be strictly smaller than $\nub(\fl{H})$. The separate argument given to show Part~\ref{prop:Nnormal_core_of_tidy2}.~of Proposition~\ref{prop:Nnormal_core_of_tidy} is needed therefore. A comparison between the nub and lower nub is made in \S\ref{sec:Levi_flat}. 
	
A special case of Proposition~\ref{prop:Nnormal_core_of_tidy} in which $\fl{H}$ is abelian by finite is given in \cite[Lemma~25]{BaumRemyWill}.
\end{remark}

\subsection{Roots for a flat group and the scaling subgroups associated with them}
\label{sec:irroots}

Roots of a flat group, $\fl{H}$, of automorphisms of $G$ and subgroups of $G$ that are scaling for $\fl{H}$ are key concepts for describing the structure of $\fl{H}$ and its action on $G$. 
\begin{definition}
\label{defn:scaling_subgroup}
A closed, non-compact, $\fl{H}$-stable subgroup $H\leq G$ that has a compact, relatively open group $U\leq H$ such that 
\begin{align*}
\alpha\in\fl{H} \implies& \text{ either }\alpha(U)\geq U\text{ or }\alpha(U)\leq U\\
\text{ and }H =& \ \bigcup\left\{ \alpha(U)\mid \alpha\in\fl{H}\right\}
 \end{align*}
 is a \emph{scaling subgroup} for $\fl{H}$.
\end{definition}
Scaling subgroups are distinct from scale groups, see~\cite{GWil_scalegps} and \S\ref{sec:tree_rep} above. A scale group is a closed group of isometries of a regular tree $\tree_{s+1}$ that fixes an end of the tree and is transitive on the vertices, whereas saying that a subgroup is scaling refers to its behaviour under the action of a flat group $\fl{H}$. Scale groups are thus a class of concrete \tdlc~groups, while scaling subgroups are contained in general abstract \tdlc~groups. The connection is that every scaling subgroup has a representation with compact kernel as a scale group, see~\cite[Theorem 4.1]{ContractionB}. An understanding of scale groups will contribute to knowledge about scaling groups and, through the results shown in these notes, to knowledge of flat groups of automorphisms.

The condition that either $\alpha(U)\geq U$ or $\alpha(U)\leq U$ implies that $U$ is tidy for $\fl{H}|_H$ and hence that $\fl{H}|_H$ is flat. This condition also implies that the set $\left\{ \alpha(U) \mid \alpha\in\fl{H}\right\}$ is totally ordered by inclusion. Since $\alpha(U)$ is compact and open for each $\alpha\in\fl{H}$, the index $[\beta(U) : \alpha(U)]$ is a positive integer for all $\alpha,\beta\in\fl{H}$ with $\alpha(U)\leq\beta(U)$ and so there is a minimum such index bigger than $1$. (Note that non-compactness of $H$ implies that there is $\alpha\in\fl{H}$ such that $\alpha(U)>U$.) Choose $\beta_*\in \fl{H}$ such that $[\beta_*(U):U]$ is equal to this minimum index. Then 
$$
\left\{ \alpha(U) \mid \alpha\in\fl{H}\right\} = \left\{\beta_*^n(U)\mid n\in\mathbb{Z}\right\}.
$$ 
Hence, $N := \bigcap_{\alpha\in\fl{H}} \alpha(U) = \bigcap_{n\in\mathbb{Z}} \beta_*^n(U)$ is a compact, $\fl{H}$-stable subgroup of $H$ and, since $\bigcup_{n\in\mathbb{Z}}\beta_*^n(U) = H$, we have that $H = \con(\alpha|_H/N)$ for every $\alpha\in\fl{H}$ such that $\alpha(U)<U$.

The modular function on the restriction of $\fl{H}$ to $H$ is a homomorphism $\fl{H} \to (\mathbb{Q}^+,\times)$ given by
$$
\triangle_{H}(\alpha|_H) = \begin{cases}
[\alpha(U):U], &\text{ if }\alpha(U)\geq U\\
[U:\alpha(U)]^{-1}, & \text{ if } \alpha(U)\leq U
\end{cases}
$$ 
and is a power of $[\beta_*(U):U]$ for every $\alpha\in\fl{H}$.

The preceding discussion establishes all but part~\ref{prop:scale&homomorphism5}.~of the following. 
\begin{proposition}
\label{prop:scale&homomorphism}
Let $\fl{H}\leq\Aut(G)$ be flat and suppose that $H\leq G$ is a scaling subgroup for $\fl{H}$. With $U$ as in Definition~\ref{defn:scaling_subgroup}, set 
$$
s = \min\left\{[\beta(U):U] \mid \beta\in\fl{H}\mbox{ with }[\beta(U):U]>1\right\}
$$ 
and choose $\beta_*\in \fl{H}$ that achieves this minimum. Then 
\begin{enumerate}
\item \label{prop:scale&homomorphism1}
for every $\alpha\in\fl{H}$ there is $\roo(\alpha)\in\mathbb{Z}$ such that $\alpha(U) = \beta_*^{\roo(\alpha)}(U)$;  
\item \label{prop:scale&homomorphism4}
$N := \bigcap_{\alpha\in\fl{H}} \alpha(U)$ is a compact $\fl{H}$-stable subgroup of $H$, and is the largest such subgroup;
\item \label{prop:scale&homomorphism2}
for this $N$, we have $H = \con(\alpha|_{H}/N) = \con(\alpha|_H)N$ for each $\alpha\in\fl{H}$ with $\roo(\alpha)<0$, and in particular for $\alpha=\beta_*^{-1}$;
\item \label{prop:scale&homomorphism3}
the map $\roo:\fl{H} \to \mathbb{Z}$ is a surjective group homomorphism from $\fl{H}$ to $(\mathbb{Z},+)$ and $\Delta(\alpha|_H) = s^{\roo(\alpha)}$ for every $\alpha\in\fl{H}$; and
\item \label{prop:scale&homomorphism5}
if $V\leq H$ is tidy for $\fl{H}|_H$, then $\widehat{V}_{\fl{H}} = \bigcup\left\{ \alpha(V)\mid \alpha\in\fl{H}\right\}$ is an open subgroup with finite index in $H$ and $H = N\widehat{V}_{\fl{H}}$. 
\endproof
\end{enumerate} 
\end{proposition}
\begin{proof}
\ref{prop:scale&homomorphism5}. If $V$ is tidy for $\fl{H}$, then for every $\alpha\in\fl{H}$ we have $\alpha(V)\leq V$ or $\alpha(V)\geq V$. Then $\widehat{V}_{\fl{H}}$ is an open, $\fl{H}$-stable subgroup of $H$, and hence $N\Upp{V}{\fl{H}}$ is an open, $\fl{H}$-stable neighbourhood of $N$. Therefore $N\widehat{V}_{\fl{H}}$ contains $\beta_*^m(U)$ for some $m$ and, since it is $\fl{H}$-stable, must equal $H$. The surjection 
$$
x(N\cap\widehat{V}_{\fl{H}})\mapsto x\widehat{V}_{\fl{H}} :   N/(N\cap\widehat{V}_{\fl{H}}) \to N\widehat{V}_{\fl{H}}/\widehat{V}_{\fl{H}} = H/\widehat{V}_{\fl{H}}
$$ 
then implies that $\widehat{V}_{\fl{H}}$ has finite index in $H$ because $N\cap\Upp{V}{\fl{H}}$ is an open subgroup of the compact group $N$. 
\end{proof}

Exercise~\ref{exer:scaling} gives an example of a scaling group $H$ and $V\in\COS(H)$ that is tidy for $\fl{H}$ such that $\widehat{V}_{\fl{H}}$ is a proper subgroup of $H$. 

Proposition~\ref{prop:scale&homomorphism} shows that each scaling subgroup for $\fl{H}$ determines a root of $\fl{H}$, defined as follows.
\begin{definition}
\label{defn:roots}
A \emph{root} of a flat group $\fl{H}$ is a surjective homomorphism $\roo:\fl{H}\to \mathbb{Z}$ such that there are $H\leq G$ scaling for $\fl{H}$ and $s_\roo\in \mathbb{N}$, greater than $1$, such that 
$$
\triangle_{H}(\alpha|_H) = s_\roo^{\roo(\alpha)} \mbox{ for every }\alpha\in\fl{H}.
$$ 
We shall say that \emph{$\roo$ is the root associated to $H$}, and that \emph{$H$ is a scaling subgroup associated with $\roo$}. \\
The set of roots of $\fl{H}$ is denoted by $\red(\fl{H})$.
\end{definition}

Scaling subgroups for $\fl{H}$ associated with a given root $\roo$ need not be unique, as seen in Exercise~\ref{exer:scaling}. Part~\ref{prop:scale&homomorphism2}.~of Proposition~\ref{prop:scale&homomorphism} shows that scaling subgroups differ in the compact subgroup $N$. The following will be useful.
\begin{lemma}
\label{lem:subgroup_of_scaling}
Suppose that $H$ is a scaling subgroup for $\fl{H}$ associated with $\roo$ and that $L\leq H$ is $\fl{H}$-stable and non-compact. Then $L$ is a scaling subgroup associated with $\roo$.
\end{lemma}
\begin{proof}
If $U$ is the compact open subgroup satisfying the conditions of Definition~\ref{defn:scaling_subgroup} showing that $H$ is scaling and associated with $\roo$, then $U\cap L$ does the same for $L$. Since $L$ is also assumed to be non-compact, it follows that $L$ is scaling.
\end{proof}

The next result shows for each $U$ tidy for $\fl{H}$ and $\roo\in\red(\fl{H})$, that there is a scaling subgroup whose intersection with $U$ is maximal. 

\begin{proposition}
	\label{prop:roots}
	Let $\roo:\fl{H}\to\mathbb{Z}$ be a surjective homomorphism, suppose that $U\in \COS(G)$ is tidy for $\fl{H}$, and define 
	$$
	U_\roo := \bigcap\left\{\alpha(U) \mid \roo(\alpha)\geq0\right\}.
	$$ 
	Then $U_\roo\leq U$ and, for $\alpha,\beta\in\fl{H}$ we have: $\alpha(U_\roo) = \beta(U_\roo)$ if and only if  $\roo(\alpha)=\roo(\beta)$; and $\alpha(U_\roo)\geq U_\roo$ if and only if $\roo(\alpha)\geq0$.

	\begin{enumerate}
		\item 	\label{prop:roots1}
		There is $\alpha\in\fl{H}$ such that $\alpha(U_\roo)<U_\roo$ if and only if $\roo\in\red(\fl{H})$. In this case, the $\fl{H}$-stable set 
		$$
		\widehat{U}_\roo := \bigcup\left\{ \alpha(U_\roo) \mid \alpha\in\fl{H}\right\}
		$$ 
		is a scaling subgroup for $\fl{H}$ associated with $\roo$.
		\item 	\label{prop:roots2}  
		If  $\roo\in\red(\fl{H})$ and $H$ is a scaling subgroup associated with $\roo$, then: $H\cap U\leq U_\roo$; $H\cap U$ is tidy for $\fl{H}|_H$; and $H\cap \widehat{U}_\roo$ is a scaling subgroup with finite index in $H$. 
		\item 	\label{prop:roots3}
		If $\roo\not\in\red(\fl{H})$, then $\widehat{U}_\roo = \Uz{U}{\fl{H}} = \bigcap\left\{\alpha(U)\mid \alpha\in\fl{H}\right\}$.
	\end{enumerate}
\end{proposition}
\begin{proof}
	Since the identity automorphism $\iota$ satisfies $\roo(\iota)=0$, we have that $U_\roo\leq U$. If $\beta\in\fl{H}$ and $\roo(\beta)\geq0$, then $\beta(U_\roo) = \bigcap\left\{\beta\alpha(U) \mid \roo(\alpha)\geq0\right\}$ and the latter contains $U_\roo$ because $\roo(\beta\alpha) = \roo(\alpha)+\roo(\beta)\geq \roo(\alpha)$. It follows in particular that $\alpha(U_\roo)=U_\roo$ if $\roo(\alpha)=0$ and, since $\roo$ is a homomorphism, that $\alpha(U_\roo) = \beta(U_\roo)$ if and only if $\roo(\alpha)=\roo(\beta)$.
	 
 \ref{prop:roots1}.  Since $\alpha(U_\roo)\leq U_\roo$ if and only if $\roo(\alpha)\leq 0$, if $\alpha(U_\roo)< U_\roo$ for some $\alpha\in\fl{H}$, then $\bigcup\left\{ \alpha(U_\roo) \mid \alpha\in\fl{H}\right\}$ is a nested union of subgroups of $G$ and is itself a subgroup. To see that it is scaling for $\fl{H}$, it must be shown that it is closed and, for this, it suffices, by~\cite[Theorem II.5.9]{HandR} or \cite[Ch. III, p.~221]{BourbakiTop}, to show that $\widehat{U}_\roo\cap U = U_\roo$. Consider $x\in \widehat{U}_\roo\cap U$. Then $x\in \alpha(U_\roo)$ for some $\alpha\in\fl{H}$ and for any $\beta\in\fl{H}$ with $\roo(\beta)\leq0$ we have $\beta^{n}(x)\in \alpha(U_\roo)$ for every $n\geq0$. Since $\alpha(U_\roo)$ is compact and $U$ is tidy for $\beta$, it follows by Proposition~\ref{prop:tidy_criteria} that $x\in \Um{U}{\beta}$. Since this holds for every $\beta$ with $\roo(\beta)\geq0$, it follows that $x\in U_\roo$ as required. 
 	
	 \ref{prop:roots2}. Suppose that $\roo\in\red(\fl{H})$ and $H$ is an associated scaling subgroup. Consider $x\in H\cap U$. If $\roo(\alpha)\leq 0$, then $\{\alpha^n(x)\}_{n\geq0}$ is precompact because $x\in H$, and so, since $x\in U$, Proposition~\ref{prop:tidy_criteria} implies that  $\{\alpha^n(x)\}_{n\geq0}\subseteq U$. Hence $\alpha(H\cap U)\leq H\cap U$ if $\roo(\alpha)\leq0$ and we have shown that $H\cap U\leq U_\roo$ and is tidy for $\fl{H}|_H$.
	 
	  That $H\cap \widehat{U}_\roo$ is a scaling subgroup follows from $H\cap U$ being tidy for $\fl{H}|_H$, and that it has finite index in $H$ then follows from Proposition~\ref{prop:scale&homomorphism}.\ref{prop:scale&homomorphism5}.

 \ref{prop:roots3}. Suppose that $\roo\not\in\red(\fl{H})$.	Then $\alpha(U_\roo) = U_\roo$ for every $\alpha\in\fl{H}$ by part~\ref{prop:roots1}. Hence
	$$
	U_\roo = \bigcap\left\{\alpha(U) \mid \alpha\in\fl{H}\right\} = \Uz{U}{\fl{H}} = \widehat{U}_\roo.
	$$ 
\end{proof}

\begin{corollary}
	\label{cor:roots}	
	If $\fl{H}$ is a uniscalar flat group, that is, $\fl{H}$ stabilizes subgroups tidy for it, then $\red(\fl{H}) = \emptyset$. 
\end{corollary}
\begin{proof}
If $\alpha(U) = U$ for every $\alpha\in \fl{H}$ and $U$ tidy for $\fl{H}$, then $U_\roo = U$ for every $\roo:\fl
H\to \mathbb{Z}$ and $\alpha({U}_\roo)= U_\roo$ for every $\alpha\in \fl{H}$.  Hence no homomorphism $\roo$ belongs to $\red(\fl{H})$ by Proposition~\ref{prop:roots}.\ref{prop:roots1}.  
\end{proof}

It is convenient to state the next result in terms of the  sub-semigroup $\fl{A}_\roo = \left\{\alpha\in\fl{H} \mid \roo(\alpha)<0\right\}$ of $\fl{H}$. Recall that $\con(\fl{A}_\roo/\nub(\fl{H}))$ is defined in~\ref{defn:relative_contraction_semigroup} to be the intersection of all subgroups $\con(\alpha/\nub(\fl{H}))$ with $\alpha\in\fl{A}_\roo$.
\begin{proposition}
\label{prop:con_Aroo} 
\begin{enumerate}
	\item \label{prop:con_Aroo1} 
	Let $\roo\in \red(\fl{H})$ and $U\in\COS(G)$ be tidy for $\fl{H}$. Then 
$\widehat{U}_\roo =  \con(\fl{A}_\roo/U_{\fl{H}0})$. 
	\item \label{prop:con_Aroo2} 
$ \bigcap\left\{\con(\fl{A}_\roo/U_{\fl{H}0})\mid U\text{ tidy for }\fl{H}\right\} = \con(\fl{A}_\roo/\nub(\fl{H}))$ 
and is a scaling subgroup for $\fl{H}$ associated with $\roo$. 
\item \label{prop:con_Aroo3} 
Every compact, $\fl{H}$-stable group $L\leq G$ normalizes $\con(\fl{A}_\roo/\nub(\fl{H}))$ and $L\con(\fl{A}_\roo/\nub(\fl{H}))$ is a scaling subgroup associated with $\roo$. 
	\item \label{prop:con_Aroo4} 
	If $H$ is scaling for $\fl{H}$, contains $\nub(\fl{H})$, and is associated with $\roo$, then
$$
H = N\con(\fl{A}_\roo/\nub(\fl{H}))
$$
where $N$ is the largest $\fl{H}$-stable subgroup of $H$.
\end{enumerate}
\end{proposition}
\begin{proof}
\ref{prop:con_Aroo1}. Let $U$ be tidy for $\fl{H}$. Then $\widehat{U}_\roo = \con(\alpha|_{\widehat{U}_\roo}/U_{\fl{H}0})$ for every $\alpha\in \fl{A}_\roo$ by Propositions~\ref{prop:roots}.\ref{prop:roots1} and~\ref{prop:scale&homomorphism}.\ref{prop:scale&homomorphism2} whence it follows that $\widehat{U}_\roo \leq \con(\alpha/U_{\fl{H}0})$ for every $\alpha\in \fl{A}_\roo$. Hence $\widehat{U}_\roo \leq \con(\fl{A}_\roo/U_{\fl{H}0})$. On the other hand, for each $x\in U\cap \con(\fl{A}_\roo/U_{\fl{H}0})$ and $\alpha\in\fl{A}_\roo\cup\ker\roo$, we have that $\{\alpha^n(x)\}_{n\in\mathbb{N}}$ is precompact because $x\in \con(\alpha/U_{\fl{H}0})$ and so $\{\alpha^n(x)\}_{n\in\mathbb{N}}\subset U$, by Proposition~\ref{prop:tidy_criteria}. In particular,
$$
U\cap \con(\fl{A}_\roo/U_{\fl{H}0}) \leq  \alpha^{-1}(U\cap \con(\fl{A}_\roo/U_{\fl{H}0}))
$$
for every $\alpha\in\fl{A}_\roo\cup\ker\roo$. Therefore $U\cap \con(\fl{A}_\roo/U_{\fl{H}0}) \leq U_\roo$ and $\widehat{U}_\roo = \con(\fl{A}_\roo/U_{\fl{H}0})$.

\ref{prop:con_Aroo2}. Since $\nub(\fl{H})\leq U_{\fl{H}0}$, we have that $\con(\fl{A}_\roo/\nub(\fl{H})) \leq \con(\fl{A}_\roo/U_{\fl{H}0})$ for every $U$ tidy for $\fl{H}$. Hence
$$
\con(\fl{A}_\roo/\nub(\fl{H})) \leq \bigcap\left\{\con(\fl{A}_\roo/U_{\fl{H}0})\mid U\text{ tidy for }\fl{H}\right\}.
$$
The reverse inclusion holds because $\nub(\fl{H})$ is the intersection of all the subgroups tidy for $\fl{H}$. Hence for every open $\mathcal{O}\supseteq \nub(\fl{H})$ there is $U$ tidy for $\fl{H}$ such that $U\subseteq \mathcal{O}$ and so, if $x\in\con(\fl{A}_\roo/U_{\fl{H}0})$ for every $U$ tidy for $\fl{H}$, then $x\in\con(\fl{A}_\roo/\nub(\fl{H}))$. 

By Lemma~\ref{lem:subgroup_of_scaling}, to complete the proof that $\con(\fl{A}_\roo/\nub(\fl{H}))$ is a scaling subgroup associated with $\roo$, it suffices to show that it is not compact. For this, choose $\alpha\in\fl{A}_\roo$ and $V$ tidy for $\fl{H}$ and observe that, for each $U\leq V$ tidy for $\fl{H}$, $\con(\fl{A}_\roo/U_{\fl{H}0})\cap \left(\alpha^{-n-1}(V)\setminus \alpha^{-n}(V)\right)$ is non-empty and compact for every $n\in\mathbb{Z}$ and 
$$
\con(\fl{A}_\roo/U_{\fl{H}0}) = \bigsqcup_{n\in\mathbb{Z}} \con(\fl{A}_\roo/U_{\fl{H}0})\cap \left(\alpha^{-n-1}(V)\setminus \alpha^{-n}(V)\right).
$$
Then, for each $n\in\mathbb{Z}$,  $\left\{\con(\fl{A}_\roo/U_{\fl{H}0})\cap \left(\alpha^{-n-1}(V)\setminus \alpha^{-n}(V)\right)\right\}_{U\leq V}$ is a set of non-empty conpact sets directed by the inclusion relation and so
$$
\bigcap\left\{\con(\fl{A}_\roo/U_{\fl{H}0})\cap \left(\alpha^{-n-1}(V)\setminus \alpha^{-n}(V)\right)\mid U\leq V\text{ tidy for }\fl{H}\right\} \ne \emptyset.
$$
Therefore $\con(\fl{A}_\roo/\nub(\fl{H}))\cap \left(\alpha^{-n-1}(V)\setminus \alpha^{-n}(V)\right)$ is not empty for every $n\in\mathbb{Z}$ and $\con(\fl{A}_\roo/\nub(\fl{H}))$ is not compact. 

\ref{prop:con_Aroo3}. It is shown in Proposition~\ref{prop:Nnormal_core_of_tidy}.\ref{prop:Nnormal_core_of_tidy2} that $L$ normalizes $\nub(\fl{H})$ and, since $L$ is compact, there is a base of neighbourhoods of $\id_G$ that are also normalized by $L$. Hence, if $\alpha^n(x)\to \nub(\fl{H})$ as $n\to\infty$, then so does $\alpha^n(lxl^{-1})$ for every $l\in L$ because $L$ is $\fl{H}$-stable. Hence $L$ normalizes $\con(\fl{A}_\roo/\nub(\fl{H}))$ and $L\con(\fl{A}_\roo/\nub(\fl{H}))$ is a group. 
This group is closed because $L$ is compact, non-compact because $\con(\fl{A}_\roo/\nub(\fl{H}))$ is, and is $\fl{H}$-stable because $L$ and $\con(\fl{A}_\roo/\nub(\fl{H}))$ are. 

To see that $L\con(\fl{A}_\roo/\nub(\fl{H}))$ satisfies the remaining criterion for it to be scaling, let $U$ be a compact, open subgroup of $\con(\fl{A}_\roo/\nub(\fl{H}))$ such that $\bigcup\{\alpha(U)\mid\alpha\in\fl{H}\} = \con(\fl{A}_\roo/\nub(\fl{H}))$ and $\alpha(U)$ is comparable with $U$ for every $\alpha\in\fl{H}$. Set $U' = \bigcap\{lUl^{-1}\mid l\in L\}$. Then $U'$ is a compact open subgroup of $\con(\fl{A}_\roo/\nub(\fl{H}))$ that contains $\nub(\fl{H})$ because $\nub(\fl{H})$ is normalized by $L$. Hence there are $n\in\mathbb{N}$ and $\beta\in\fl{A}_\roo$ such that $\beta^n(U)\leq U'$ and it follows that $\bigcup\{\alpha(U')\mid\alpha\in\fl{H}\} = \con(\fl{A}_\roo/\nub(\fl{H}))$. Furthermore, for $\alpha\in\fl{H}$: $\alpha(U')\geq U'$ if and only if $\alpha(U)\geq U$; and $\alpha(U')\leq U'$ if and only if $\alpha(U)\leq U$. Therefore the group $LU'$ is compact and open in $L\con(\fl{A}_\roo/\nub(\fl{H}))$ and
\begin{align*}
\alpha\in\fl{H} &\implies \text{ either }\alpha(LU')\geq LU'\text{ or }\alpha(LU')\leq LU'\\
\text{ and }L\con(\fl{A}_\roo/\nub(\fl{H})) &= \ \bigcup\left\{ \alpha(LU')\mid \alpha\in\fl{H}\right\}
\end{align*}
and so $L\con(\fl{A}_\roo/\nub(\fl{H}))$ is scaling for $\fl{H}$ and associated with $\roo$.

\ref{prop:con_Aroo4}. Suppose that $H\leq G$ is scaling for $\fl{H}$, associated with $\roo$ and contains $\nub(\fl{H})$. Choose $\alpha\in \fl{A}_\roo$. Then $H = \con(\alpha|_H/N)$ because $H$ is scaling and $\nub(\fl{H})\leq N$. It follows that $\con(\fl{A}_\roo/\nub(\fl{H}))\leq H$ and hence that $N\con(\fl{A}_\roo/\nub(\fl{H}))\leq H$. Since $H$ and $N\con(\fl{A}_\roo/\nub(\fl{H}))$ are both scaling, there are compact subgroups $U\leq H$ and $V\leq N\con(\fl{A}_\roo/\nub(\fl{H}))$ satisfying the conditions of Definition~\ref{defn:scaling_subgroup}. Since $\bigcap_{\alpha\in\fl{H}} \alpha(U)$ and $\bigcap_{\alpha\in\fl{H}} \alpha(V)$ both equal $N$, there is $\alpha\in\fl{H}$ such that $\alpha(U)\leq V$. Hence  $\bigcup_{\alpha\in\fl{H}} \alpha(V)$ is equal to both $H$ and $N\con(\fl{A}_\roo/\nub(\fl{H}))$. 
\end{proof}

The groups $\widehat{U}_\roo$, with $U$ tidy for $\fl{H}$, and $\con(\fl{A}_\roo/\nub(\fl{H}))$ described in Propositions~\ref{prop:roots} and~\ref{prop:con_Aroo} will be called respectively the \emph{shrinking} and \emph{contraction} subgroups associated with $\roo$. 

The following is an immediate consequence of part~\ref{prop:con_Aroo4}. of Proposition~\ref{prop:con_Aroo}.
\begin{corollary}
\label{cor:con_Aroo}
Suppose that $\nub(\fl{H})$ is trivial and let $\roo\in\red(\fl{H})$. Then $\con(\fl{A}_\roo)$ is the smallest scaling subgroup of $G$ associated with $\roo$.
\end{corollary}

The following examples illustrate the concepts of roots of flat groups and scaling subgroups and the results just shown.
\begin{example}
\label{examp:scaling_groups}
\begin{enumerate}
	\item 	Let $G$ be a \tdlc~group and $\fl{H} = \langle\alpha\rangle$ with $s(\alpha)>1$. Let $U$ be tidy for $\alpha$. Then $\Upp{U}{\alpha}$ and $\con(\alpha^{-1})^-$ are both scaling for $\fl{H}$ and are associated with the root $\roo:\alpha^n\mapsto n$. In the first case, $\Uz{U}{\alpha}$ is the maximal $\fl{H}$-stable subgroup and in the second this subgroup is $\nub(\alpha)$. If $V$ is also tidy for $\alpha$ with $V< U$,  then $\Upp{U}{\alpha}<\Upp{V}{\alpha}$ if and only if $\Uz{U}{\alpha}<\Uz{V}{\alpha}$. Should $s(\alpha^{-1})$ be greater than $1$, then $-\roo:\alpha^n\mapsto -n$ is a root associated with $\Umm{U}{\alpha}$ and $\con(\alpha)^-$.
	\item Let $G=\Isom(\tree_{p+1})$ and $\fl{H} = \langle \alpha_x\rangle$ with $x$ a translation of $\tree_{p+1}$ towards the attracting end $\eta$. Then $\Isom(\tree_{p+1})_{[\eta]} = \left\{g\in G\mid g\text{ fixes a ray in }\eta\right\}$ is scaling for $\fl{H}$ and is associated with the root $\roo : \alpha_x^n\mapsto n$. Choosing a vertex $v$ on the axis of $x$, the fixator, $U$, of the ray $[v,\eta)$ is compact and open in $\Isom(\tree_{p+1})_{[\eta]}$ and satisfies $\bigcup_{n\in\mathbb{Z}} \alpha_x(U) = \Isom(\tree_{p+1})_{[\eta]}$ and $\bigcap_{n\in\mathbb{Z}} \alpha_x(U) = \nub(\alpha_x)$, which is the fixator of the axis of $x$. We have that $s(\alpha_x) = p^d$, where $d$ is the distance that $x$ translates its axis towards $\eta$. 
	
	However, $\Isom(\tree_{p+1})_{[\eta]}$ is not the smallest scaling subgroup associated with $\roo$ and, indeed, there is no smallest such subgroup. The tree $\tree_{p+1}$ is the Bruhat-Tits building for the groups $PSL_s(\mathbb{Q}_p)$ and $PSL_2(\mathbb{F}_p(\!(X)\!))$. Both groups thus embed as closed subgroups of $G$ and the matrices $\left(\begin{array}{cc}
p^{-1} & 0\\ 0 & p
	\end{array}\right)$, respectively $\left(\begin{array}{cc}
	X^{-1} & 0\\ 0 & X
	\end{array}\right)$, embed as a translation $x$. The groups of upper triangular matrices $\left\{\left(\begin{array}{cc}
	1 & a\\ 0 & 1
	\end{array}\right)\mid a\in \mathbb{Q}_p\right\}$, respectively $\left\{\left(\begin{array}{cc}
	1 & a\\ 0 & 1
	\end{array}\right)\mid a\in \mathbb{F}_p(\!(X)\!)\right\}$, embed as minimal scaling groups for $\alpha_x$ but they are distinct. In the first case the scaling group has no torsion but in the second case it is a torsion group. 
	\item Let $G$ be $SL_n(\mathbb{Q}_p)$ and let $\fl{H}$ be the group of inner automorphisms given by the group of diagonal matrices $D = \{(d_{ii})_{i=1}^n\mid d_{ii}\in\mathbb{Q}_p\setminus\{0\}\}$. Then $\fl{H}$ is flat and $U = \left\{(a_{ij})_{i,j=1}^n\in SL_n(\mathbb{Z}_p)\mid a_{ij}\in p\mathbb{Z}_p\text{ if }i>j\right\}$ is tidy for $\fl{H}$. Also, $\nub(\fl{H})$ is trivial. For each pair $i\ne j$, $i,j\in\{1,\dots,n\}$, define a homomorphism $\roo_{ij}:\fl{H}\to \mathbb{Z}$ by $\roo(d_{ii}) = |d_{ii}/d_{jj}|_p$. Then
	$$
	\red(\fl{H}) = \left\{\roo_{ij}\mid i\ne j\in\{1,\dots,n\}\right\}.
	$$
	The scaling group $U_{\roo_{ij}}$ is equal to the group of triangular matrices 
	$$
	\left\{(a_{ij})_{i,j=1}^n\in SL_n(\mathbb{Z}_p)\mid a_{kk}\in \mathbb{Z}_p,\ a_{ij}\in\mathbb{Q}_p\text{ and }a_{rs}=0\text{ otherwise}\right\},
	$$
	whereas the scaling group $\con(\fl{A}_{\roo_{ij}})$ is equal to the root subgroup
	$$
	\left\{(a_{ij})_{i,j=1}^n\in SL_n(\mathbb{Z}_p)\mid a_{kk}=1,\ a_{ij}\in\mathbb{Q}_p\text{ and }a_{rs}=0\text{ otherwise}\right\}.
	$$
\end{enumerate}
\end{example}

\subsection{Main Theorems: The structure of flat groups and of subgroups tidy for them}
\label{sec:main}


These notes describe the structure of flat groups of automorphisms of a \tdlc~group and their corresponding tidy subgroups. Briefly, it is shown that there are sufficiently many roots to separate $\fl{H}/\fl{H}_u$, and contraction and shrinking subgroups associated with an element $\alpha\in\fl{H}$ factor as products of the corresponding groups associated with roots that are positive on~$\alpha$. More complete statements are given in the following theorems. The first describes the structure of the flat group $\fl{H}$ of automorphisms of $G$.
\begin{thmintro}
\label{thm:flat_group_structureA}
Let $\fl{H}\leq \Aut(G)$ be flat and $\red(\fl{H})$ be its set of roots. Then the map $R : \fl{H} \to \mathbb{Z}^{\red(\fl{H})}$ defined by
$$
R(\alpha) = \{\roo(\alpha)\}_{\roo\in\red(\fl{H})}
$$
is a homomorphism with kernel equal to $\fl{H}_u$ and range contained in $\bigoplus_{\roo\in{\red(\fl{H})}} \mathbb{Z}$.  In particular, 
\begin{enumerate}
\item \label{thm:flat_group_structureA1}
$\{\roo\in \red(\fl{H}) \mid \roo(\alpha)\ne0\}$ is finite for each $\alpha\in\fl{H}$ and 
\item \label{thm:flat_group_structureA2}
$\fl{H}/\fl{H}_u$ is a free abelian group.
\end{enumerate}
The set $\red(\fl{H})$ empty if and only if $\fl{H}$ is uniscalar. 
\end{thmintro}

The next theorem describes the structure of compact open subgroups of $G$ tidy for $\fl{H}$. 
\begin{thmintro}
\label{thm:flat_group_structureB}
Let $\fl{H}$ be flat with $U$ tidy for $\fl{H}$, $\roo\in\red(\fl{H})$ and suppose that $\red(\fl{H})$ is finite. Let $\con(\fl{A}_\roo/\nub(\fl{H}))$, ${U}_\roo$ and $\widehat{U}_\roo$ be  as in Propositions~\ref{prop:con_Aroo} and~\ref{prop:roots}. Then the following hold.
\begin{enumerate}
\item \label{thm:flat_group_structureB1}
$\widehat{U}_\roo = \Uz{U}{\fl{H}} \con(\fl{A}_\roo/\nub(\fl{H}))$ and is a closed $\fl{H}$-stable subgroup of $G$. The restrictions of $\fl{H}$ to $\widehat{U}_\roo$ and $ \con(\fl{A}_\roo/\nub(\fl{H}))$ are flat with ${U}_\roo$ and $U\cap \con(\fl{A}_\roo/\nub(\fl{H}))$ tidy for the respective restrictions. There are $\beta_\roo\in\fl{H}$ and an integer $s_\roo>1$ such that 
\begin{equation}
\label{eq:modular}
\alpha({U}_\roo) = \beta_\roo^{\roo(\alpha)}({U}_\roo) \mbox{ and }\Delta(\alpha|_{\widehat{U}_\roo}) = s_\roo^{\roo(\alpha)} \mbox{ for every }\alpha\in\fl{H}.
\end{equation}
\item \label{thm:flat_group_structureB2}
There is a total order on $\red(\fl{H})$, independent of $U$, such that
$$
U = \prod_{\roo\in\red(\fl{H})} {U}_\roo = \Uz{U}{\fl{H}}\prod_{\roo\in\red(\fl{H})} \left(U\cap  \con(\fl{A}_\roo/\nub(\fl{H}))\right).
$$
\item \label{thm:flat_group_structureB3}
For each $\alpha\in\fl{H}$,
$$
\con(\alpha/\nub(\fl{H})) = \prod_{\roo(\alpha^{-1})>0} \con(\fl{A}_\roo/\nub(\fl{H}))
$$
with $\{\roo\in\red(\fl{H})\mid \roo(\alpha^{-1})>0\}$ ordered as a subset  of $\red(\fl{H})$.
\end{enumerate}
\end{thmintro}

Theorem~\ref{thm:flat_group_structureB}.\ref{thm:flat_group_structureB1} summarizes results from \S\ref{sec:irroots}. Proofs of the remaining assertions in the two theorems are given in the next four sections.  

The proof that there are sufficiently many roots to separate $\fl{H}/\fl{H}_u$ and hence that $\fl{H}/\fl{H}_u$ is free abelian, as stated in Theorem~\ref{thm:flat_group_structureA}, is done by showing that each shrinking subgroup $\Umm{U}{\alpha}$ contains scaling subgroups whose product is co-compact in $\Umm{U}{\alpha}$. The roots associated with these scaling subgroups are non-zero at $\alpha$, and each non-uniscalar element in $\fl{H}$ thus separated from $0$ by at least one root. In \S\ref{sec:factor-tidy}, repeated application of the tidiness above condition for elements of $\fl{H}$ other than $\alpha$ is used to reduce $\Umm{U}{\alpha}$ to a product of smaller subgroups. A difficulty encountered is that $\Umm{U}{\alpha}$ might not be $\fl{H}$-invariant and it is shown in \S\ref{sec:nub_contraction} that this difficulty may be avoided by instead working with $\con(\alpha/\nub(\fl{H}))$. It may already be seen at this point that $\fl{H}/\fl{H}_u$ is abelian by showing that commutators in $\fl{H}$ are uniscalar. That the repeated reduction of $\Umm{U}{\alpha}$ into smaller factors eventually terminates with scaling subgroups is seen in \S\ref{sec:reduced1} by showing that each reduction leads to an integer factoring of the scale. These arguments yield the information needed to prove Theorem~\ref{thm:flat_group_structureB}. Consolidating that information involves defining the order on $\red(\fl{H})$ required for the statement of the theorem, which is done in \S\ref{sec:reduced2}.


\subsection{Tidiness for multiple automorphisms in a flat group}
\label{sec:factor-tidy}

In this section it is seen how the factoring of $U$ seen in Theorem~\ref{thm:flat_group_structureB} emerges by applying the tidiness conditions simultaneously for multiple elements of $\fl{H}$. It will be shown that:
\begin{itemize}
\item given $\alpha_1$, \dots, $\alpha_n$ in $\fl{H}$, tidiness above implies that~$U$ is the product of subgroups each of which is expanded or shrunk by each~$\alpha_i$; 
\item the shrinking subgroup, $\Umm{U}{\alpha}$, also factors according to the action of each $\beta\in\fl{H}$ but these factors, and $\Umm{U}{\alpha}$ itself, need not be $\beta$-invariant. 
\end{itemize}

The starting point is the following consequence of joint tidiness above.
\begin{lemma}
\label{lem:factor_U}
Suppose that $\fl{H}\leq\Aut(G)$ is flat and let $U$ be tidy for $\fl{H}$. Then, for any $\alpha,\beta\in\fl{H}$, we have that
$\Um{U}{\alpha} = (\Um{U}{\alpha}\cap \Um{U}{\beta}) (\Um{U}{\alpha}\cap \Up{U}{\beta})$.
\end{lemma}
\begin{proof}
For $K\geq0$, set $U(\alpha,K) = \bigcap_{k=0}^K \alpha^{-k}(U)$. Then $U(\alpha,K)$ is tidy for $\fl{H}$ by Theorem~\ref{thm:tidy_intersection} and Proposition~\ref{prop:invariance_of_tidiness}, and $\bigcap_{K\geq0} U(\alpha,K) = \Um{U}{\alpha}$.

Consider $x\in \Um{U}{\alpha}$. For each $K$ we have that $x\in U(\alpha,K)$ and hence that $x = x_-x_+$ with $x_\pm \in \Upm{U(\alpha,K)}{\beta}$. Hence
$$
X^{(K)} := \left\{ (x_-,x_+) \in \Um{U(\alpha,K)}{\beta}\times \Up{U(\alpha,K)}{\beta} \mid x_-x_+ = x\right\}
$$
is not empty for each $K$. This set is also compact, because it is a closed subset of a compact set, and $X^{(K+1)}\subseteq X^{(K)}$. Hence $\bigcap_{K\geq0} X^{(K)} \ne\emptyset$. Choose $(x_-,x_+)$ in this intersection. Then $x = x_-x_+$ and 
$$
x_\pm \in \bigcap_{K\geq0} \Upm{U(\alpha,K)}{\beta}.
$$ 
An application of Proposition~\ref{prop:tidy_criteria} shows that this intersection is contained in $\Um{U}{\alpha}\cap \Upm{U}{\beta}$, thus proving the claim.
\end{proof}

The same argument shows that $\Up{U}{\alpha}$ factors in a corresponding way and so, since $U$ itself factors as $U = \Um{U}{\alpha}\Up{U}{\alpha}$, Lemma~\ref{lem:factor_U} expresses each tidy subgroup $U$ as a product of four subgroups. This factoring continues.
\begin{proposition}
\label{prop:factor_U}
Suppose that $\fl{H}\leq\Aut(G)$ is flat and that $U$ is tidy for $\fl{H}$. Let $\alpha_1$, \dots, $\alpha_n$ be in $\fl{H}$ and denote $U_\varepsilon = U_{\alpha_1\varepsilon_1}\cap \dots \cap U_{\alpha_n\varepsilon_n}$ for each $\varepsilon\in\{-,+\}^n$. Then 
$$
U = \prod_{\varepsilon\in\{-,+\}^n} U_\varepsilon,
$$
with the subgroups $U_\varepsilon$ in the lexicographic order on $\{-,+\}^n$ given by setting $-$ to be less than $+$. 
\end{proposition}
\begin{proof}
The proof is by induction on $n$. The base case is tidiness above of~$U$ and the $n=2$ case follows from Lemma~\ref{lem:factor_U} applied to $\Um{U}{\alpha_1}$ and $\Um{U}{\alpha_1^{-1}}$. Assume that the claim has been established for $n$ and let $\alpha_1,\dots,\alpha_{n+1}$ be in $\fl{H}$. Define $U(\alpha_1,K)$, for $K\geq0$, as in the proof of Lemma~\ref{lem:factor_U}. Then $U(\alpha_1,K)$ is tidy for $\fl{H}$ and so, by the induction hypothesis applied to the sequence $\alpha_2$, \dots, $\alpha_{n+1}$,
$$
U(\alpha_1,K) = \prod_{\varepsilon\in\{-,+\}^n} U_{\varepsilon}
$$
with $U_\varepsilon = U_{\alpha_2\varepsilon_2}\cap \dots \cap U_{\alpha_{n+1}\varepsilon_{n+1}}$. Then repeating the argument used in the proof of Lemma~\ref{lem:factor_U} yields that
$$
\Um{U}{\alpha_1} =  \prod_{\varepsilon\in\{-,+\}^n} (\Um{U}{\alpha_1}\cap U_{\varepsilon})\ \mbox{ and }\ \Up{U}{\alpha_1} =  \prod_{\varepsilon\in\{-,+\}^n} (\Up{U}{\alpha_1}\cap U_{\varepsilon})
$$
and the induction continues. 
\end{proof}

Proposition~\ref{prop:factor_U} prescribes the order in which the factors $U_\varepsilon$ appear. Although $\Um{U}{\alpha}\Up{U}{\alpha} = \Up{U}{\alpha}\Um{U}{\alpha}$, the subgroups $U_\varepsilon$ cannot be permuted arbitrarily.

That $U_{\alpha-}\cap U_{\beta-}$ is invariant under $\alpha$ and $\beta$ is also important in the sequel. The proof uses an idea from the previous argument. 
\begin{lemma}
\label{lem:factors_invariant}
Suppose that $\fl{H}\leq\Aut(G)$ is flat. Let $U$ be tidy for $\fl{H}$ and let $\alpha,\beta\in\fl{H}$. Then, for every $\upsilon$ in the semigroup generated by $\alpha$ and $\beta$,
$$
\upsilon(U_{\alpha-}\cap U_{\beta-}) \leq U_{\alpha-}\cap U_{\beta-}.
$$
\end{lemma}
\begin{proof}
It suffices to show that $\beta(U_{\alpha-}\cap U_{\beta-}) \leq U_{\alpha-}\cap U_{\beta-}$ because the claim then follows by symmetry between $\alpha$ and $\beta$ and induction.

As in the proof of Lemma~\ref{lem:factor_U}, set $U(\alpha,K) = \bigcap_{k=0}^K \alpha^{-k}(U)$. Then $U(\alpha,K)$ is tidy for $\fl{H}$ and $U_{\alpha-}\cap U_{\beta-}\leq U(\alpha,K)_{\beta-}$ by Proposition~\ref{prop:tidy_criteria}. Hence
$$
\beta(U_{\alpha-}\cap U_{\beta-}) \leq \beta(U(\alpha,K)_{\beta-}) \leq U(\alpha,K)_{\beta-}
$$
for every $K\geq0$ and 
$$
\beta(U_{\alpha-}\cap U_{\beta-}) \leq \bigcap_{K\geq0} U(\alpha,K)_{\beta-}.
$$
An application of Proposition~\ref{prop:tidy_criteria} shows that this intersection is contained in $\Um{U}{\alpha}\cap \Um{U}{\beta}$, as required.
\end{proof}

It is seen next that the factoring of $U_{\alpha-}$ in Lemma~\ref{lem:factor_U} extends to $\Umm{U}{\alpha}$, albeit with a third factor required.
\begin{proposition}
\label{prop:factor_U--}
Suppose that $\fl{H}\leq\Aut(G)$ is flat. Let $U$ be tidy for $\fl{H}$ and let $\alpha,\beta\in\fl{H}$. Then 
\begin{equation}
\label{eq:factor_U--}
\Umm{U}{\alpha} = (\Umm{U}{\alpha}\cap \lev(\beta))(\Umm{U}{\alpha}\cap \Umm{U}{\beta})(\Umm{U}{\alpha}\cap \Upp{U}{\beta}).
\end{equation}
\end{proposition}
\begin{proof}
The argument uses the temporary definitions  
\begin{align} 
\MUab &= \bigcup_{k\in\mathbb{N}} \alpha^{-k}(U_{\alpha-})\cap\alpha^{-k}(U)_{\beta-},
\label{eq:factor_U--1}\\ 
\mbox{ and }\ \  \PUab &= \bigcup_{k\in\mathbb{N}} \alpha^{-k}(U_{\alpha-})\cap\alpha^{-k}(U)_{\beta+}
\label{eq:factor_U--2}.
\end{align}
\begin{Claim} 
\label{cl:one}
$\MUab$ and $\PUab$ are subgroups of $\Umm{U}{\alpha}$.
\end{Claim} 
To see that $\MUab$ is a subgroup, note first that $\parb(\beta)\geq \alpha^{-k}(U)_{\beta-}$. On the other hand, since $\alpha^{-k}(U)$ is tidy for~$\beta$ by Proposition~\ref{prop:invariance_of_tidiness}, Proposition~\ref{prop:tidy_criteria} implies that
$$
\alpha^{-k}(U_{\alpha-})\cap\parb(\beta)\leq \alpha^{-k}(U)_{\beta-}.
$$ 
Hence
$$
\alpha^{-k}(U_{\alpha-})\cap\alpha^{-k}(U)_{\beta-} = \alpha^{-k}(U_{\alpha-})\cap\parb(\beta). 
$$ 
Since $\alpha^{k-1}(U_{\alpha-})\geq \alpha^{k}(U_{\alpha-})$ for each~$k$, it follows that\footnote{Note that $\alpha^{k-1}(U)\not\geq \alpha^{k}(U)$ in general and $\alpha^{-k-1}(U)_{\beta-}\not\geq \alpha^{-k}(U)_{\beta-}$.}
$$
\alpha^{-k-1}(U_{\alpha-})\cap\alpha^{-k-1}(U)_{\beta-}\geq \alpha^{-k}(U_{\alpha-})\cap\alpha^{-k}(U)_{\beta-}
$$ 
and hence that~\eqref{eq:factor_U--1} defines $\MUab$ to be an increasing union of subgroups of $\Umm{U}{\alpha}$. A similar argument shows that $\PUab$ is a subgroup of $\Umm{U}{\alpha}$. \endproof

An intermediate stage towards the claimed factoring is as follows: 
\begin{Claim} 
\label{cl:two}
$\Umm{U}{\alpha} = \MUab\PUab$. 
\end{Claim}
Equations~\eqref{eq:factor_U--1} and~\eqref{eq:factor_U--2} imply that $\MUab\PUab\leq \Umm{U}{\alpha}$ because $\alpha^{-k}(U_{\alpha-})\leq \Umm{U}{\alpha}$ for every $k\in\mathbb{N}$. For the reverse inclusion, recall that $\Umm{U}{\alpha} = \bigcup_{k\in\mathbb{N}} \alpha^{-k}(U)_{\alpha-}$ and that, by Lemma~\ref{lem:factor_U}, 
\begin{equation*}
\label{eq:factor_intermediate}
\Um{\alpha^{-k}(U)}{\alpha} = (\alpha^{-k}(\Um{U}{\alpha})\cap\Um{\alpha^{-k}(U)}{\beta}) (\alpha^{-k}(\Um{U}{\alpha})\cap\Up{\alpha^{-k}(U)}{\beta}).
\end{equation*}
Then $\Umm{U}{\alpha}\leq \MUab\PUab$ because $\Um{\alpha^{-k}(U)}{\alpha} = \alpha^{-k}(\Um{U}{\alpha})$. \endproof

The proposition will follow by combining Claim~\ref{cl:two} with 
\begin{Claim} 
\label{cl:three}
$\PMUab = (\Umm{U}{\alpha}\cap \lev(\beta))(\Umm{U}{\alpha}\cap \Utpm{U}{\beta})$. 
\end{Claim}
We factor the terms in the definition of $\MUab$ in Equation~\eqref{eq:factor_U--1}. Consider $y\in \alpha^{-k}(\Um{U}{\alpha})\cap\Um{\alpha^{-k}(U)}{\beta}$. As seen in the proof of Claim~\ref{cl:one}, $y\in \parb(\beta)$ and so by part \ref{prop:acc.pt_in_lev3}.~of Proposition~\ref{prop:acc.pt_in_lev}, there are: $z\in\con(\beta)$; $u\in \Uz{U}{\beta}$; and $x\in \bigl(\bigcap_{l\in\mathbb{N}} \overline{\{\beta^k(y)\}}_{k\geq l}\bigr)\subseteq \lev(\beta)$ such that $y=xuz$. Applying  Lemma~\ref{lem:factors_invariant} to $y$ shows that $x$ belongs to $\alpha^{-k}(U_{\alpha-})$ and hence 
$$
x\in\alpha^{-k}(\Um{U}{\alpha})\cap \lev(\beta). 
$$
Since $y$ and $x$ belong to $\alpha^{-k}(\Um{U}{\alpha})$, so does $uz$. Hence, since $\Uz{U}{\beta}\con(\beta) = \Umm{U}{\beta}$, it follows that 
$$uz\in \alpha^{-k}(U_{\alpha-})\cap \Umm{U}{\beta}
$$ 
and we have shown that
$$
\alpha^{-k}(U_{\alpha-})\cap\alpha^{-k}(U)_{\beta-} \leq \bigl(\alpha^{-k}(U_{\alpha-})\cap \lev(\beta)\bigr)\bigl(\alpha^{-k}(U_{\alpha-})\cap \Umm{U}{\beta}\bigr).
$$
For the reverse inclusion, note that, since $\alpha^{-k}(U)$ is tidy for $\beta$,  Proposition~\ref{prop:tidy_criteria} implies that $\alpha^{-k}(U_{\alpha-})\cap \lev(\beta)$ and $\alpha^{-k}(U_{\alpha-})\cap \Umm{U}{\beta}$ are both contained in $\alpha^{-k}(U)_{\beta-}$. Hence 
\begin{equation*}
\label{eq:factor_intermediate2}
\alpha^{-k}(U_{\alpha-})\cap\alpha^{-k}(U)_{\beta-} = \bigl(\alpha^{-k}(U_{\alpha-})\cap \lev(\beta)\bigr)\bigl(\alpha^{-k}(U_{\alpha-})\cap \Umm{U}{\beta}\bigr).
\end{equation*}
The definition of $\MUab$ in Equation~\eqref{eq:factor_U--1} now implies that
$$
\MUab = \bigl(\Umm{U}{\alpha}\cap \lev(\beta)\bigr)\bigl(\Umm{U}{\alpha}\cap \Umm{U}{\beta}\bigr).
$$
The claim for $\PUab$ follows by replacing $\beta$ with $\beta^{-1}$. \endproof

Observing that $\bigl(\Umm{U}{\alpha}\cap \Umm{U}{\beta}\bigr)\bigl(\Umm{U}{\alpha}\cap \lev(\beta)\bigr) = \bigl(\Umm{U}{\alpha}\cap \lev(\beta)\bigr)\bigl(\Umm{U}{\alpha}\cap \Umm{U}{\beta}\bigr)$, Claims~\ref{cl:two} and~\ref{cl:three} imply~\eqref{eq:factor_U--}.
\end{proof}

It is also useful to highlight the following, which was shown as an intermediate step in the last proof.
\begin{corollary}
\label{cor:factor_U--}
Assume the hypotheses of Proposition~\ref{prop:factor_U--} and let $k$ be a positive integer. Then
$$
\alpha^{-k}(\Um{U}{\alpha})\cap\Um{\alpha^{-k}(U)}{\beta} = \left(\alpha^{-k}(\Um{U}{\alpha})\cap \lev(\beta)\right)\bigl(\alpha^{-k}(\Um{U}{\alpha})\cap \Umm{U}{\beta}\bigr).
$$
\end{corollary}

The shrinking subgroup $\Umm{U}{\alpha}$ is automatically invariant under $\alpha$ and each of the groups $\Utpm{U}{\beta}$ and $\lev(\beta)$ is invariant under $\beta$. However, the factors of $\Umm{U}{\alpha}$ given in Proposition~\ref{prop:factor_U--} might not be invariant under $\alpha$ and $\beta$. Here is the counterpart of Lemma~\ref{lem:factors_invariant} for this situation. 
\begin{lemma}
\label{lem:factors_invariantii}
Suppose that $\fl{H}\leq\Aut(G)$ is flat. Let $U$ be tidy for $\fl{H}$ and let $\alpha,\beta\in\fl{H}$. Then $\Umm{U}{\alpha}\cap \Umm{U}{\beta}$ and $\Umm{U}{\alpha}\cap \lev(\beta)$ are invariant under $\beta$ and 
\begin{equation}
\label{eq:factors_invariantii}
\beta\bigl(\Umm{U}{\alpha}\cap \Upp{U}{\beta}\bigr) \leq \bigl(\Umm{U}{\alpha}\cap \Upp{U}{\beta}\bigr) \beta(\Uz{U}{\alpha}).
\end{equation}
\end{lemma}
\begin{proof}
Since $\alpha^{-k}(U)$ is tidy for $\fl{H}$ for each $k$, Lemma~\ref{lem:factors_invariant} implies that $\MUab$, defined in Equation~\eqref{eq:factor_U--1}, is $\beta$-invariant. Then $\beta$-invariance of $\Umm{U}{\alpha}\cap \Umm{U}{\beta}$ and $\Umm{U}{\alpha}\cap \lev(\beta)$ holds because they are equal to the intersections of $\MUab$ with the $\beta$-invariant groups $\Umm{U}{\beta}$ and $\lev(\beta)$ respectively.

The same argument also shows that $\Umm{U}{\alpha}\cap \Upp{U}{\beta}$ is $\beta^{-1}$-invariant, rather than $\beta$-invariant. If the restriction of $\beta^{-1}$ to $\Umm{U}{\alpha}\cap \Upp{U}{\beta}$ were surjective, then we would have that $\beta\left(\Umm{U}{\alpha}\cap \Upp{U}{\beta}\right) = \Umm{U}{\alpha}\cap \Upp{U}{\beta}$. This restriction need not be surjective however. Instead, it will be shown that 
\begin{equation}
\label{eq:not_surjective}
\Umm{U}{\alpha}\cap \Upp{U}{\beta}\leq \beta^{-1}\left(\Umm{U}{\alpha}\cap \Upp{U}{\beta}\right) \Uz{U}{\alpha}.
\end{equation}

Since $\Umm{U}{\alpha}\cap\Upp{U}{\beta} = \bigcup_{k\in\mathbb{Z}} \beta^k(\Up{U}{\beta})\cap \Umm{U}{\alpha}$,~\eqref{eq:not_surjective} follows from
\setcounter{Claim}{0}
\begin{Claim-non} 
\label{cl:four}
$\Up{\beta^{k}(U)}{\beta}\cap \Umm{U}{\alpha} \leq \beta^{-1}\left(\Up{\beta^{k+1}(U)}{\beta}\cap \Umm{U}{\alpha}\right)\Uz{U}{\alpha}$ for each~$k$. 
\end{Claim-non}
To see this, note that
\begin{align*}
\beta^{k+1}(U_{\beta+}) &=\left(\beta^{k+1}(U_{\beta+})\cap \beta^{k+1}(U)_{\alpha-}\right)\left(\beta^{k+1}(U_{\beta+})\cap \beta^{k+1}(U)_{\alpha+}\right), \\
\intertext{by Lemma~\ref{lem:factor_U} with $U$ replaced by $\beta^{k+1}(U)$, $\alpha$ by $\beta^{-1}$ and $\beta$ by $\alpha$,}
&= \bigl(\beta^{k+1}(U_{\beta+})\cap \Umm{U}{\alpha}\bigr)\bigl(\beta^{k+1}(U_{\beta+})\cap \beta^{k+1}(U)_{\alpha+}\bigr), 
\end{align*}
by Corollary~\ref{cor:factor_U--} and adsorbing the $\beta^{k+1}(U_{\beta+})\cap \lev(\alpha)$ term produced into $\beta^{k+1}(U_{\beta+})\cap \beta^{k+1}(U)_{\alpha+}$. Applying $\beta^{-1}$ to both sides yields
$$
\beta^k(U_{\beta+}) = \beta^{-1}\bigl(\beta^{k+1}(U_{\beta+})\cap \Umm{U}{\alpha}\bigr)
\beta^{-1}\bigl(\beta^{k+1}(U_{\beta+})\cap \beta^{k+1}(U)_{\alpha+}\bigr)
$$
and hence, applying Lemma~\ref{lem:factors_invariant} to the second factor, that
$$
\beta^k(U_{\beta+}) \leq \beta^{-1}\bigl(\beta^{k+1}(U_{\beta+})\cap \Umm{U}{\alpha}\bigr)
\bigl(\beta^{k+1}(U_{\beta+})\cap \beta^{k+1}(U)_{\alpha+}\bigr).
$$
Observe that $\beta^{k+1}(U_{\beta+})\cap \Umm{U}{\alpha} = \beta^{k+1}(U_{\beta+})\cap \bigl(\Upp{U}{\beta}\cap \Umm{U}{\alpha}\bigr)$ is $\beta^{-1}$-invariant and hence that the first factor on the right hand side is contained in $\Umm{U}{\alpha}$. Hence 
$$
\beta^k(U_{\beta+})\cap \Umm{U}{\alpha} \leq \beta^{-1}\bigl(\beta^{k+1}(U_{\beta+})\cap \Umm{U}{\alpha}\bigr)
\bigl(\beta^{k+1}(U_{\beta+})\cap \beta^{k+1}(U)_{\alpha+}\cap \Umm{U}{\alpha}\bigr).
$$
Since $\beta^{k+1}(U)_{\beta+} \cap \beta^{k+1}(U)_{\alpha+}\cap \Umm{U}{\alpha}\leq \parb(\alpha^{-1})\cap \Umm{U}{\alpha}$, which is contained in $\Uz{U}{\alpha}$ by Proposition~\ref{prop:tidy_criteria}, the claim follows.  \endproof

Having thus established the claim, \eqref{eq:not_surjective} follows and then, applying $\beta$ to both sides, so does Equation~\eqref{eq:factors_invariantii}.
\end{proof}

\begin{proposition}
\label{prop:U_++_invariance}
Suppose that $\fl{H}\leq\Aut(G)$ is flat. Let $U$ be tidy for $\fl{H}$ and let $\alpha,\beta\in\fl{H}$.  Then
$$
\beta(\Umm{U}{\alpha}) \leq  \Umm{U}{\alpha}\,\beta(\Uz{U}{\alpha}).
$$
\end{proposition}
\begin{proof}
By Proposition~\ref{prop:factor_U--},
$$
\beta(\Umm{U}{\alpha}) = \beta(\Umm{U}{\alpha}\cap \lev(\beta))\beta(\Umm{U}{\alpha}\cap \Umm{U}{\beta})\beta(\Umm{U}{\alpha}\cap \Upp{U}{\beta}).
$$
Lemma~\ref{lem:factors_invariantii} shows that $\Umm{U}{\alpha}\cap \lev(\beta)$ and $\Umm{U}{\alpha}\cap \Umm{U}{\beta}$ are invariant under~$\beta$ and that $\beta(\Umm{U}{\alpha}\cap \Upp{U}{\beta})\leq \bigl(\Umm{U}{\alpha}\cap \Upp{U}{\beta}\bigr) \beta(\Uz{U}{\alpha})$.
\end{proof}

\begin{remark} 
\label{rem:not_invariant}
The factor $\beta(\Uz{U}{\alpha})$ appearing in Proposition~\ref{prop:U_++_invariance} is needed because it may happen $\Umm{U}{\alpha}$ is not $\fl{H}$-invariant, as seen in the example in Exercise~\ref{exer:zero_necessary}.
\end{remark}

\subsection{Contraction modulo the nub}
\label{sec:nub_contraction}

Arguments to follow in \S\ref{sec:reduced1} will apply the results in \S\ref{sec:factor-tidy} repeatedly by restricting $\fl{H}$ to smaller factors. In order to do this, the smaller factors must be stable under $\fl{H}$ and the restriction of $\fl{H}$ to them must be flat. As pointed out in Remark~\ref{rem:not_invariant}, that approach is not going to work with $\Umm{U}{\alpha}$. Results in this section avoid the difficulty by showing, for each flat group $\fl{H}$, that: 
\begin{itemize}
\item the subgroup $\con(\alpha/\nub(\fl{H}))$ is stable under $\fl{H}$ for each $\alpha\in\fl{H}$; 
\item  the restriction of $\fl{H}$ to $\con(\alpha/\nub(\fl{H}))$ is flat; 
\item  $\fl{H}$ is abelian modulo $\fl{H}_u$.
\end{itemize}

The proof that $\con(\alpha/\nub(\fl{H}))$ is stable under $\fl{H}$ uses the following characterisation, which extends \cite[Theorem~3.26]{ContractionB} to flat groups.
\begin{proposition}
\label{prop:rep_description}
Suppose that $\fl{H}\leq\Aut(G)$ is flat and let $\alpha\in\fl{H}$. Then $\nub(\fl{H})$ normalises ${\con(\alpha)}$ and 
$$
\bigcap\bigl\{ \Umm{U}{\alpha} \mid U\mbox{ is tidy for }\fl{H}\bigr\} = \con(\alpha/\nub(\fl{H})) = \con(\alpha) \nub(\fl{H}).
$$
\end{proposition}
\begin{proof}
That $\con(\alpha/\nub(\fl{H}))$ is equal to $\con(\alpha)\nub(\fl{H})$ follows from   Theorem~\ref{thm:convergence_mod_H} because $\nub(\fl{H})$ is $\alpha$-stable and compact, and $\nub(\fl{H})$ normalises $\con(\alpha)$ by Lemma~\ref{lem:nub_normalised}. 

Suppose that $U$ is tidy for $\fl{H}$. Then $\Umm{U}{\alpha} = \con(\alpha)\Uz{U}{\alpha}$, by Proposition~\ref{prop:U--^con}. Furthermore, $\nub(\fl{H})\leq \Uz{U}{\alpha}$ by Proposition~\ref{prop:invariance_of_tidiness} because  $\Uz{U}{\alpha} = \bigcap_{n\in\mathbb{Z}} \alpha^n(U)$. Hence $\con(\alpha) \nub(\fl{H})\leq \bigcap\bigl\{ \Umm{U}{\alpha} \mid U\mbox{ is tidy for }\fl{H}\bigr\}$. 

For the reverse inclusion, we show that 
\begin{equation}
\label{eq:rep_description}
\bigcap\bigl\{ \Umm{U}{\alpha} \mid U\mbox{ is tidy for }\fl{H}\bigr\}\leq \con(\alpha/\nub(\fl{H})).
\end{equation}
Given $\open{O}\supseteq \nub(\fl{H})$ open, use Lemma~\ref{lem:open_nub} to choose~$U$ tidy for $\fl{H}$ with $U\subseteq\open{O}$. Then the definition of $\Umm{U}{\alpha}$ implies that, for each $x\in \Umm{U}{\alpha}$, there is $n\geq0$ such that $\alpha^n(x)\in \Um{U}{\alpha}\subset\open{O}$. Hence~\eqref {eq:rep_description} holds.
\end{proof}

Here is the main result for this section.
\begin{theorem}
\label{thm:factors_invariant}
Suppose that $\fl{H}\leq\Aut(G)$ is flat and let $\alpha\in\fl{H}$.  Then $\con(\alpha/\nub(\fl{H}))$ is stable under $\fl{H}$. 

Moreover, given $V$ tidy for $\fl{H}$, the subgroup $V\cap \con(\alpha/\nub(\fl{H}))$ is tidy for the restriction of $\fl{H}$ to $\con(\alpha/\nub(\fl{H}))$. Therefore this restriction is a flat subgroup of $\Aut(\con(\alpha/\nub(\fl{H})))$.
\end{theorem}
\begin{proof} The proof that $\con(\alpha/\nub(\fl{H}))$ is $\fl{H}$-stable uses that 
\begin{equation}
\label{eq:factors_invariant}
\con(\alpha/\nub(\fl{H})) = \bigcap\left\{ \Umm{U}{\alpha}V \mid V\in\COS(G),\ U\mbox{ tidy for }\fl{H}\right\},
\end{equation}
which may be derived from Proposition~\ref{prop:rep_description} by noting that, if $x$ is not in $\con(\alpha/\nub(\fl{H}))$, then there are $U$ tidy for $\fl{H}$ and $V\in\COS(G)$ such that $\Umm{U}{\alpha}\cap xV = \emptyset$, from which it follows that $x\not\in \Umm{U}{\alpha}V$ and \eqref{eq:factors_invariant} holds. Consider such $U$ and $V$ and let $\beta\in\fl{H}$. Since $\beta$ is continuous, there is $W\in\COS(G)$ such that $\beta(W)\leq V$. Use Lemma~\ref{lem:open_nub} to choose $U'$ tidy for $\fl{H}$ such that $U'\subseteq U\cap \nub(\fl{H})W$. Then $\Uz{U'}{\alpha}\subseteq \nub(\fl{H})W$ and, by Proposition~\ref{prop:U_++_invariance},
$$
\beta(\Umm{U'}{\alpha})\leq \Umm{U'}{\alpha}\beta(\Uz{U'}{\alpha}) \subseteq \Umm{U'}{\alpha}\nub(\fl{H})\beta(W)\subseteq \Umm{U}{\alpha}V.
$$
Hence $\beta(\con(\alpha/\nub(\fl{H})))$ is contained in $\Umm{U}{\alpha}V$ for every $U$ tidy for $\fl{H}$ and every $V\in\COS(G)$, and~\eqref{eq:factors_invariant} implies that  $\con(\alpha/\nub(\fl{H}))$ is $\beta$-stable for every $\beta\in \fl{H}$. 

Next, let $V$ be tidy for~$\fl{H}$ and set $\widetilde{V} := V\cap \con(\alpha/\nub(\fl{H}))$. It must be shown that $\widetilde{V}$ is tidy for $\beta|_{\con(\alpha/\nub(\fl{H}))}$ for every $\beta\in\fl{H}$. That $\widetilde{V}$ is tidy below for $\beta|_{\con(\alpha/\nub(\fl{H}))}$ holds by Proposition~\ref{prop:functor}. However, as seen in \cite[Example 6.4]{FurtherTidy} and Exercise~\ref{ex:tidymeet}, tidiness above need not be preserved when a subgroup tidy for $\beta$ meets a closed $\beta$-stable one. A special argument is needed therefore to show that $\widetilde{V}$ is tidy above. (Lemma~\ref{lem:tidymeetsshrinking} does not apply because it may happen that both $s(\beta)>1$ and $s(\beta^{-1})>1$.)

To show that $\widetilde{V}$ is tidy above for every $\beta\in\fl{H}$, we use the following modification of the description of $\con(\alpha/\nub(\fl{H}))$ given in Proposition~\ref{prop:rep_description}:
\begin{equation}
\label{eq:con_mod}
\con(\alpha/\nub(\fl{H})) = \bigcap\bigl\{ \Umm{U}{\alpha} \mid U\mbox{ tidy for }\fl{H}\mbox{ and }U\geq\widetilde{V} \bigr\}.
\end{equation}
To see this, observe that for each $U$ tidy for $\fl{H}$ there is $k\geq0$ such that $V\cap\Umm{U}{\alpha} =  V\cap \Um{\alpha^{-k}(U)}{\alpha}$ because $\Umm{U}{\alpha} = \bigcup_{k\geq0} \alpha^{-k}(\Um{U}{\alpha})$ and $V\cap\Umm{U}{\alpha}$ is compact. Replacing $U$ with $\alpha^{-k}(U)$ leaves $\Umm{U}{\alpha}$ unchanged and ensures that $U\geq\widetilde{V}$ as required in~\eqref{eq:con_mod}. 

Let $U$ be as in~\eqref{eq:con_mod} and $\beta\in\fl{H}$. Then $V\cap \Umm{U}{\alpha} = \Um{(V\cap U)}{\alpha}$, by Proposition~\ref{prop:tidy_criteria}, and $V\cap U$ is tidy for $\beta$, by Theorem~\ref{thm:tidy_intersection}. Hence, by Lemma~\ref{lem:factor_U}, 
$$
\Um{(V\cap U)}{\alpha} = \bigl(\Um{(V\cap U)}{\alpha}\cap \Up{(V\cap U)}{\beta}\bigr)\bigl(\Um{(V\cap U)}{\alpha}\cap \Um{(V\cap U)}{\beta}\bigr).
$$
Fixing $x\in \widetilde{V}\leq V\cap \Umm{U}{\alpha}= \Um{(V\cap U)}{\alpha}$ and denoting $\Um{(V\cap U)}{\alpha}\cap \Upm{(V\cap U)}{\beta}$ by $L_{\pm}$ respectively, this implies that
$$
C(x,U) := \left\{(x_+,x_-)\in L_+\times L_-\mid x = x_+x_-\right\}\mbox{ is not empty.}
$$ 
Hence 
$$
\mathcal{C} := \bigl\{C(x,U) \mid U\mbox{ tidy for }\fl{H}\mbox{ and }U\geq\widetilde{V} \bigr\}
$$ 
is a set of non-empty compact subsets of $G$. Further, $\mathcal{C}$ has the finite intersection property because the conditions of being tidy for $\fl{H}$ and of containing $\widetilde{V}$ are preserved under finite intersections. Therefore $\bigcap\mathcal{C} \ne \emptyset$. By construction, if $(x_+,x_-)\in \bigcap\mathcal{C}$, then $x_\pm\in V\cap \Umm{U}{\alpha}$ for all $U$ tidy for $\fl{H}$ and containing $\widetilde{V}$, whence~\eqref{eq:con_mod} implies that $x_\pm\in \widetilde{V}$. Further, since $\beta^{\mp n}(L_{\pm})\leq L_{\pm}$, by  Lemma~\ref{lem:factors_invariant}, we have that $\beta^{\mp n}(x_{\pm})\in \widetilde{V}$ for all $n\geq0$ and hence that $x_{\pm}\in \Upm{\widetilde{V}}{\beta}$. We also have that $x = x_+x_-$ by construction. 

Since this argument holds for all $x\in \widetilde{V}$, we have shown that $\widetilde{V}$ is tidy above for $\beta$. and, since this holds for all $\beta\in\fl{H}$, $\widetilde{V}$ is tidy for $\fl{H}$.
\end{proof}

That $\fl{H}/\fl{H}_u$ is abelian may now be deduced.
\begin{corollary}
\label{cor:Flat=>abelian_mod_uniscalar}
Let $\fl{H}\leq\Aut(G)$ be flat. Then $[\fl{H},\fl{H}]\leq \fl{H}_u$. 
\end{corollary}
\begin{proof}
Let $\alpha,\beta\in\fl{H}$ and put $\gamma = [\alpha,\beta]$. Denote the restrictions of $\alpha$, $\beta$ and $\gamma$ to $\con(\gamma^{-1}/\nub(\fl{H}))$ by $\hat{\alpha}$, $\hat{\beta}$ and $\hat{\gamma}$ respectively. Then 
$$
s(\gamma) = s(\hat{\gamma}) = \Delta(\hat{\gamma})
$$ 
by Proposition~\ref{prop:restriction_of_scale}. Since the modular function is multiplicative and, as seen in Theorem~\ref{thm:factors_invariant}, $\con(\gamma^{-1}/\nub(\fl{H}))$ is stable under $\alpha$ and $\beta$, we have that
$$
\Delta(\hat{\gamma}) = \Delta(\hat{\alpha})\Delta(\hat{\beta})\Delta(\hat{\alpha})^{-1}\Delta(\hat{\beta})^{-1} = 1
$$
and hence that $[\alpha,\beta]$ is uniscalar. 
\end{proof}

\subsection{Subgroups of $G$ that are scaling for $\fl{H}$}
\label{sec:reduced1}

In this section the existence of subgroups scaling for $\fl{H}$ is shown by proving that every
subgroup $\Umm{U}{\alpha}\leq G$, with $\alpha\in\fl{H}$, is a product of a finite number of subgroups scaling for $\fl{H}$. This is achieved by using results from \S\S\ref{sec:factor-tidy}, \ref{sec:nub_contraction} to reduce $\Umm{U}{\alpha}$ to a product of smaller $\fl{H}$-invariant subgroups. (The term `reduced' is used in analogy with linear algebra, where a subspace invariant under a linear operator reduces the operator to a pair of operators on lower dimensional spaces.) The scale on $\fl{H}$ factors along with each reduction and reduction terminates with a finite number of scaling subgroups when the integer-valued scale cannot be factored further. 

Throughout this section, $N$ is a compact $\fl{H}$-invariant subgroup of $G$ and $\nub(\fl{H})\leq N$. The next result shows the factoring of the scale. The hypothesis that $s(\alpha)=1$ holds, for instance, if $\alpha$ has been restricted to $\con(\alpha/N)$.
\begin{proposition}
\label{prop:factor_U-}
Let $\alpha,\beta$ be in the flat group $\fl{H}$ with $s(\alpha)=1$. Let $N$ be a compact $\fl{H}$-stable subgroup of $G$ and $U$ be tidy for $\fl{H}$; and suppose that  $\nub(\fl{H})\leq N \leq U$. 
Then $\alpha^{-1}(\Up{U}{\beta})\cap U =  \Up{U}{\beta}$ and 
$$
s(\alpha^{-1}) = [\alpha^{-1}(\Up{U}{\beta}) : \Up{U}{\beta}] s(\alpha^{-1}|_{\con(\beta/N)}).
$$
Further, $[\alpha^{-1}(\Up{U}{\beta}) : \Up{U}{\beta}]=1$ only if $\Up{U}{\beta}\leq \Uz{U}{\alpha}$ and $s(\alpha^{-1}|_{\con(\beta/N)})=1$ only if $\con(\beta/N)\cap U\leq \Uz{U}{\alpha}$.
\end{proposition}
\begin{proof}
An application of Lemma~\ref{lem:factors_invariant} shows that $\alpha^{-1}(\Up{U}{\beta})\geq \Up{U}{\beta}$ and hence that $\alpha^{-1}(\Up{U}{\beta})\cap U\geq \Up{U}{\beta}$. For the reverse inclusion, suppose that $u\in \alpha^{-1}(\Up{U}{\beta})\cap U$. Then, tidiness above and Proposition~\ref{prop:U--^con} imply that $u=u_+u_-$, with $u_+\in \Up{U}{\beta}$ and $u_-\in \con(\beta)\cap U =: \Um{C}{\beta}$, and we have that
$$
\alpha(u_-) = \alpha(u_+)^{-1}\alpha(u)\in \con(\beta/\nub(\fl{H}))\cap \Up{U}{\beta},
$$
because $\alpha(u_+)^{-1}\in \Up{U}{\beta}$ by Lemma~\ref{lem:factors_invariant} and $\alpha(u_-)\in \con(\beta/\nub(\fl{H}))$ by Theorem~\ref{thm:factors_invariant}. Then $\alpha(u_-)\in \nub(\fl{H})$ by Corollary~\ref{cor:convergence_mod_H}. Since $\nub(\fl{H})$ is $\alpha$-stable, it follows that $u_-\in\nub(\fl{H})$ and hence that $u\in\Up{U}{\beta}$. For future reference, note that the same argument applied to $\beta^n(\Up{U}{\beta})=\Up{\beta^n(U)}{\beta}$ also shows that, for each $n\geq0$,
\begin{equation}
\label{eq:n_version}
\alpha^{-1}\beta^n(\Up{U}{\beta})\cap U = \Up{U}{\beta}.
\end{equation} 

Towards the formula for $s(\alpha^{-1})$, note that the hypotheses that $s(\alpha)=1$ and $U$ is tidy for $\fl{H}$ imply that $\alpha^{-1}(U)\geq U$ and $\alpha^{-1}\beta^n(U)\geq \beta^n(U)$. Hence 
\begin{equation}
\label{eq:factor_U-}
\alpha^{-1}\left(\beta^n(U)\cap U\right) \geq \alpha^{-1}\beta^n(U)\cap U \geq \beta^n(U)\cap U.
\end{equation}
Tidiness of $\beta^n(U)\cap U$ for $\alpha^{-1}$ and Theorem~\ref{thm:tidy_intersection} then imply that
\begin{align*}
s(\alpha^{-1}) &= \left[\alpha^{-1}\left(\beta^n(U)\cap U\right) : \beta^n(U)\cap U\right]\\
&= \left[\alpha^{-1}\left(\beta^n(U)\cap U\right) : \alpha^{-1}\beta^n(U)\cap U\right]\left[\alpha^{-1}\beta^n(U)\cap U : \beta^n(U)\cap U\right].
\end{align*}
The formula for $s(\alpha^{-1})$ thus follows once it is shown that 
\begin{align}
\left[\alpha^{-1}\left(\beta^n(U)\cap U\right) : \alpha^{-1}\beta^n(U)\cap U\right] &= [\alpha^{-1}(\Up{U}{\beta}) : \Up{U}{\beta}]\label{1st_identity}
\\
\mbox{ and } \quad\left[\alpha^{-1}\beta^n(U)\cap U : \beta^n(U)\cap U\right] &= s(\alpha^{-1}|_{\con(\beta/N)}). \label{2nd_identity}
\end{align}

\stepcounter{equation}
As in the first paragraph, set $\Um{C}{\beta} = \con(\beta)\cap U$. Then $U = \Um{C}{\beta}\Up{U}{\beta}$ and, since $\beta(\Um{C}{\beta})\leq \Um{C}{\beta}$ and $\beta(\Up{U}{\beta})\geq \Up{U}{\beta}$, it follows that
\begin{align}
\beta^n(U)\cap U &= \beta^n(\Um{C}{\beta})\Up{U}{\beta}\mbox{ and }\tag{\theequation a}
\label{eq:equ1}\\
\alpha^{-1}\left(\beta^n(U)\cap U\right) &=  \alpha^{-1}\beta^n(\Um{C}{\beta})\alpha^{-1}(\Up{U}{\beta}).\tag{\theequation b}
\label{eq:equ2}
\end{align}
Choose $n>0$ such that $\beta^n(\Um{C}{\beta})\leq \alpha(U)$, as we may because $\Um{C}{\beta}$ is compact and $\beta^n(\Um{C}{\beta})\to N$, and $\alpha(U)$ is an open neighbourhood of $N$. Then $\alpha^{-1}\beta^n(\Um{C}{\beta})\leq \Um{C}{\beta}$ because of the $\fl{H}$-stability of $N$ and, by Theorem~\ref{thm:factors_invariant},  $\con(\beta/N)$.  Hence $\alpha^{-1}\beta^n(U) = \alpha^{-1}\beta^n(\Um{C}{\beta})\alpha^{-1}\beta^n(\Up{U}{\beta})$ with $\alpha^{-1}\beta^n(\Um{C}{\beta})\leq\Um{C}{\beta}$ and, by~\eqref{eq:n_version}, $\alpha^{-1}\beta^n(\Up{U}{\beta})\cap U = \Up{U}{\beta}$, and it follows that
\begin{equation}
\label{eq:equ3}
\alpha^{-1}\beta^n(U)\cap U = \alpha^{-1}\beta^n(\Um{C}{\beta})\Up{U}{\beta}. \tag{\theequation c}
\end{equation}
Then \eqref{eq:equ2} and \eqref{eq:equ3} imply that the map 
$$
\alpha^{-1}(\Up{U}{\beta})/\Up{U}{\beta} \to \alpha^{-1}\left(\beta^n(U)\cap U\right)/(\alpha^{-1}\beta^n(U)\cap U)
$$ 
given by $x\Up{U}{\beta} \mapsto x(\alpha^{-1}\beta^n(U)\cap U)$ for $x\in\alpha^{-1}(\Up{U}{\beta})$ 
is well-defined and bijective and hence that~\eqref{1st_identity} holds. Towards ~\eqref{2nd_identity}, \eqref{eq:equ1} and \eqref{eq:equ3} imply that the map 
$$
\alpha^{-1}\beta^n(\Um{C}{\beta}) / \beta^n(\Um{C}{\beta}) \to \left(\alpha^{-1}\beta^n(U)\cap U\right) / \left(\beta^n(U)\cap U\right)
$$ 
given by $x\beta^n(\Um{C}{\beta})\mapsto x\left(\beta^n(U)\cap U\right)$ for $x\in \alpha^{-1}\beta^n(\Um{C}{\beta})$ is a well-defined bijection. Theorem~\ref{thm:factors_invariant} and the $\fl{H}$-stability of $N$ imply that $\beta^n(\Um{C}{\beta})$ is tidy for the restriction of $\alpha^{-1}$ to $\con(\beta/N)$. Hence $[\alpha^{-1}\beta^n(\Um{C}{\beta}) : \beta^n(\Um{C}{\beta})]$ is equal to the scale of this restriction and~\eqref{2nd_identity} holds. 

To see the final claim, suppose first of all that $[\alpha^{-1}(\Up{U}{\beta}) : \Up{U}{\beta}]=1$. Then $\alpha^{-1}(\Up{U}{\beta}) = \Up{U}{\beta}$ and hence $\Up{U}{\beta}\leq \Uz{U}{\alpha}$. On the other hand, if $s(\alpha^{-1}|_{\con(\beta/N)}) =1$, then $\alpha^{-1}(\Um{C}{\beta})=\Um{C}{\beta}$ and  $\Uz{U}{\alpha} \geq \Um{C}{\beta} = \con(\beta/N)\cap U$. Both reverse inclusions are clear because $\Uz{U}{\alpha}$ is compact and $\alpha$-stable.
\end{proof}

The next proposition applies to the restriction of $\fl{H}$ to $\con(\alpha/N)$. In this and the following results, it is assumed that $G = \con(\alpha/N)$.
\begin{proposition}
\label{prop:choose_beta}
Suppose that $\fl{H}\leq\Aut(G)$ is flat, $U$ is tidy for $\fl{H}$, and that $G = \con(\alpha/N)$ for some $\alpha\in\fl{H}$ and compact $\fl{H}$-stable subgroup of $N$ with $U\geq N \geq \nub(\fl{H})$. Let $\beta\in \fl{H}$. Then there is $\gamma\in\fl{H}$ such that 
$$
\con(\alpha/N) = \con(\gamma/N)\con(\gamma^{-1}/N)
$$ 
with $\con(\gamma/N) =  \con(\beta/N)$, $\con(\gamma^{-1}/N)\geq \Upp{U}{\beta}$ and $\con(\gamma^{-1}/N) \cap U = \Up{U}{\beta}$.
\end{proposition}
\begin{proof} Denote $\Um{C}{\beta} = \con(\beta/N)\cap U$. Then $U = \Um{C}{\beta}\Up{U}{\beta}$, by Proposition~\ref{prop:U--^con}. Since $\con(\beta/N) = N\con(\beta/\nub(\fl{H}))$, this subgroup is $\fl{H}$-stable by Theorem~\ref{thm:factors_invariant}. Hence $\alpha^{-2}(\Um{C}{\beta})\leq \con(\beta/N)$ and we may choose $n\geq0$ such that $\beta^n\alpha^{-2}(\Um{C}{\beta})\leq \Um{C}{\beta}$. Put $\gamma = \beta^n\alpha^{-1}$. Then, $\gamma^k = \alpha^k(\beta^n\alpha^{-2})^k$ and $\gamma^{-k}=\alpha^k\beta^{-nk}$ modulo the commutator subgroup $[\alpha,\beta]$. Since $[\alpha,\beta]$ stabilises $\Um{C}{\beta}$ and $\Up{U}{\beta}$, by Corollaries~\ref{cor:uniscalar} and~\ref{cor:Flat=>abelian_mod_uniscalar}, we have for each $k\geq0$ that
\begin{equation}
\label{eq:partial_claim}
\gamma^k(\Um{C}{\beta}) \leq \alpha^k(\Um{C}{\beta})
\mbox{ and } \gamma^{-k}(\Up{U}{\beta}) \leq \alpha^k(\Up{U}{\beta}).
\end{equation}
Hence $\Um{C}{\beta}\leq \con(\gamma/N)$ and $\Up{U}{\beta}\leq \con(\gamma^{-1}/N)$, which implies that 
$$
U<  \con(\gamma/N) \con(\gamma^{-1}/N).
$$ 
Then $\alpha^{-k}(U)\leq \con(\gamma/N) \con(\gamma^{-1}/N)$ for every $k\geq0$, because $\con(\gamma/N)$ and $\con(\gamma^{-1}/N)$ are invariant under $\alpha$, by Theorem~\ref{thm:factors_invariant}. Hence, since $U$ is an open neighbourhood of $N$ and $G = \con(\alpha/N)$,
$$
G = \con(\gamma/N)\con(\gamma^{-1}/N).
$$ 

To see how $\con(\gamma/N)$ and $\con(\gamma^{-1}/N)$ relate to $\beta$, observe first  that~\eqref{eq:partial_claim} implies that 
$$
\con(\beta/N)\cap U\leq \con(\gamma/N)\mbox{ and } \Up{U}{\beta} \leq \con(\gamma^{-1}/N)\cap U.
$$ 

Next, consider $x\in U\cap\con(\gamma/N)$ and deduce from the tidiness of $U$ that $x = x_-x_+$ with $x_-\in \Um{C}{\beta}$ and $x_+\in \Up{U}{\beta}$. Then: $x_+\in \con(\gamma/N)$ because $x$ and $x_-$ are; and $x_+\in \con(\gamma^{-1}/N)$ because $U_{\beta+}\leq \con(\gamma^{-1}/N)$. Hence $x_+\in \con(\gamma/N)\cap \con(\gamma^{-1}/N) = N$, whence it follows that $x\in\con(\beta/N)$. We have thus shown that $\con(\gamma/N)\cap U\leq \con(\beta/N)$ and so
$$
\con(\gamma/N)\cap U = \con(\beta/N)\cap U.
$$
That $\con(\gamma/N) = \con(\beta/N)$ follows, because for each $y\in\con(\gamma/N)$ there is $n\geq0$ with $y=\gamma^{-n}(z)$ and $z\in \con(\beta/N)\cap U$, and invariance of $\con(\beta/N)$ under $\gamma^{-1}$ implies that $\gamma^{-n}(z)\in\con(\beta/N)$. Hence $\con(\gamma/N)\leq \con(\beta/N)$. The reverse inclusion holds by interchanging $\beta$ and $\gamma$.

 Finally, consider $x\in U\cap\con(\gamma^{-1}/N)$ and write $x = x_-x_+$ with $x_-\in \Um{C}{\beta}$ and $x_+\in \Up{U}{\beta}$. Then $x_-\in \con(\gamma^{-1}/N)$ because $x$ and $x_+$ are.  Since we already have $x_-\in\con(\gamma/N)$, it follows that $x_-\in N$ and hence that $x\in \Up{U}{\beta}$. Therefore $U\cap\con(\gamma^{-1}/N) = \Up{U}{\beta}$. That $\Upp{U}{\beta}\leq \con(\gamma^{-1}/N)$ folllows, because for each $y\in\Upp{U}{\beta}$ there $n\geq0$ with $y=\beta^n(z)$ and $z\in \Up{U}{\beta}$, and invariance of $\con(\gamma^{-1}/N)$ under $\beta$ implies that $\beta^n(z)\in\con(\gamma^{-1}/N)$.  
\end{proof}

\begin{remark}
\label{rem:choose_beta}	
The conclusion in Proposition~\ref{prop:choose_beta} that $\con(\gamma^{-1}/V)\geq \Upp{U}{\beta}$ cannot be strengthened to equality, see Exercise~\ref{exer:choose_beta}.
\end{remark}

The following factoring of $U$ is shown in the course of the preceding proof.
\begin{corollary}
	\label{cor:choose_beta}
Assume the hypotheses of Proposition~\ref{prop:choose_beta}. Then
$$
U = (U\cap \con(\gamma/N))(U\cap \con(\gamma^{-1}/N)).
$$
\endproof
\end{corollary}
The next result includes the assertion made in  Theorem~\ref{thm:flat_group_structureB3}.\ref{thm:flat_group_structureB3}.
\begin{proposition}
\label{prop:factor_con}
Suppose that $\fl{H}\leq\Aut(G)$ is flat, $N\geq \nub(\fl{H})$ is a compact $\fl{H}$-stable subgroup of $G$, and that $G = \con(\alpha/N)$ for some $\alpha\in\fl{H}$. Then there are $\fl{H}$-invariant subgroups $H_i$, $i\in\{1,\dots,n\}$, with $N\leq H_i$ for each $i$ and such that:
\begin{enumerate}
\item \label{prop:factor_con1}
$G = H_1\dots H_n$ with $n$ no larger than the number of prime factors of $s(\alpha^{-1})$ counted according to multiplicity; 
\item \label{prop:factor_con2}
$H_i = \con(\alpha|_{H_i}/N)$ for each $i\in\{1,\dots,n\}$ and is scaling for $\fl{H}|_{H_i}$;
and 
\item \label{prop:factor_con3}
$s(\alpha^{-1}) = \prod_{i=1}^n s(\alpha^{-1}|_{H_i})$.
\end{enumerate} 
\end{proposition}
\begin{proof}
\ref{prop:factor_con1}. and \ref{prop:factor_con2}. If $\con(\alpha/N)$ is scaling for $\fl{H}$, then the claims hold with $n=1$. 

Suppose that $\con(\alpha/N)$ is not scaling and let $U$ be tidy for $\fl{H}$. It may be supposed, by Proposition~\ref{prop:Nnormal_core_of_tidy}, that $U$ contains $N$. Then $\bigcup_{n\in\mathbb{Z}} \alpha^n(U)$ is equal to $\con(\alpha/N)$ and so there is $\beta\in\fl{H}$ with $\beta(U)\not\geq U$ and $\beta(U)\not\leq U$, in which case $s(\beta)$ and $s(\beta^{-1})$ are both greater than~$1$. By Proposition~\ref{prop:factor_U-},
$$
s(\alpha^{-1}) = [\alpha^{-1}(\Up{U}{\beta}) : \Up{U}{\beta}] s(\alpha^{-1}|_{\con(\beta/N)})
$$
with neither factor being equal to~$1$ unless either  $\Up{U}{\beta}$ or $\con(\beta/N)\cap U$ is contained in $\Uz{U}{\alpha}$. However, $\Uz{U}{\alpha} = N$ and $N$ is $\fl{H}$-stable, which implies that $s(\beta)=1$ if $\Up{U}{\beta}\leq \Uz{U}{\alpha}$ and $s(\beta^{-1})=1$ if $\con(\beta/N)\cap U\leq \Uz{U}{\alpha}$. Therefore neither $\Up{U}{\beta}$ nor $\con(\beta/N)\cap U$ is contained in $\Uz{U}{\alpha}$ and both $[\alpha^{-1}(\Up{U}{\beta}) : \Up{U}{\beta}]$ and $s(\alpha^{-1}|_{\con(\beta/N)})$ are strictly greater than~$1$.

By Proposition~\ref{prop:choose_beta}, there is $\gamma\in\fl{H}$ such that 
\begin{align*}
\con(\beta/N) &= \con(\gamma/N),\\
\Up{U}{\beta} &= \con(\gamma^{-1}/N)\cap U\\
\mbox{ and }\quad G&= \con(\gamma/N)\con(\gamma^{-1}/N).
\end{align*}
Denote $\con(\gamma^{\pm1}/N) = G_{\pm}$. Then $G=G_+G_-$ with $$
s(\alpha^{-1}|_{G_+})=s(\alpha^{-1}|_{\con(\beta/N)})\mbox{ and }s(\alpha^{-1}|_{G_-})=[\alpha^{-1}(\Up{U}{\beta}) : \Up{U}{\beta}].
$$ 
Hence the scale of $\alpha^{-1}$ factors as  
$$
s(\alpha^{-1}) = s(\alpha^{-1}|_{G_+})s(\alpha^{-1}|_{G_-})
$$
with both factors strictly greater than~$1$. By Theorem~\ref{thm:factors_invariant}, $\con(\gamma/\nub(\fl{H}))$ and $\con(\gamma^{-1}/\nub(\fl{H}))$ are both stable under $\fl{H}$ and the respective restrictions of $\fl{H}$ are flat. Since $N$ is compact and $\fl{H}$-stable, both subgroups are normalised by $N$ and $G_\pm = N\con(\gamma^{\pm1}/\nub(\fl{H}))$. It follows that $G_+$ and $G_-$ are both $\fl{H}$-stable and the restrictions of $\fl{H}$ are flat. Since $G_+$ and $G_-$ are contained in $G = \con(\alpha/N)$, we have that $G_\pm = \con(\alpha|_{G_\pm}/N)$. Should $\con(\alpha|_{G_+}/N)$  and $\con(\alpha|_{G_-}/N)$ both be scaling, the proof is complete with $H_1=G_+$ and $H_2=G_-$. 

If $\con(\alpha|_{G_\pm}/N)$ are not both scaling, the argument may be repeated by replacing $G$ with $G_+$ or $G_-$ as the case may be. Since $s(\alpha^{-1})$ is an integer and has only finitely many prime factors, this process terminates after a finite number of steps with $G$ expressed as a product of a finite number of scaling factors. 

\ref{prop:factor_con1}. and \ref{prop:factor_con3}. Proposition~\ref{prop:factor_U-} shows that each division of $G$ into factors corresponds to dividing $s(\alpha^{-1})$ into strictly smaller factors that are the scales of the restrictions of $\alpha^{-1}$ to the factors of $G$. The expression for $s(\alpha^{-1})$ in \ref{prop:factor_con3}. follows, and implies that the number of factors in \ref{prop:factor_con1}. is at most equal to the number of prime divisors of $s(\alpha^{-1})$. 
\end{proof}

The repeated application of Proposition~\ref{prop:choose_beta} in the preceding argument and the observation made in Corollary~\ref{cor:choose_beta} yield the following.
\begin{corollary}
	\label{cor:intersect_U}
	Assume the hypotheses of Proposition~\ref{prop:factor_con} and let $U\geq N$ be tidy for $\fl{H}$. Then 
	$$
	U = (U\cap H_1)(\dots )(U\cap H_n).
	$$
\end{corollary}

\begin{remark}
\label{rem:N_not_changed}
Applying Proposition~\ref{prop:factor_con} with $N=\nub(\fl{H})$, the same $N$ is kept through the proof and statement of the proposition. However, passing from $G$ to a factor $G_+\leq G$ say, it may happen that $\nub(\fl{H}|_{G_+})$ is strictly smaller than $\nub(\fl{H})$ as the Example~\ref{examp:nub_smaller} shows.
\end{remark}

\begin{remark}
\label{rem:permute_factors}
The factors $H_1$, \dots, $H_n$ in Proposition~\ref{prop:factor_con} may be permuted because at each iteration of the argument in the proof we may take $G = G_+G_-$ or $G = G_-G_+$. However, they need not commute or normalise each other. 
\end{remark}

\begin{definition} 
\label{defn:support}
Let $\fl{H}\leq\Aut(G)$ be flat. The \emph{positive support} of $\alpha\in \fl{H}$ is 
$$
\supp^+_{\fl{H}}(\alpha) = \{\roo\in \red(\fl{H}) \mid \roo(\alpha)>0\}
$$ 
and the \emph{support} of $\alpha$ is
$$
\supp_{\fl{H}}(\alpha) = \supp^+_{\fl{H}}(\alpha)\cup \supp^+_{\fl{H}}(\alpha^{-1}).
$$ 
The support of $\alpha$ is denoted by $\supp(\alpha)$ if the flat group $\fl{H}$ is understood. 
For a subset $\fl{K}\subseteq\fl{H}$, the union of the supports of elements of $\fl{K}$ is denoted by $\supp_{\fl{H}}(\fl{K})$.  
\end{definition}
Implicit in the definition is that automorphisms in $\fl{H}_u$ have empty support. 

The scaling factors $H_1$, \dots, $H_n$ in Proposition~\ref{prop:factor_con} are associated with distinct roots of $\fl{H}$, each of which is in $\supp^+_{\fl{H}}(\alpha^{-1})$. That the converse holds is shown next. 
\begin{proposition}
\label{prop:uniqueness}
Let $\fl{H}\leq \Aut(G)$ be flat and let $\alpha\in\fl{H}$. Then the set of roots associated with scaling $\fl{H}$-invariant subgroups, $H_i$, appearing in the product
\begin{equation}
\label{eq:uniqueness}
\con(\alpha/\nub(\fl{H})) = H_1\dots H_n
\end{equation}
of Proposition~\ref{prop:factor_con} is equal to $\supp^+_{\fl{H}}(\alpha^{-1})$. In particular, the subgroups $H_i$ are unique up to permutation and $\supp^+_{\fl{H}}(\alpha^{-1})$ is finite.
\end{proposition}
\begin{proof}
The third part of Proposition~\ref{prop:factor_con} implies that the roots associated with a subgroup $H_i$ lie in $\supp^+_{\fl{H}}(\alpha^{-1})$, and it remains to show that every $\roo\in\supp^+_{\fl{H}}(\alpha^{-1})$ is associated with a subgroup $H_i$ in \eqref{eq:uniqueness}. The proof is by induction on $n$. The claim holds when $n=0$ because $\alpha\in\fl{H}_u$ in this case.

Let $k\geq1$ and assume that the claim holds for all $n < k$, all $\fl{H}\leq \Aut(G)$ with $G$ any \tdlc~group, and all $\alpha\in\fl{H}$. Suppose that 
$$
\con(\alpha/\nub(\fl{H})) = H_1\dots H_k
$$ 
with each $H_i$ scaling for $\fl{H}$. Let $\roo\in\supp^+_{\fl{H}}(\alpha^{-1})$ and denote by $\roo_i$ the root associated with $H_i$. If $\roo=\roo_{k}$ there is nothing to prove and so it may be assumed that $\roo \ne \roo_{k}$, in which case there is $\gamma\in\fl{H}$ such that $\roo(\gamma)>0$ and $\roo_k(\gamma)<0$. Proposition~\ref{prop:choose_beta} shows that it may in fact be assumed that 
$$
\con(\alpha/\nub(\fl{H})) = \con(\gamma/\nub(\fl{H}))\con(\gamma^{-1}/\nub(\fl{H})) = G_+G_-,
$$ 
where we adopt the notation from the proof of Proposition~\ref{prop:factor_con}. Then $\{\gamma^l(x)\}_{i\in\mathbb{Z}}$ is unbounded if $x\not\in\nub(\fl{H})$ and so, for each $i\in\{1,\dots,n\}$, either $\roo_i(\gamma)>0$ or $\roo_i(\gamma)<0$. Hence $G_+$ is the product of those $H_i$ such that $\roo_i(\gamma)<0$ and $G_-$ is the product of those $H_i$ such that $\roo_i(\gamma)>0$, and the condition that $\roo_k(\gamma)<0$ implies that $G_-$ is the product of strictly fewer than $k$ scaling subgroups. On the other hand, the restriction of $\fl{H}$ to $G_-$ is flat, by Theorem~\ref{thm:factors_invariant}, and the condition that $\roo(\gamma)>0$ implies that $\roo\in\supp^+_{\fl{H}}(\gamma^{-1}|_{G_-})$. Therefore, by the induction hypothesis, $\roo$ is associated with one of the scaling factors of $G_-$, and hence with a subgroup $H_i$ appearing in \eqref{eq:uniqueness}.
\end{proof}


\subsection{The structure of flat groups and tidy subgroups}
\label{sec:reduced2}

This section collects together results from previous sections to complete the proofs of Theorems~\ref{thm:flat_group_structureA} and~\ref{thm:flat_group_structureB}. 

We begin with the proof of Theorem~\ref{thm:flat_group_structureA}.
\begin{proposition}
\label{prop:free_abelian}
Let $\fl{H}\leq\Aut(G)$ be flat. Then the map 
$$
R : \fl{H} \to  \mathbb{Z}^{\red(\fl{H})} ; \quad \alpha\mapsto \{\roo(\alpha)\}_{\roo\in\red(\fl{H})} 
$$
is a homomorphism with kernel $\fl{H}_u$ and range contained in $\bigoplus_{\roo\in\red(\fl{H})} \mathbb{Z}$. Hence $\fl{H}/\fl{H}_u$ is a free abelian group. 
\end{proposition}
\begin{proof}
Each root $\roo\in\red(\fl{H})$ is a homomorphism and hence so is $R$. The range of $R$ is contained in $\bigoplus_{\roo\in\red(\fl{H})} \mathbb{Z}$ because $\supp_{\fl{H}}(\alpha)$ is finite for every $\alpha$, by Proposition~\ref{prop:uniqueness}. The kernel of $R$ contains $\fl{H}_u$ because each $\roo$ sends $\fl{H}_u$ to~$0$. If $\alpha\in\fl{H}\setminus\fl{H}_u$, then at least one of $s(\alpha^{-1})$ and $s(\alpha)$ is greater than~$1$ and so there is a root, $\ell$, in $\supp(\alpha)$ such that $\roo(\alpha)\ne0$, by Proposition~\ref{prop:uniqueness}. Hence the kernel of the homomorphism is equal to $\fl{H}_u$. It follows that $\fl{H}/\fl{H}_u$ is isomorphic to a subgroup of the free abelian group $\bigoplus_{\roo\in\red(\fl{H})} \mathbb{Z}$, and is therefore free abelian itself by~\cite[Theorem 7.3]{Lang_Algebra}.  
\end{proof}

\begin{corollary}
\label{cor:free_abelian}
\begin{enumerate}
\item \label{cor:free_abelian1}
If $\fl{H}\leq\Aut(G)$ is flat, then the rank of $\fl{H}/\fl{H}_u$ is at most equal to the number of roots for $\fl{H}$. 
\item \label{cor:free_abelian2}
There is no upper bound for the number of roots in terms of the rank of $\fl{H}/\fl{H}_u$, but the rank is finite if and only if the number of roots is finite. 
\item \label{cor:free_abelian3}
If $\fl{H}/\fl{H}_u$ has infinite rank, then that rank and the number of roots have the same cardinality.  
\end{enumerate}
\end{corollary}
\begin{proof} 
\ref{cor:free_abelian1}.~holds because the rank of a subgroup of a free abelian group is at most the rank of the free group. The lack of an upper bound asserted in~\ref{cor:free_abelian2}.~is demonstrated by the example studied in Exercises~\ref{exer:Qpn}--\ref{exer:rank_weights} in \S\ref{sec:single_automorphism}. 

The claims in \ref{cor:free_abelian2}.~and \ref{cor:free_abelian3}.~about the cardinality of the number of roots follow from \ref{cor:free_abelian1}.~and because each generator of $\fl{H}/\fl{H}_u$ is supported on only finitely many roots, as shown in Proposition~\ref{prop:uniqueness}.
\end{proof}

\begin{definition}
\label{defn:rank_of_flat}
The \emph{rank} of the flat group $\fl{H}\leq\Aut(G)$, denoted $\rk(\fl{H})$, is the rank of the free abelian group $\fl{H}/\fl{H}_u$.
\end{definition}

\begin{remark}
\label{rem:free_abelian}
The embedding $\fl{H}/\fl{H}_u\hookrightarrow\bigoplus_{\roo\in\red(\fl{H})} \mathbb{Z}$ need not be surjective and the exact sequence of this embedding is possibly a better invariant than $\rk(\fl{H})$. Note that the quotient $\bigl(\bigoplus_{\roo\in\red(\fl{H})} \mathbb{Z}\bigr)/R(\fl{H})$ may have torsion, see Exercise~\ref{exer:torsion_quotient} in \S\ref{sec:flatness}. 
\end{remark}

Proposition~\ref{prop:factor_U} shows that each group, $U$, tidy for a finitely generated $\fl{H}$ is a product of finitely many subgroups in a certain order with the subgroups and the order depending on a given sequence $\alpha_1, \dots, \alpha_n\in\fl{H}$. Proposition~\ref{prop:factor_con} shows that $\con(\alpha/\nub{\fl{H}})$ is the product of a finite number of subgroups scaling for $\fl{H}$ in a certain order. A common generalisation of these results is shown in the remainder of this section, namely, that $U$ is a product of subgroups each of which is the intersection with $U$ of a subgroup scaling over $\fl{H}$. The generalisation does extend to all flat groups but only the case when $\red(\fl{H})$ is finite is treated in these notes.

\begin{proposition}
\label{prop:red&support} 
Let $\fl{H}\leq\Aut(G)$ be flat and suppose that $\{\gamma_i\fl{H}_u\}_{i\in I}$ is a generating set for $\fl{H}/\fl{H}_u$. Then $\red(\fl{H}) = \bigcup\{\supp_{\fl{H}}(\gamma_i) \mid i\in I\}$. 
\end{proposition}
\begin{proof}
By definition, $\supp_{\fl{H}}(\gamma_i)$ is a subset of  $\red(\fl{H})$ for each $i\in I$ and so it remains only to show that $\red(\fl{H})\subseteq \bigcup\{\supp_{\fl{H}}(\gamma_i) \mid i\in I\}$. Suppose that $\roo\in\red(\fl{H})$. Then there is $\alpha\in\fl{H}$ such that $\roo(\alpha)>0$ and, since $\{\gamma_i\fl{H}_u\}_{i\in I}$ generates $\fl{H}/\fl{H}_u$, there are $i_1,\dots,i_s\in I$ such that 
$$
\alpha+\fl{H}_u = \gamma_{i_1} + \dots + \gamma_{i_s} + \fl{H}_u.
$$ 
Since it must be that $\roo(\gamma_{i_j})>0$ for at least one $\gamma_{i_j}$, 
$$
\roo\in \supp_{\fl{H}}(\gamma_{i_1})\cup \dots \cup \supp_{\fl{H}}(\gamma_{i_s})
$$ 
and the claim is established.
\end{proof}

For the next lemma, recall from Definition~\ref{defn:nub} that for a flat group $\fl{H}$, the Levi subgroup is $\lev(\fl{H}) = \left\{x\in G\mid \{\alpha(x)\}_{\alpha\in\fl{H}}^-\mbox{ is compact}\right\}$ and this subgroup is closed by Lemma~\ref{lem:lev_closed}.
\begin{lemma}
\label{lem:U_cap_lev}
Let $\fl{H}\leq\Aut(G)$ be flat and suppose that $\fl{K}\triangleleft \fl{H}$. Let $U$ be tidy for $\fl{H}$. Then $\lev(\fl{K})$ is $\fl{H}$-stable, and $\Uz{U}{\fl{K}}$ is a compact open subgroup of $\lev(\fl{K})$ that is tidy for the restriction of $\fl{H}$ to $\lev(\fl{K})$.
\end{lemma}
\begin{proof}
	By definition of flatness, it suffices to show for every $\alpha\in\fl{H}$ that $\lev(\fl{K})$ is stabilised by $\alpha$ and $\Uz{U}{\fl{K}}$ is tidy for $\alpha$.
	
Consider $x\in\lev(\fl{K})$, so that $\{\beta(x) \mid \beta\in\fl{K}\}$ has compact closure. Then
$$
\left\{\beta(\alpha(x)) \mid \beta\in\fl{K}\right\} = \left\{\alpha(\alpha^{-1}\beta\alpha(x)) \mid \beta\in\fl{K}\right\} = \alpha\left(\left\{\beta(x) \mid \beta\in\fl{K}\right\}\right)
$$
because $\fl{K}$ is normal in $\fl{H}$. Since the latter set has compact closure, $\alpha(x)$ belongs to $\lev(\fl{K})$ and $\lev(\fl{K})$ is $\alpha$-stable.

Since $U_{\fl{K}0}$ is compact and $\fl{K}$-stable, it is contained in $\lev(\fl{K})$. Moreover, $U_{\fl{K}0} = U\cap\lev(\fl{K})$ by Lemma~\ref{lem:lev_closed} and is therefore open. Proposition~\ref{prop:functor} shows that $\Uz{U}{\fl{K}}$ is tidy below for $\alpha|_{\lev(\fl{K})}$ and it remains to show that it is tidy above. 

Consider $u\in \Uz{U}{\fl{K}}$. Then $\fl{K}$-stability of $\Uz{U}{\fl{K}}$ implies that $u\in\beta(U)$ for every $\beta\in\fl{K}$ and so, since $\beta(U)$ is tidy for $\alpha$, $u = u_+u_-$ with $u_\pm\in \Upm{\beta(U)}{\alpha}$. In other words, 
$$
C_\beta := \left\{ (u_+,u_-)\in \Upm{\beta(U)}{\alpha} \mid u = u_+u_-\right\}
$$ 
is non-empty for every $\beta\in\fl{K}$, and $C_\beta$ is a closed, and therefore compact, subset of $\beta(U)$. Since the intersection of finitely many tidy subgroups is tidy, $\left\{ C_\beta \mid \beta\in\fl{K}\right\}$ has the finite intersection property. Therefore $\bigcap \left\{ C_\beta \mid \beta\in\fl{K}\right\} \ne\emptyset$. If $(u_+,u_-)$ lies in this intersection, then $u_\pm\in \Uz{U}{\fl{K}}$ and $u_+u_- = u$. Furthermore, $\alpha^{\mp k}(u_\pm)\in \beta(U)$ for every $k\geq0$ and every $\beta\in\fl{K}$, which implies that $u_\pm\in (U_{\fl{K}0})_{\pm}$. Therefore $\Uz{U}{\fl{K}}$ is tidy above for $\alpha$ as claimed.
\end{proof}

The final result in this section completes the proof of Theorem~\ref{thm:flat_group_structureB}.\ref{thm:flat_group_structureB2}. 
\begin{proposition}
\label{prop:factor_tidyU}
Let $\fl{H}\leq\Aut(G)$ be flat and suppose that $\rk(\fl{H})$ is finite. Let $U$ be tidy for $\fl{H}$. Then the roots in $\red(\fl{H})$ may be ordered as $\roo_1$,\dots, $\roo_q$ so that
$$
U = \Uz{U}{\fl{H}}U_{\roo_1} \dots U_{\roo_q} = \Uz{U}{\fl{H}}\prod_{\roo\in\red(\fl{H})} \left(U\cap  \con(\fl{A}_\roo/\nub(\fl{H}))\right).
$$
\end{proposition} 
\begin{proof} 
Since $U_{\roo_i} = \Uz{U}{\fl{H}}\left(U\cap  \con(\fl{A}_{\roo_i}/\nub(\fl{H}))\right)$ for each $\roo_i\in \red(\fl{H})$, by Proposition~\ref{prop:con_Aroo}, and since $\Uz{U}{\fl{H}}$ normalizes $\con(\fl{A}_{\roo_i}/\nub(\fl{H}))$, it suffices to show that $U$ is equal to the second product. 

The proof is by induction on $\rk(\fl{H})$, that is, on the number of generators of $\fl{H}/\fl{H}_u$. The base case when $\rk(\fl{H})=0$ holds because we then have $U = \Uz{U}{\fl{H}}$ and $\red(\fl{H})=\emptyset$. 

Assume that the claim has been shown for all flat $\fl{K}$ with $\rk(\fl{K})\leq k$. Let $\fl{H}$ be flat with rank $k+1$ and suppose that $\gamma_1\fl{H}_u$, \dots, $\gamma_{k+1}\fl{H}_u$ generate $\fl{H}/\fl{H}_u$. Then  Corollary~\ref{cor:Flat=>abelian_mod_uniscalar} shows that $
\fl{K} =\langle\gamma_{k+1},\fl{H}_u\rangle$ is a normal subgroup of $\fl{H}$ and Propositions \ref{prop:U--^con} and \ref{prop:nub_smallest} show that
$$
U = \Uz{U}{\fl{K}}\Up{D}{\gamma_{k+1}}\Um{D}{\gamma_{k+1}}
$$
with $\Upm{D}{\gamma_{k+1}} = U\cap \con(\gamma_{k+1}^{\mp1}/\nub(\fl{H}))$. (We have also used that $\Uz{U}{\fl{K}}$ normalises $\con(\gamma_{k+1}/\nub(\fl{H}))$ and $\con(\gamma_{k+1}^{-1}/\nub(\fl{H}))$.) Lemma~\ref{lem:U_cap_lev} shows that $\Uz{U}{\fl{K}}$ is a compact open subgroup of $\lev(\fl{K})$ and is tidy for $\fl{H}|_{\lev(\fl{K})}$. Since $\gamma_{k+1}$ is uniscalar on $\lev(\fl{K})$, it follows that $\rk(\fl{H}|_{\lev(\fl{K})}) \leq k$ and we may apply the induction hypothesis. Hence the roots in $\red(\fl{H}|_{\lev(\fl{K})})$ may be ordered as $\roo_1$, \dots, $\roo_r$ such that, abbreviating $U\cap \con(\fl{A}_\roo/\nub(\fl{H}))$ by ${C}_\roo$ and denoting $\bigcap\left\{ \alpha|_{\lev(\fl{K})}(\Uz{U}{\fl{K}})\mid \alpha\in\fl{H}\right\}$ by $\widetilde{U}$, 
$$
\Uz{U}{\fl{K}} = \widetilde{U} C_{\roo_1}\dots {C}_{\roo_r}.
$$ 
Then the fact that $\fl{H}|_{\lev(\fl{K})}$ is a quotient of $\fl{H}$ and $\roo_1$, \dots, $\roo_r$ pull back to roots of $\fl{H}$, combined with the identities
\begin{align*}
\Uz{U}{\fl{K}} &= \bigcap\left\{\gamma_{k+1}^n(U)\mid n\in\mathbb{Z}\right\}\\
\mbox{ and }\fl{H} &= \langle\gamma_1,\dots, \gamma_k\rangle\fl{K},
\end{align*}
imply that $\widetilde{U} = \Uz{U}{\fl{H}}$. Hence, noting that 
$$
\Uz{U}{\fl{K}}\cap\con(\fl{A}_\roo/\nub(\fl{H})) = U\cap\con(\fl{A}_\roo/\nub(\fl{H}))
$$ 
if $\roo\in\red(\fl{H}|_{\lev(\fl{K})})$ because $\langle\gamma_{k+1}\rangle\subset \fl{A}_\roo$ in this case, we have that
$$
U = \Uz{U}{\fl{H}}C_{\roo_1}\dots C_{\roo_r}\Up{D}{\gamma_{k+1}}\Um{D}{\gamma_{k+1}}.
$$
Next, Theorem~\ref{thm:factors_invariant} shows that $\con(\gamma_{k+1}^{-1}/\nub(\fl{H}))$ and $\con(\gamma_{k+1}/\nub(\fl{H}))$ are stable under $\fl{H}$ and Propositions~\ref{prop:factor_con} and~\ref{prop:uniqueness} show that these subgroups are products of groups $\con(\fl{A}_\roo/\nub(\fl{H}))$ with $\roo\in\red(\fl{H})$ in the support of $\gamma_{k+1}$. Writing these roots in order $\roo_{r+1}$, \dots, $\roo_q$ according to the order they appear as factors in $\con(\gamma_{k+1}^{-1}/\nub(\fl{H}))$ and $\con(\gamma_{k+1}/\nub(\fl{H}))$ respectively and abbreviating $U\cap \con(\fl{A}_\roo/\nub(\fl{H}))$ by ${C}_\roo$, yields 
$$
\Up{D}{\gamma_{k+1}}\Um{D}{\gamma_{k+1}} = C_{\roo_{r+1}}\dots C_{\roo_q}
$$
by Corollary~\ref{cor:intersect_U}. Hence 
$$
U = \Uz{U}{\fl{H}} C_{\roo_1}\dots C_{\roo_q}
$$
with $\roo_1$, \dots, $\roo_q$ the roots of $\fl{H}$. (Note that every root in $\red(\fl{H})$ appears: those in $\supp_{\fl{H}}(\gamma_{k+1})$ are in the list $\roo_{r+1}$, \dots, $\roo_q$, and those in the complement of $\supp_{\fl{H}}(\gamma_{k+1})$ are in the list $\roo_1$, \dots, $\roo_r$.) Therefore the claim of the induction holds with $\rk(\fl{H})=k+1$.
\end{proof}

\begin{remark}
\label{rem:factor_con}
A weaker version of Proposition~\ref{prop:factor_con} was shown in~\cite{SimulTriang} by using the same bisection argument as used above to prove Proposition~\ref{prop:factor_con}. The present proof highlights this argument more than is done in~\cite{SimulTriang}, produces the statement about relative contraction groups, and shows more explicitly how the order of the factors $C_{\roo_1}$, \dots, $C_{\roo_q}$ emerges from order of the generators $\gamma_1$, \dots, $\gamma_{k+1}$. This approach may also be used to see how the tidy subgroup $U$ factors when $\fl{H}/\fl{H}_u$ is not finitely generated but it must first be said what an infinite product of non-commuting subgroups means, and that is beyond the scope of these notes.
\end{remark}

\subsection{The lower nub for a  flat group}
\label{sec:Levi_flat}

The nub subgroup for an automorphism $\alpha\in\Aut(G)$ is simultaneously the largest subgroup of $G$ contained in all subgroups tidy for $\alpha$, by Theorem \ref{thm:nub}, and the smallest $\alpha$-stable subgroup, $C$, such that $\con(\alpha/C)$ is closed, by Proposition \ref{prop:nub_smallest}. The nub subgroup, $\nub(\fl{H})$ for flat $\fl{H}$ is the largest group contained in all subgroups tidy for $\fl{H}$. However, although $\con(\alpha/\nub(\fl{H}))$ is closed for every $\alpha\in\fl{H}$, $\nub(\fl{H})$ might not be the smallest $\fl{H}$-invariant group with this property. 

The nub of a flat group of automorphisms was first studied by Colin Reid in~\cite{Reid_DynamicsNYJ_2016}. He also introduced what he called the \emph{lower nub}, denoted by $\lnub(\fl{H})$, which is the closed subgroup generated by all groups $\nub(\alpha)$ with $\alpha\in\fl{H}$. Since $\beta(\nub(\alpha)) = \nub(\beta\alpha\beta^{-1})$ and $\fl{H}$ normalises itself, $\lnub(\fl{H})$ is $\fl{H}$-stable. Since a subgroup tidy for $\fl{H}$ is tidy for every $\alpha\in\fl{H}$, Theorem~\ref{thm:nub} implies the following.
\begin{proposition}
\label{prop:nubinnub}
Suppose that $\fl{H}\leq\Aut(G)$ is flat. Then 
$$
\langle \nub(\alpha) \mid\alpha\in\fl{H}\rangle \leq \nub(\fl{H}).
$$
Hence $\lnub(\fl{H})\leq \nub(\fl{H})$ and $\lnub(\fl{H})$ is compact. 
\end{proposition}
Since $\lnub(\fl{H})$ is compact and $\fl{H}$-invariant, the contraction group for $\alpha$ modulo $\lnub(\fl{H})$, see Definition~\ref{defn:relative_contraction}, is well-defined. 
\begin{corollary}
\label{cor:nubinnub} 
The contraction group $\con(\alpha/\lnub(\fl{H}))$ is closed for every $\alpha\in\fl{H}$ and $\lnub(\fl{H})$ is the smallest closed group with this property. Furthermore, $\con(\alpha/\nub(\fl{H}))$ is closed for every $\alpha\in\fl{H}$.
\endproof
\end{corollary}

Example~4.1 in~\cite{Reid_DynamicsNYJ_2016} shows that the inclusion $\lnub(\fl{H})\leq \nub(\fl{H})$ in Proposition~\ref{prop:nubinnub} may be strict. The flat group in that example is abelian and not finitely generated. Here are a couple of additional examples. The first takes a different approach and gives a finitely generated flat group, while the second is based on the idea in~\cite{Reid_DynamicsNYJ_2016} but incorporates additional features which illustrate arguments in earlier sections. 
\begin{example}
\label{examp:lnub_not_nub}
Let $\Gamma$ be a finitely generated (discrete) group and $F$ be a finite group. Then $G:=F^{\Gamma}$ is a compact group and the translation action embeds $\Gamma$ as a flat subgroup of $\Aut(G)$. Since the action of $\Gamma$ on itself by translation is transitive, $G$ itself is the smallest subgroup tidy for $\Gamma$ and so $\nub(\Gamma)= G$. However, if every element of $\Gamma$ has finite order, as constructed  in~\cite{Adian,Golod,Grigorchuk_Burnside} for instance, then $\nub(\gamma) = \triv$ for every $\gamma\in\Gamma$ and hence $\lnub(\Gamma) = \triv$. 
\end{example}
\begin{example}
\label{examp:lnub_not_nub2}
Let $\mathbb{F}_p$ be the additive group of the field of order~$p$ and let $B = \mathbb{F}_p(\!(t)\!) \times \mathbb{F}_p^{\mathbb{Z}}$ be the abelian group that is the direct product of the additive group of the field of formal Laurent series over $\mathbb{F}_p$ and the product over $\mathbb{Z}$ of copies of $\mathbb{F}_p$. Denote elements of $\mathbb{F}_p(\!(t)\!)$ by $f = \sum_{m\in\mathbb{Z}} f(m)t^m$. For each $n\in\mathbb{Z}$, let $u_n\in \mathbb{F}_p^{\mathbb{Z}}$ be given by $u_n(i) = \begin{cases} 1, & \mbox{ if }i=n\\ 0, & \mbox{ otherwise}\end{cases}$. Define, for each $m,n\in\mathbb{Z}$, an automorphism $\beta_{m,n}$ of $B$ by 
$$
\beta_{m,n}(f,g) = (f,g+f(m)u_n), \quad f\in \mathbb{F}_p(\!(t)\!),\ g\in\mathbb{F}_p^{\mathbb{Z}}.
$$
Then $\fl{B} := \langle\beta_{m,n} \mid m,n\in\mathbb{Z}\rangle$ is a flat subgroup of $\Aut(B)$ with every compact open subgroup of $B$ containing $\{0\}\times \mathbb{F}_p^{\mathbb{Z}}$ being invariant under $\fl{B}$ and therefore tidy. Every subgroup tidy for $\fl{B}$ must contain $\{0\}\times \mathbb{F}_p^{\mathbb{Z}}$. Hence $\nub(\fl{B}) = \{0\}\times \mathbb{F}_p^{\mathbb{Z}}$. On the other hand, since every element of $\fl{B}$ has order $p$, $\nub(\beta) = \{0\}\times\{0\}$ for every $\beta\in\fl{B}$ and so $\lnub(\fl{B}) = \{0\}\times\{0\}$.

Next, define $\tau\in\Aut(B)$ by $\tau(f,g) = (tf,g)$, so that $\tau$ is the shift on the first coordinate and the identity on the second. Then $\fl{H}:= \langle \tau, \fl{B}\rangle$ is a flat subgroup of $\Aut(B)$ with $t^m\mathbb{F}_p[\![t]\!]\times \mathbb{F}_p^{\mathbb{Z}}$ tidy for $\fl{H}$ every $m\in\mathbb{Z}$. Hence $\nub(\fl{H}) = \{0\}\times \mathbb{F}_p^\mathbb{Z}$. Once again we have that $\nub(\beta) =  \{0\}\times\{0\}$ for every $\beta\in\fl{H}$. 

Further define $\sigma\in\Aut(B)$ by $\sigma(f,g) = (f,g^{(1)})$, where $g^{(1)}(n) = g(n-1)$ is the shift of $g$. Then $\fl{K}:= \langle \sigma, \tau, \fl{B}\rangle$ is flat but now we have that $\nub(\sigma) = \{0\}\times \mathbb{F}_p^{\mathbb{Z}}$ and $\nub(\fl{K}) = \lnub(\fl{K})$. 
\end{example}

\begin{example}
	\label{examp:nub_smaller} 
	Let 
	$$
	B = \mathbb{F}_p(\!(t)\!) \times \mathbb{F}_p^{\mathbb{Z}},\ \fl{B} = \langle\beta_{m,n} \mid m,n\in\mathbb{Z}\rangle\mbox{ and }\fl{H} = \langle \tau, \fl{B}\rangle\leq\Aut(B)
	$$ 
	be as in Example~\ref{examp:lnub_not_nub2}. Put $G = B_1\times B_2$ with $B_i = B$ for each $i$ and $\fl{H}^{(1)} = \fl{H}\times \{\id\}$ and $\fl{H}^{(2)} = \{\id\}\times\fl{H}$, so that $\fl{H}^{(1)} \times \fl{H}^{(2)} \leq \Aut(G)$. Then $\fl{H}^{(1)} \times \fl{H}^{(2)}$ is flat because $\fl{H}$ is and 
	$$
	\nub(\fl{H}^{(1)} \times \fl{H}^{(2)}) = \nub(\fl{H}^{(1)}) \times \nub(\fl{H}^{(2)}) = (\{0\} \times \mathbb{F}_p^{\mathbb{Z}})\times(\{0\} \times \mathbb{F}_p^{\mathbb{Z}}).
	$$ 
	Furthermore, $\con((\tau,\tau)/\nub(\fl{H}^{(1)} \times \fl{H}^{(2)})) = B_1\times B_2 = G$. Choosing $\beta = (\tau,\tau^{-1})$, yields that
	$$
	\con(\beta/\nub(\fl{H}^{(1)} \times \fl{H}^{(2)})) = \left(\mathbb{F}_p(\!(t)\!) \times \mathbb{F}_p^{\mathbb{Z}}\right)\times\left(\{0\} \times \mathbb{F}_p^{\mathbb{Z}}\right).
	$$ 
	Then the restriction of $\fl{H}^{(1)} \times \fl{H}^{(2)}$ to $\con(\beta/\nub(\fl{H}^{(1)} \times \fl{H}^{(2)}))$ has nub $\left(\{0\} \times \mathbb{F}_p^{\mathbb{Z}}\right)\times\left(\{0\} \times \{0\}\right)$, which is strictly smaller than $\nub(\fl{H}^{(1)} \times \fl{H}^{(2)})$.
\end{example}



%
%
\newpage
\section{Exercises}

\subsection{Minimising subgroups and the scale for single endomorphisms}
\label{sec:single_automorphism}

\begin{enumerate}

	\item \label{exer:not_tidy}
	Let $G = \bigoplus_{n\in\mathbb{Z}} F$ with $F$ a finite group and equip $G$ with the discrete topology. Let $\alpha\in\Aut(G)$ be the shift, {\it i.e.\/}
	$$
	\alpha f(n) = f(n+1), \quad (f\in G,\ n\in\mathbb{Z}).
	$$
	Define $U = \left\{ f\in G \mid f(n)=\id_F \mbox{ if } n<0\mbox{ or } n>5\right\}$. 
	\begin{enumerate}
		\item Explain why $U\in\COS(G)$.
		\item Show that $\alpha^m(U)\cap\alpha^n(U)  = \bigcap_{k=m}^n \alpha^k(U)$ for all $m\leq n\in\mathbb{Z}$.
\item Show that $U$ is not tidy above for $\alpha$. 
	\end{enumerate}

	\item \label{exer:not_tidy2}
 Let $G = F^{\mathbb{Z}}$ with $F$ a finite group and equip $G$ with the product topology. 

Let $\sigma\in \Aut(G)$ be the shift, {\it i.e.\/} 
$$
(\sigma f)(n) = f(n+1), \quad f\in F^{\mathbb{Z}},\ n\in\mathbb{Z}.
$$
Let $U_1 = \left\{ f\in F^{\mathbb{Z}} \mid f(0)=\id_F\right\}$ and $U_2 = \left\{ f\in F^{\mathbb{Z}} \mid f(0)=\id_F = f(2)\right\}$.
\begin{enumerate}
	\item Compute the indices $[\sigma(U_i) : \sigma(U_i)\cap U_i]$ for $i=1,2$.
	\item Are $U_1$ and $U_2$ tidy above? tidy below for $\sigma$? 
	\item What is the value of $s(\sigma)$? Find a subgroup tidy for $\sigma$.
\end{enumerate}

\item \label{ex:comparable_tidy}
Let $\alpha$ be an automorphism of $G$ and suppose that there is $U\in\COS(G)$ such that $\alpha(U)\geq U$. Show that $U$ is tidy for $\alpha$. Suppose, on the other hand, that there is $U\in\COS(G)$ with $\alpha(U)\leq U$. Show that $U$ is tidy in this case.

\item \label{ex:comparable_not_tidy}
Let $G = F^{\mathbb{N}}$ with $F$ a finite group and let 
$$
(\alpha f)(n) = f(n+1),\quad f\in F^{\mathbb{N}},\ n\in\mathbb{N},
$$ 
be the shift endomorphism.
\begin{enumerate}
	\item What is $\ker\alpha$?
	\item What is $s(\alpha)$?
	\item Let $U = \left\{f\in G \mid f(0)=f(1)=\id_F\right\}$. Show that $\alpha(U)\geq U$ and find $[\alpha(U):\alpha(U)\cap U]$. Observe that $U$ is not minimising for $\alpha$.
	\item Show that $U$ is tidy above for $\alpha$.
	\item Determine $\Umm{U}{\alpha}$ and deduce that $U$ is not tidy below for $\alpha$. 
\end{enumerate}

\item 
\label{exer:endo_tidy_below}
Let $\alpha\in\End(G)$. Let $n\in\mathbb{N}$ and suppose that $U$ is tidy for $\alpha$. 
\begin{enumerate}
	\item Show that $\Um{U}{\alpha^n} = \Um{U}{\alpha}$. (It may help to use Proposition~\ref{prop:tidy_criteria}.)
	\item Show that $\Umm{U}{\alpha^n} = \Umm{U}{\alpha}$ and deduce that $U$ is tidy below for $\alpha^n$.
\end{enumerate}

\item 
\label{exer:scale_prop1}
Let $\alpha\in\Aut(G)$ and $U\in \COS(G)$. 
\begin{enumerate}
	\item Show that $[U : \alpha(U)\cap U] = [\alpha^{-1}(U) : \alpha(U)^{-1}\cap U]$.
\end{enumerate}
Let $m$ be a left-invariant Haar measure on $G$. Use the identity 
$$
\frac{m(\alpha(U))}{m(U)} = \frac{m(\alpha(U))}{m(\alpha(U)\cap U))}\times \frac{m(\alpha(U)\cap U))}{m(U)}.
$$
to deduce the following.
\begin{enumerate}
	\setcounter{enumii}{1}
	\item \label{exer:scale_prop12}
	$[\alpha(U) : \alpha(U)\cap U]$ is minimised at $U$ if and only if \\
	$[\alpha^{-1}(U) : \alpha(U)^{-1}\cap U]$ is minimised at $U$, and hence that $U$ is tidy for $\alpha$ if and only if it is tidy for $\alpha^{-1}$.
	\item If $\Delta : \Aut(G) \to \mathbb{R}^+$ is the modular function on $\Aut(G)$, then $\Delta(\alpha) = s(\alpha)/s(\alpha^{-1})$.
\end{enumerate}

\item 
\label{exer:scale_prop2}
Let $\alpha\in\Aut(G)$.
\begin{enumerate}
	\item Show that, if $U$ is tidy for $\alpha$, then $U$ is tidy for $\alpha^n$ with $n\in\mathbb{Z}$.
	\item Show that, for $n\in\mathbb{N}$, $s(\alpha^n) = s(\alpha)^n$.
\end{enumerate}

\item 
\label{exer:power_tidy_not}
Let $F$ be a finite group and put
$$
G = \left\{f\in F^\mathbb{Z}\mid \exists N\in\mathbb{Z} \mbox{ such that }f(n)=\id_F\mbox{ if }n\leq N\right\}.
$$
Equip $G$ with the group topology such that $U_N :=\left\{f\in G\mid f(n)=\id_F\mbox{ if }n\leq N\right\}$ is a compact open subgroup of $G$. (Thus $U_N \cong F^{[N,\infty)}$ with the product topology.) Define
$$
(\alpha f)(n) = f(n+1),\quad f\in G, n\in\mathbb{Z}.
$$
\begin{enumerate}
	\item Show that $\alpha\in\Aut(G)$.
	\item Let $U = \left\{f\in G\mid f(n)=\id_F\mbox{ if }n<0\mbox{ or }n=2\right\}$. Show that $U$ is tidy for $\alpha^3$ but not for $\alpha$.
\end{enumerate}

\item \label{exer:tree}
Let $G = \Isom(\tree_{q+1})$ be the isometry group of the regular tree in which every vertex has $q+1$ neighbours and equip $G$ with the permutation topology, or equivalently, the topology of convergence on finite sets of vertices. 

Let $x$ be a translation of $\tree_{q+1}$ by distance $d$ along an axis $\ell\subset V(\tree_{q+1})$. Denote the vertices on $\ell$ by $v_n$, $n\in\mathbb{Z}$, with $v_{n+1}$ and $v_{n-1}$ the neighbours of $v_n$ on the path. Then $x.v_n = v_{n-d}$ for all $v_n\in\ell$. Let $\alpha_x$ be the inner automorphism 
$$
\alpha_x : y\mapsto xyx^{-1}, \quad y\in \Isom(\tree_{q+1}).
$$ 
Suppose that $w_0,w_1\in V(\tree_{q+1})$ be neighbours of $v_0$ and $v_1$ respectively and not on $\ell$.  

Let $U_1 = \stab_G(v_0)$, $U_2 = \stab_G(v_0)\cap \stab_G(v_1)$, $U_3 = \stab_G(w_0)$ and $U_4 = \stab_G(w_0)\cap \stab_G(w_1)$ be compact open subgroups of $G$.
\begin{enumerate}
\item Compute the indices $[\alpha_x(U_i) : \alpha_x(U_i)\cap U_i]$ for $i=1,2,3,4$.
\item Which of $U_1$, $U_2$, $U_3$ and $U_4$ are tidy above? tidy below for $\alpha_x$? 
\item What is the value of $s(\alpha_x)$? Find a subgroup tidy for $\alpha_x$.
\end{enumerate}

\item \label{exer:SL2}
Let $G = SL(2,\mathbb{Q}_p)$ and equip $G$ with the subspace topology as a subset of $M(2,\mathbb{Q}_p)\cong \mathbb{Q}_p^4$. ($\mathbb{Q}_p$ denotes the field of $q$-adic numbers and $\mathbb{Z}_p$ the subring of $p$-adic integers.)

Let $\alpha$ be the automorphism of conjugation by the matrix $\left(\begin{array}{cc} p & 0 \\ 0 & 1\end{array}\right)$. 

Let $U_1 = \left\{ \left(\begin{array}{cc} a & b \\ c & d\end{array}\right)\in SL(2,\mathbb{Q}_p) \mid a,b,c,d\in \mathbb{Z}_p \right\}$ and \\
$U_2 = \left\{ \left(\begin{array}{cc} a & pb \\ c & d\end{array}\right)\in SL(2,\mathbb{Q}_p) \mid a,b,c,d\in \mathbb{Z}_p \right\}$ be compact open subgroups of $G$. 
\begin{enumerate}
\item Compute the indices $[\alpha(U_i) : \alpha(U_i)\cap U_i]$ for $i=1,2$.
\item Which of $U_1$ and $U_2$ are tidy above? tidy below for $\alpha$? 
\item What is the value of $s(\alpha)$? Find a subgroup tidy for $\alpha$.
\item Let $\beta$ be the automorphism of conjugation by the matrix $\left(\begin{array}{cc} 0 & p \\ 1 & 0\end{array}\right)$. Answer the same questions for $\beta$.
\end{enumerate}

\item  \label{exer:orbit}
Suppose that $U$ is tidy for the endomorphism $\alpha$ and let $x\in \alpha(U_+)\setminus U_+$. Show that $\alpha^n(x)\in \alpha^{n+1}(U_+)\setminus \alpha^n(U_+)$ and deduce that $\{\alpha^n(x)\}_{n\in\mathbb{N}}$ has no accumulation point. 

\item Using Exercise~\ref{exer:orbit} or otherwise, prove Proposition 1.1.5.

\item Using Exercise~\ref{exer:orbit} or otherwise, prove Proposition 1.1.6.

\item Prove Lemma 1.1.8.

\item Let $U$ be a compact open subgroup of $G$ and let $\alpha\in\Aut(G)$. 
\begin{enumerate}
\item Show that $\Delta(\alpha) = \frac{[\alpha(U):\alpha(U)\cap U]}{[U:\alpha(U)\cap U]} = \frac{[\alpha(U):\alpha(U)\cap U]}{[\alpha^{-1}(U):\alpha^{-1}(U)\cap U]}$.
\item Deduce that $U$ minimises $[\alpha(U):\alpha(U)\cap U]$ if and only if it minimises $[\alpha^{-1}(U):\alpha^{-1}(U)\cap U]$.
\item Show that $\Delta(\alpha) = s(\alpha)/s(\alpha^{-1})$.
\item Show that $s(\alpha)=1 = s(\alpha^{-1})$ if and only if there is $U\in \COS(G)$ such that $\alpha(U) = U$.
\item Show that $s(\alpha^n) = s(\alpha)^n$ for every $n\in\mathbb{N}$. 
\end{enumerate}

\item Let $\mbox{Aff}(P)$ be the group of affine motions of the plane, {\it i.e.\/} the group of translations, rotations and reflections of the plane and equip $\mbox{Aff}(P)$ with the compact-open topology. Equivalently, $\mbox{Aff}(P) = \mathbb{R}^2\rtimes O(2)$. Show that every translation is a limit of rotations.

\item \label{ex:tidymeet}
Let $G = \mathbb{F}_p(\!(t)\!) \times \mathbb{F}_p(\!(t^{-1})\!)$ and define $\alpha\in\Aut(G)$ by 
$$
\alpha(f(t),g(t^{-1})) = (tf(t),tg(t^{-1})).
$$ 
\begin{enumerate}
\item Show that $U^{(m,n)} := t^m\mathbb{F}_p[\![t]\!] \times t^n\mathbb{F}_p[\![t^{-1}]\!]$ is tidy for $\alpha$ for every $m,n\in\mathbb{Z}$. What are $\Up{U^{(m,n)}}{\alpha}$ and $\Um{U^{(m,n)}}{\alpha}$?
\item Let $H = \left\{ (f(t),g(t^{-1})) \in G \mid f_n=g_n\mbox{ for every }n\in\mathbb{Z}\right\}$. Show that $H$ is a closed $\alpha$-invariant subgroup of $G$. 
\item Show that $H$ is discrete and deduce that $s(\alpha|_H)=1$.
\item Show that $U^{(-3,3)}\cap H$ is tidy below for $\alpha|_H$ but not tidy above. 
\item Show that $G/H\cong \mathbb{F}_p^{\mathbb{Z}}$ and that $\alpha$ induces the shift on $\mathbb{F}_p^{\mathbb{Z}}$.
\item Show that $U^{(-3,3)}H/H$ is tidy above for the shift on $\mathbb{F}_p^{\mathbb{Z}}$ but not tidy below. 
\end{enumerate}

\item \label{ex:P(G)_not_closed}
Let $G = \mathbb{R}^2\rtimes SO(2)$, where $SO(2)$ is the special orthogonal group with its natural action on $\mathbb{R}^2$. Denote elements of $\mathbb{R}^2$ by $\xi$ and the rotation through $\theta$ in $SO(2)$ by $R_\theta$.
\begin{enumerate}
	\item Show that, if $g = (\xi,R_\theta)\in G$ with $\theta\ne0$, then $\{g^n\}_{n\in\mathbb{Z}}$ has compact closure in $G$.
	\item Deduce that $P(G) = \left\{(\xi,R_\theta)\in G \mid \theta\ne0\right\}$ and hence that $P(G)$ is open and dense in $G$. Conclude that $P(G)$ is not closed.
	\item For $\xi\in\mathbb{R}^2$ and $R_\theta\in SO(2)$, let 
	$$
	(\xi,R_0)(0,R_\theta)(\xi,R_0)^{-1} = (R_\theta(\xi),R_\theta).
	$$ 
	Show that, putting $\xi_\theta = (1/\theta,0)$, we have that $(R_\theta(\xi),R_\theta)$ is a rotation and
	$$
	(R_\theta(\xi),R_\theta)\to ((0,-1),R_0)\mbox{ as }\theta\to0.
	$$ 
Conclude directly that $P(G)$ is not closed.
\end{enumerate} 

\subsection{Subgroups associated with automorphisms}

\item\label{ex:con_not_closed}
Let $G = F^{\mathbb{Z}}$ with $F$ a finite group and let 
$$
(\alpha f)(n) = f(n+1),\quad f\in F^{\mathbb{Z}},\ n\in\mathbb{Z},
$$ 
be the shift automorphism. Determine $\con(\alpha)$ and observe that it is not a closed subgroup of $G$. 

\item \label{ex:con_tree}
Let $\alpha_x$ be the automorphism of $\Isom(\tree_{q+1})$ defined in Exercise~\ref{exer:tree}. Determine $\con(\alpha_x)$ and $\nub(\alpha_x)$. Deduce that $\con(\alpha_x)$ is not closed.

\item  \label{ex:con_matrix}
Let $\alpha$ be the automorphism of $SL(2,\mathbb{Q}_p)$ defined in Exercise~\ref{exer:SL2}. 
\begin{enumerate}
\item For each $n\geq1$, let $U_n = \left\{ (a_{ij})\in SL(2,\mathbb{Z}_p) \mid (a_{ij})\equiv (\delta_{ij}) \mod p^n\right\}$. Show that $U_n$ is tidy for $\alpha$.
\item Deduce that $\nub(\alpha) = \{\id\}$ and hence that $\con(\alpha)$ is closed. 
\end{enumerate}

\item Show that, if $G$ is a totally disconnected, locally compact group and there is $\alpha\in\Aut(G)$ for which there is no proper, non-empty open $\alpha$-invariant $\mathscr{O}\subset G$, then $G$ is compact. 

Give an example of a totally disconnected compact group $G$ and automorphism $\alpha$ that is topologically ergodic, {\it i.e.\/}, if $\mathscr{O}\subset G$ is open and $\alpha$-invariant, then $\mathscr{O} = G$ or $\mathscr{O} = \emptyset$.

\item Let $G$ be a totally disconnected compact group and $\alpha\in\Aut(G)$. Show that the action of $\alpha$ on $G$ is ergodic if and only if $G$ itself is the only subgroup tidy for $\alpha$.

\subsection{The Tree Representation Theorem}

\item Let $G$ be a totally disconnected locally compact group and $\alpha\in\Aut(G)$. Suppose that there is $U\in\COS(G)$ such that 
$$
\alpha(U) > U\mbox{ and }\bigcup_{n\in\mathbb{Z}} \alpha^n(U) = G.
$$ 
\begin{enumerate}
\item Verify that $U$ is tidy above and below for $\alpha$ and hence show that $s(\alpha) = [\alpha(U) : U]$. 
\item Define a graph, $\Gamma$,  with vertex set $V(\Gamma) = (G\rtimes \langle\alpha\rangle)/U$ and edge set $E(\Gamma) = \left\{ \{x\alpha^nU,x\alpha^{n+1}U\}\mid x\in G,\ n\in\mathbb{Z}\right\}$. Show that $\Gamma$ is a regular tree in which every vertex has degree $s(\alpha)+1$.
\item Let $G\rtimes \langle\alpha\rangle$ act on $V(\Gamma)$ by left translation of cosets. Show that this action represents $G\rtimes \langle\alpha\rangle$ as a group of tree automorphisms. 
\item Verify that the representation of $G\rtimes \langle\alpha\rangle$ in $\Isom(\Gamma)$ in the previous part has closed range and has kernel equal to the largest compact normal subgroup of $G\rtimes \langle\alpha\rangle$.
\item Verify that the action of $G\rtimes \langle\alpha\rangle$ in $\Gamma$ fixes the end containing the ray $\{\alpha^nU\}_{n\leq0}\subset V(\Gamma)$ and that the stabiliser of the vertex $U\in V(\Gamma)$ is the subgroup $U\leq G\rtimes \langle\alpha\rangle$.
\end{enumerate}
\end{enumerate}



 \newpage
 
\subsection{Flat Groups of Automorphisms}
\label{sec:flatness}

\begin{enumerate}

\item \label{exer:singly_generated_flat}
Show that $\langle\alpha\rangle$ is a flat group for any $\alpha\in\Aut(G)$.

\item\label{exer:irreducible}
Let $G = \mathbb{Q}_p\times F$ with $F$ a non-trivial finite group and let $\alpha = (p,\id)$ be the automorphism that multiplies $\mathbb{Q}_p$ by $p$ and is the identity on $F$. Let $U = \mathbb{Z}_p\times F$. 
\begin{enumerate}
	\item Show that $\alpha(U)<U$ and that $G = \bigcup_{n\in\mathbb{Z}} \alpha^n(U)$ and deduce that $G$ is irreducible over $\langle\alpha\rangle$.
	\item Let $V = \mathbb{Z}_p\times\triv$. Show that $V$ is tidy for $\langle\alpha\rangle$ but that $\bigcup_{n\in\mathbb{Z}} \alpha^n(V)$ is not equal to $G$. 
\end{enumerate}

\item \label{exer:matrix_flat}
Let $G = SL(3,\mathbb{Q}_p)$ and let $\fl{H}$ be the group of automorphisms given by conjugating by diagonal matrices (not necessarily having determinant~$1$). 
\begin{enumerate}
\item Show that the subgroup 
$$
U = \left\{ \left(a_{ij}\right)_{i,j\in\{1,2,3\}} \mid a_{ij}\in\mathbb{Z}_p\mbox{ if }i\geq j\mbox{ and }a_{ij}\in p\mathbb{Z}_p\mbox{ if }i<j\right\}
$$
is tidy for $\fl{H}$.
\item Deduce that $\fl{H}$ is flat.
\item Determine the uniscalar subgroup $\fl{H}_u$.
\end{enumerate}

\item Let $G=\Isom(\tree_{q+1})$ and $\{v_n\}_{n\in\mathbb{Z}}$ and $\{w_n\}_{n\in\mathbb{Z}}$ be two bi-infinite paths in $\tree_{q+1}$ such that $w_0=v_0$ and $w_1=v_1$ and $w_n\ne v_n$ otherwise. Thus the paths have exactly one edge in common. Let $\alpha$ and $\beta$ be translations of $\tree_{q+1}$ such that $\alpha(v_n) = v_{n+1}$ and $\beta(w_n) = w_{n+1}$. 
\begin{enumerate}
\item Show that $\stab_G(v_0)\cap\stab_G(v_1)$ is tidy for both $\alpha$ and $\beta$.
\item Show that $\stab_G(v_2)\cap\stab_G(v_3)$ is not tidy for $\beta$.
\item By appealing to Proposition~2.1.2, deduce that $\langle\alpha,\beta\rangle$ is not flat.
\end{enumerate}

\item What is the set of homomorphisms $\red(\fl{H})$ and what are the subgroups $\widetilde{U}_\rho$, $\rho\in\red(\fl{H})$, that are described in Theorem~2.1.4 for the flat group in Exercise~\ref{exer:singly_generated_flat}.?

\item 
Determine $\nub(\fl{H})$ and $\lev(\fl{H})$ for the flat group $\fl{H}$ in Exercise~\ref{exer:matrix_flat}.

\item \label{exer:Qpn}
Let $G = \mathbb{Q}_{\fp}^n$ and for each $\mathbf{a} = (a_i)\in (\mathbb{Q}_{\fp}\setminus\{0\})^n$ define 
$$
\alpha_{\mathbf{a}}: \mathbb{Q}_{\fp}^n\to \mathbb{Q}_{\fp}^n; (\xi_i) \mapsto (a_i\xi_i).
$$
Let $\fl{H} = \left\{ \alpha_{\mathbf{a}}\mid \mathbf{a} \in (\mathbb{Q}_{\fp}\setminus\{0\})^n\right\}$. 
\begin{enumerate}
\item Show that $\mathbb{Z}_{\fp}^n$ is tidy for $\fl{H}$ and deduce that $\fl{H}$ is flat.
\item Deduce that $\nub(\fl{H})$ is the trivial subgroup.
\end{enumerate}

\item \label{exer:vector_examp}
Let $\alpha,\beta\in\Aut(\mathbb{Q}_{\fp}^n)$ be given by
$$
\alpha(\xi_i) = ({\fp}\xi_i)\mbox{ and } \beta(\xi_i) = ({\fp}^{i-2}\xi_i),\quad (\xi_i)\in \mathbb{Q}_{\fp}^n.
$$
Then $\langle\alpha,\beta\rangle$ is a subgroup of the flat group from Exercise~\ref{exer:Qpn}. Hence $\langle\alpha,\beta\rangle$ is flat and $U := \mathbb{Q}_{\fp}^n$ is tidy for it.
\begin{enumerate}
\item Determine the subgroups $\Upm{U}{\alpha}$ and $\Upm{U}{\beta}$. Are $\Uz{U}{\alpha}$ and $\Uz{U}{\beta}$ trivial?
\item Determine the subgroups $\con(\beta)$ and $\con(\beta^{-1})$. (Note that $\langle\alpha,\beta\rangle$ has trivial nub and so these contraction subgroups are also the contraction subgroups modulo $\nub(\langle\alpha,\beta\rangle)$.)
\item Verify that $\Um{U}{\alpha} = (\Um{U}{\alpha}\cap \Up{U}{\beta})(\Um{U}{\alpha}\cap\Um{U}{\beta})$.
\item Find $\gamma\in\langle\alpha,\beta\rangle$ such that 
\begin{align*}
U &= \con(\gamma/\nub(\langle\alpha,\beta\rangle))\con(\gamma^{-1}/\nub(\langle\alpha,\beta\rangle))\\
\mbox{ and }&\con(\gamma/\nub(\langle\alpha,\beta\rangle))=\con(\beta/\nub(\langle\alpha,\beta\rangle)).
\end{align*} 
\end{enumerate}

\item Let $\langle\alpha,\beta\rangle$ be the flat group in Exercise~\ref{exer:vector_examp}.~and let 
$$
V = \left\{ (\xi_i)\in\mathbb{Z}_{\fp}^n \mid \xi_i \equiv \xi_3 \mod {\fp} \mbox{ for all }i\geq3\right\}.
$$
\begin{enumerate}
\item Show that $V$ is tidy for $\alpha$ and $\beta$.
\item Is $V$ tidy for $\langle\alpha,\beta\rangle$ when $n>3$?
\end{enumerate} 

\item Let $\langle\alpha,\beta\rangle$ be the flat group in Exercise~\ref{exer:vector_examp}. Show that the subgroup $\Um{U}{\beta}$ is not $\fl{H}$-invariant.

\item 
\label{exer:rank_weights}
Let $\langle\alpha,\beta\rangle$ be the flat group in Exercise~\ref{exer:vector_examp}. What is the rank of $\langle\alpha,\beta\rangle$ and what is the cardinality of $\red(\langle\alpha,\beta\rangle)$?

\item 
\label{exer:torsion_quotient}
Let $\alpha,\beta\in\Aut(\mathbb{Q}_\fp^2)$ be given by
$$
\alpha(\xi_1,\xi_2) = (\fp\xi_1,\fp\xi_2)\mbox{ and }\beta(\xi_1,\xi_2) = (\fp\xi_1,\fp^{-1}\xi_2),\quad (\xi_1\xi_2)\in\mathbb{Q}_\fp^2.
$$
Then $\langle\alpha,\beta\rangle$ is a subgroup of the flat group from Exercise~\ref{exer:Qpn}. Hence $\langle\alpha,\beta\rangle$ is flat and $U := \mathbb{Q}_{\fp}^2$ is tidy for it.
\begin{enumerate}
\item Determine the subgroups $\Upm{U}{\alpha}$ and $\Upm{U}{\beta}$. 
\item Determine the subgroups $\con(\beta)$ and $\con(\beta^{-1})$.
\item Describe the homomorphism $L : \langle\alpha,\beta\rangle\to \mathbb{Z}^2$ defined in Proposition~\ref{prop:free_abelian}. 
\item Is $L$ surjective? What is $\mathbb{Z}^2/L(\langle\alpha,\beta\rangle)$?
\end{enumerate}

\item \label{exer:zero_necessary}
Let $G = (F_p(\!(t)\!),+)^2$ and define $\alpha,\beta\in\Aut(G)$ by
$$
\alpha(f_1,f_2) = (tf_1,f_2) \text{ and }\beta(f_1,f_2) = (f_1,tf_2), \quad (f_1,f_2)\in F_p(\!(t)\!)^2.
$$
\begin{enumerate}
	\item Show that $F_p[\![t]\!]^2$ is tidy for $\langle\alpha,\beta\rangle$ and hence that $\langle\alpha,\beta\rangle$ is flat.
	\item Let $U = F_p[\![t]\!]^2$ and determine $\Umm{U}{\alpha}$. Verify that $\beta(\Umm{U}{\alpha})\not\leq \Umm{U}{\alpha}$ but that $\beta(\Umm{U}{\alpha})\leq \Umm{U}{\alpha}\beta(\Uz{U}{\alpha})$.
\end{enumerate}

\item\label{exer:scaling}
Let $G = F\times \mathbb{Q}_p$ with $F$ a non-trivial finite group, and let $U = \triv\times\mathbb{Z}_p$ be a compact open subgroup of $G$. Define an automorphism, $\alpha$, of $G$ by
$$
\alpha(f,\xi) = (f,{p}\xi),\qquad f\in F,\ \xi\in\mathbb{Q}_p.
$$
\begin{enumerate}
	\item Show that $\alpha^n(U)\leq U$ or $\alpha^n(U)\geq U$ for every $\alpha^n\in\langle\alpha\rangle$, but that $\bigcup\left\{\alpha^n(U)\right\} \ne G$.
	\item Show that $G$ is nevertheless scaling for $\langle\alpha\rangle$.
\end{enumerate}

\item Let $\fl{H}$ be the flat group of automorphisms of $SL(2,\mathbb{Q}_{\fp})$ described in Question~\ref{exer:matrix_flat} and let $\alpha = \mbox{diag}[1,{\fp},{\fp}^2]\in \fl{H}$. Let $U$ be the subgroup tidy for $\fl{H}$ seen in Question~\ref{exer:matrix_flat}.
\begin{enumerate}
\item Find the groups $\Upm{U}{\alpha}$, $\Utpm{U}{\alpha}$ and $\con(\alpha/\nub(\fl{H}))$. (Recall that $\nub(\fl{H})$ was found in Question~\ref{exer:Qpn}.)
\item Observe that $\con(\alpha/\nub(\fl{H}))$ is invariant under $\fl{H}$ and denote the restriction of $\fl{H}$ to $\con(\alpha/\nub(\fl{H}))$ by $\widetilde{\fl{H}}$. \\
Verify that $\widetilde{U} := U\cap \con(\alpha/\nub(\fl{H}))$ is tidy for $\widetilde{\fl{H}}$.
\item Let $\beta = \mbox{diag}[1,1,p]$ and denote the restriction of $\beta$ to $\con(\alpha/\nub(\fl{H}))$ by $\widetilde{\beta}$. Determine the groups $\Upm{\widetilde{U}}{\beta}$, $\Utpm{\widetilde{U}}{\beta}$ and $\con(\widetilde{\beta}/\nub(\fl{H}))$.
\item Find $\gamma\in\langle\alpha,\beta\rangle$ such that $\widetilde{U} = \con(\widetilde{\gamma}^{-1}/\nub(\fl{H}))\con(\widetilde{\gamma}/\nub(\fl{H}))
$ and $\con(\widetilde{\beta}/\nub(\fl{H})) = \con(\widetilde{\gamma}/\nub(\fl{H}))$.
\end{enumerate}

\item \label{exer:choose_beta}
Let $G = \mathbb{Q}_p^3$ and $\alpha,\beta\in \Aut(G)$ be given by
$$
\alpha(a,b,c) = (pa,pb,pc) \hbox{ and }\beta(a,b,c) = (pa,b,p^{-1}c),\quad (a,b,c)\in\mathbb{Q}_p^3.
$$
\begin{enumerate}
	\item Show that $\fl{H} := \langle\alpha,\beta\rangle$ is flat.
	\item Choose $U$ and $N$ to satisfy the hypotheses of Proposition~\ref{prop:choose_beta} and determine $\Upp{U}{\beta}$ for this choice.
	\item Find an automorphism $\gamma$ as in the proof of Proposition~\ref{prop:choose_beta}.
	\item Show that $\Upp{U}{\beta} \ne \con(\gamma^{-1}/N)$.
\end{enumerate}


\end{enumerate}

\newpage

\bibliography{TDG}
\bibliographystyle{plain}
\end{document}